\pdfoutput=1
\RequirePackage{ifpdf}
\ifpdf 
\documentclass[pdftex]{sigma}
\else
\documentclass{sigma}
\fi

\usepackage{mathrsfs}

\numberwithin{equation}{section}

\newtheorem{thm}{Theorem}[section]
\newtheorem{cor}[thm]{Corollary}
\newtheorem{lem}[thm]{Lemma}
\newtheorem{prop}[thm]{Proposition}
 { \theoremstyle{definition}
\newtheorem{df}[thm]{Definition}
\newtheorem{exa}[thm]{Example}
\newtheorem{rmk}[thm]{Remark} }

\DeclareMathOperator{\Ann}{Ann}
\DeclareMathOperator{\rk}{rank}

\newcommand{\exd}{\mathrm{d}}

\newcommand{\D}{\mathrm{D}}

\newcommand{\w}{{\,{\wedge}\;}}

\newcommand{\Leg}{\textrm{Leg}}

\newcommand{\Ker}{\operatorname{Ker}}
\newcommand{\vl}{\, \vline\, }

\DeclareFontFamily{U}{MnSymbolC}{}
\DeclareSymbolFont{MnSyC}{U}{MnSymbolC}{m}{n}
\DeclareFontShape{U}{MnSymbolC}{m}{n}{
 <-6> MnSymbolC5
 <6-7> MnSymbolC6
 <7-8> MnSymbolC7
 <8-9> MnSymbolC8
 <9-10> MnSymbolC9
 <10-12> MnSymbolC10
 <12-> MnSymbolC12}{}
\DeclareMathSymbol{\im}{\mathbin}{MnSyC}{'270}

\newcommand{\n}{{n}}

\newcommand{\ve}{\varepsilon}

\newcommand{\half}{\textstyle{\frac{1}{2}}}
\newcommand{\quar}{\textstyle{\frac{1}{4}}}
\newcommand{\third}{\textstyle{\frac{1}{3}}}
\newcommand{\nfrac}{\textstyle{\frac{1}{n-1}}}
\newcommand{\sqr}{{\textstyle{\frac{1}{\sqrt{2}}}}}

\newcommand{\fomegaone}{\textstyle{\frac{\partial}{\partial\omega^1}}}

\newcommand{\fomegaz}{\textstyle{\frac{\partial}{\partial\omega^0}}}
\newcommand{\fomegann}{\textstyle{\frac{\partial}{\partial\omega^{n}}}}

\newcommand{\Itot}{\mathcal{I}_{\mathsf{tot}}}
\newcommand{\Ibas}{\mathcal{I}_{\mathsf{bas}}}
\newcommand{\Ichar}{\mathcal{I}_{\mathsf{char}}}

\newcommand{\ts}{\textstyle}

\newcommand{\ra}{\rightarrow}

\newcommand{\lra}{\longrightarrow}

\newcommand{\lmt}{\longmapsto}

\newcommand{\bc}{{\underline{c}}}
\newcommand{\bd}{{\underline{d}}}
\newcommand{\una}{{\underline{a}}}
\newcommand{\unb}{{\underline{b}}}
\newcommand{\unc}{{\underline{c}}}
\newcommand{\und}{{\underline{d}}}

\newcommand{\uni}{{\underline{i}}}

\newcommand{\wh}{\widehat}

\newcommand{\wsf}{\textsc{Wsf }}
\newcommand{\ssf}{\textsc{ssf }}

\newcommand{\orth}{\mathfrak{o}}

\newcommand{\fo}{\mathfrak{o}}
\newcommand{\fg}{\mathfrak{g}}
\newcommand{\fp}{\mathfrak{p}}

\newcommand{\fm}{\mathfrak{m}}
\newcommand{\fh}{\mathfrak{h}}

\newcommand{\sU}{\mathscr{U}}

\newcommand{\Orth}{\mathrm{O}}

\newcommand{\sN}{\mathsf{N}}

\newcommand{\bfA}{\mathbf{A}}
\newcommand{\bfB}{\mathbf{B}}
\newcommand{\bfC}{\mathbf{C}}

\newcommand{\bfE}{\mathbf{E}}
\newcommand{\bfF}{\mathbf{F}}

\newcommand{\bfM}{\mathbf{M}}

\newcommand{\bfR}{\mathbf{R}}
\newcommand{\bfS}{\mathbf{S}}

\newcommand{\bfa}{\mathbf{a}}
\newcommand{\bfb}{\mathbf{b}}
\newcommand{\bfc}{\mathbf{c}}
\newcommand{\bfd}{\mathbf{d}}
\newcommand{\bfe}{\mathbf{e}}
\newcommand{\bff}{\mathbf{f}}

\newcommand{\bfi}{\mathbf{i}}
\newcommand{\bfj}{\mathbf{j}}

\newcommand{\bfu}{\mathbf{u}}
\newcommand{\bfv}{\mathbf{v}}

\newcommand{\RR}{\mathbb{R}}

\newcommand{\FF}{\mathbb{F}}

\newcommand{\ZZ}{\mathbb{Z}}
\newcommand{\NN}{\mathbb{N}}
\newcommand{\PP}{\mathbb{P}}

\newcommand{\II}{\mathbb{I}}
\newcommand{\VV}{\mathbb{V}}

\newcommand{\cA}{\mathcal{A}}

\newcommand{\cC}{\mathcal{C}}
\newcommand{\cD}{\mathcal{D}}

\newcommand{\cF}{\mathcal{F}}

\newcommand{\cI}{\mathcal{I}}

\newcommand{\cL}{\mathcal{L}}

\newcommand{\cN}{\mathcal{N}}
\newcommand{\cE}{\mathcal{E}}
\newcommand{\cP}{\mathcal{P}}
\newcommand{\cS}{\mathcal{S}}

\newcommand{\cV}{\mathcal{V}}

\DeclareMathOperator{\GL}{GL}

\usepackage[small,nohug]{diagrams}
\diagramstyle[labelstyle=\scriptstyle]

\begin{document}
\allowdisplaybreaks

\newcommand{\arXivNumber}{1704.02542}

\renewcommand{\PaperNumber}{080}

\FirstPageHeading

\ShortArticleName{Differential Geometric Aspects of Causal Structures}

\ArticleName{Differential Geometric Aspects of Causal Structures}

\Author{Omid MAKHMALI}

\AuthorNameForHeading{O.~Makhmali}

\Address{Institute of Mathematics, Polish Academy of Sciences,\\ 8~\'Sniadeckich Str., 00-656 Warszawa, Poland}
\Email{\href{mailto:omakhmali@impan.pl}{omakhmali@impan.pl}}

\ArticleDates{Received April 25, 2017, in final form July 23, 2018; Published online August 02, 2018}

\Abstract{This article is concerned with causal structures, which are defined as a field of tangentially non-degenerate projective hypersurfaces in the projectivized tangent bundle of a manifold. The local equivalence problem of causal structures on manifolds of dimension at least four is solved using Cartan's method of equivalence, leading to an $\{e\}$-structure over some principal bundle. It is shown that these structures correspond to parabolic geomet\-ries of type $(D_n,P_{1,2})$ and $(B_{n-1},P_{1,2})$, when $n\geq 4$, and $(D_3,P_{1,2,3})$. The essential local invariants are determined and interpreted geometrically. Several special classes of causal structures are considered including those that are a lift of pseudo-conformal structures and those referred to as causal structures with vanishing \textsc{Wsf} curvature. A twistorial construction for causal structures with vanishing \textsc{Wsf} curvature is given. }


\Classification{53A55; 58A15; 58A30}

\section{Introduction}\label{sec:introduction}

\subsection{Motivation and history}\label{sec:motivation-history}

In the general theory of relativity, space-time is represented by a four dimensional manifold~$M$ equipped with a Lorentzian metric $g$. An important feature of a Lorentzian structure is its associated field of null cones in $TM$ given by the set of vectors $v\in TM$ satisfying $g(v,v)=0$. As a result, it is possible to define a relation of \emph{causality} between the points of $M$, i.e., a partial ordering $x\prec y$ defined by the property that $x$ can be joined to $y$ by a curve that is either time-like or null. If $x$ and $y$ represent two events in space-time then $x\prec y$ means that the occurrence of~$x$ has an effect on~$y$. In \cite{KP-causal} Kronheimer and Penrose examined causality in a manifold~$M$ of dimension~$n$ on an axiomatic basis. The starting point of their study is a continuous assignment of null cones in the tangent space of each point of~$M$ so that it is possible to define the set of points that are inside, outside and on the null cone in each tangent space.\footnote{It should be said that in \cite{KP-causal} it is assumed that each null cone is given by the vanishing locus of a quadratic form of signature $(n,1)$, i.e., conformal Lorentzian geometry. However, the results extend to the more general case of causal structures with strictly convex null cones.} The importance of the causal relationship between points of space-time became more significant in the context of the singularity theorems of Penrose and Hawking \cite{HE-SpaceTime, Penrose-Sing, Penrose-Technique}. The study of causal relations in spaces equipped with a field of convex null cones in $TM$ received a major contribution from the Alexandrov school in Russia under the name of \emph{chronogeometry} \cite{Alexandrov-chronogeometry, Guts-Causal}, i.e., a geometry determined by the chronological relations between the points of the manifold. The objective was to use the principles of causality and symmetry to give an axiomatic treatment of special relativity. In \cite{Segal-Book}, Segal attempted to give such a treatment for general relativity (see also \cite{DS-Segal, Levichev-SegalChronometric,Taub-Review, Wormald-Critique, EPS-GeometryOfLight}.) Subsequently, structures defined by assigning a convex cone in each tangent space, termed as causal structures, came to play a role in the Lie theory of semi-groups \cite{HO-causal}, causal boundaries \cite{GS-causal}, and hyperbolic systems of partial differential equations \cite{Schapira-Causal}.

Rather than to investigate causality relations between the points of a manifold with a given causal structure, the main idea of this article is to gain a deeper understanding of the geometry of causal structures through the solution of the Cartan equivalence problem.
This approach enables one to construct a complete set of local invariants and also to obtain some results on symmetries of causal structures.

Before going any further, a more precise definition of causal structures is in order. Given an $(n+1)$-dimensional manifold $M$, with $n\geq 3$, a \textit{causal structure} $(M,\cC)$ of signature $(p+1,q+1)$ with $p,q\geq 0$, and $p+q=n-1$ is given by a sub-bundle of the projectivized tangent bundle $\pi\colon \cC\subset \PP TM\ra M$, whose fibers $\cC_x\subset \PP T_xM$ are projective hypersurfaces\footnote{Here, a hypersurface is always taken to be a codimension one submanifold.} with projective second fundamental form of signature $(p,q)$ (see Definition~\ref{def:causal-str-def}). Note that in this definition the convexity of the null cones $\cC_x$ is not assumed.

Two causal structures $(M,\cC)$ and $(M',\cC')$ are locally equivalent around the points $p\in M$ and $p'\in M'$ if there exists a diffeomorphism $\psi\colon U\lra U', $ where $p\in U\subset M$, $p'=\psi(p)\in U'\subset M'$, satisfying $\psi_{*}(\cC_x)=\cC_{\psi(x)}$ for all $x\in U$.

In order to describe a causal structure in an open set $U\subset TM$, one can use a defining function $L\colon TM\ra\RR$, $(x^i;y^j)\mapsto L(x^i;y^j)$, where $(x^0,\dots,x^n)$ are coordinates for $M$ and $(y^0,\dots,y^n)$ are fiber coordinates. The function $L$ is assumed to be homogeneous of some degree $r_1$ in the fiber variables and is referred to as a Lagrangian for the causal structure. Then, the vanishing set of~$L$ over~$U$ coincides with $U\cap \widehat\cC$ where $\widehat\cC\subset TM$ is the cone over $\cC\subset\PP TM. $ The func\-tions~$L(x^i;y^i)$ and $S(x^i;y^i)L(x^i;y^i)$, where $S(x^i;y^i)$ is nowhere vanishing and homogeneous of degree~$r_2$ in~$y^i$'s, define the same causal structure.

As an example, if the fibers $\cC_x$ are hyperquadrics defined by a pseudo-Riemannian metric $g$ of signature $(p+1,q+1)$, then the causal structure $(M,\cC)$ given by the family of null cones of $g$ corresponds to the pseudo-conformal structure induced by the metric $g$.\footnote{A pseudo-conformal structure refers to the conformal class of a pseudo-Riemannian metric \cite{AG-Conformal}.} In other words, there exists a local defining function of this causal structure of the form $F(x^i;y^j)=g_{ij}(x)y^iy^j$.

The above description makes it clear that the relation of causal structures to pseudo-conformal structures is an analogue of what Finsler structures are to Riemannian structures. Recall that a Finsler structure $\big(M^{n+1},\Sigma^{2n+1}\big)$ is given by a codimension one sub-bundle $\Sigma\subset TM$, whose fibers $\Sigma_x\subset T_xM$ are strictly convex affine hypersurfaces, i.e., the second fundamental form of each fibers is positive definite. The sub-bundle $\Sigma$ is called the indicatrix bundle (see \cite{BryantRemarksFinsler}).\footnote{Similarly, it is possible to define pseudo-Finsler structures of signature $(p,q)$, $p+q=n+1$ by imposing the condition that the second fundamental form of the fibers $\Sigma_x$ have signature $(p-1,q)$.}

The main motivation behind the study of causal structures comes from a variety of perspectives. As a part of the geometric study of differential equations, it was shown by Holland and Sparling~\cite{HS2010causal} that there is locally a one-to-one correspondence between contact equivalence classes of third order ODEs $y'''=F(x,y,y',y'')$ and causal structures on their locally defined three dimensional solution space $M$.
A causal structure is given in this setting by a family of curves $\pi\colon \cC\ra M$ where $\cC_x:=\pi^{-1}(x)\subset \PP T_xM$ is a~\emph{non-degenerate} curve in the sense that it admits a well-defined projective Frenet frame. This causal structure descends to a conformal Lorentzian structure on~$M$ if a certain contact invariant of the third order ODE known as the W\"unschmann invariant vanishes. In other words, it is shown in~\cite{HS2010causal} that $\cC$ is locally isomorphic to $J^2(\RR,\RR)$\footnote{In this article $J^k\big(\RR^{l},\RR^m\big)$ denotes the space of $k$-jets of functions from $\RR^l$ to $\RR^m$.} endowed with a foliation by lifts of contact curves on $J^1(\RR,\RR)$ and that the W\"unschmann invariant at each point of $J^2(\RR,\RR)$ coincides with the projective curvature of the curve~$\cC_x$. As a result, the vanishing of the W\"unschmann invariant implies that the curves~$\cC_x$ are conics, i.e., the causal structure is conformal Lorentzian.

The work of Holland and Sparling is a part of the program of studying geometries arising from differential equations, which has a long history going back to Monge, Jacobi, Lie, Darboux, Goursat, Cartan and others. In the case of third order ODEs under contact transformations, Chern \cite{ChernODE} used Cartan's method to solve the equivalence problem and showed that the solution depends on two essential relative invariants, namely the W\"unschmann invariant $I(F)$ and the invariant $C(F)=\frac{\partial^4}{\partial q^4}F(x,y,p,q)$. Furthermore, he observed that if $I(F)=0$, then the space of solutions of the ODE can be endowed with a conformal Lorentzian structure. Using Tanaka theory~\cite{Tanaka-79}, Sato and Yashikawa \cite{SY-ODE} showed how one can use Chern's $\{e\}$-structure to associate a normal Cartan connection on some 10-dimensional principal bundle to the contact equivalence class of a third order ODE. In particular, they expressed the relative contact invariants of the ODE in terms of the curvature of the Cartan connection. Godlinski and Nurowski~\cite{GNThirdODEs} explicitly computed the normal Cartan connection and explored other geometric structures that arise from a third order ODE. In particular, they used a result of Fox~\cite{FoxContactProj} to show if $C(F)=0$, then the ODE induces a \emph{contact projective structure}\footnote{A contact path geometry over a contact manifold $N$ is given by a family of contact paths with the property that at each point $x\in N$, a unique curve of the family passes through each contact direction in $T_xM$. A contact projective geometry is a contact path geometry whose contact paths are geodesics of some affine connection.} on~$J^1(\RR,\RR)$.

One can realize the works mentioned above on third order ODEs in terms of geometries arising from the double fibration
\begin{gather*} 
J^1(\mathbb{R},\mathbb{R}) \overset{\tau}{\longleftarrow} J^2(\mathbb{R},\mathbb{R}) \overset{\pi}{\longrightarrow} M,
\end{gather*}
where the map $\pi$ is the quotient map of $J^2(\RR,\RR)$ by the foliation defined by the 2-jets of the solutions of the ODE. This map endows $M$ with a causal structure. The null geodesics of $M$ are given by $\pi\circ\tau^{-1}(p)$ where $p\in J^1(\RR,\RR)$, while the contact geodesics are given by $\tau\circ\pi^{-1}(q)$ where $q\in M$.

Equivalently, given a causal structure $(M,\cC)$ in dimension three, Holland and Sparling showed that the double fibration above is equivalent to
\begin{gather*}
\cN \overset{\tau}{\longleftarrow} \cC \overset{\pi}{\longrightarrow} M,
\end{gather*}
where $\cN$ is the space of \emph{characteristic curves} (or \emph{null geodesics})\footnote{In this article the term characteristic curves is used to refer to the lift of null geodesics to the null cone bundle.} of the causal structure. The precise definition of these curves is given in Section \ref{sec:char-curv}. It turns out that starting with a causal structure $(M,\cC)$, one can show that locally $\cC\cong J^2(\RR,\RR)$ and $\cN\cong J^1(\RR,\RR)$. Therefore, all the results mentioned above can be translated in terms of local invariants of the causal structure in three dimensions.

The correspondence above motivates the study of the local invariants of causal geometries in higher dimensions and differential equations that can be associated to them. As will be explained, this study turns out to be related to a number of current themes of research.

In \cite{HS2011Causal}, Holland and Sparling studied causal structures in higher dimensions and showed that an analogue of the Weyl sectional curvature can be defined. Also, they showed for causal structures the Raychaudhuri--Sachs equations from general relativity remains valid for higher dimensional causal structures (also see \cite{AS-Singularity, Minguzzi-Raychaudhuri}).

Before describing the nature of the complete set of invariants for causal structures it is helpful to make parallels with what happens in Finsler geometry. On an $(n+1)$-dimensional mani\-fold~$M$ the indicatrix bundle~$\Sigma$ of a Finsler structure is $(2n+1)$-dimensional equipped with a contact 1-form, i.e., a 1-form $\omega\in T^*\Sigma$ satisfying $\omega\w (\exd\omega)^n\neq 0$. The 1-form $\omega$ is called the Hilbert form of the Finsler structure. As indicated by Cartan \cite{CartanProbleme30, CartanFinsler} and Chern~\cite{ChernFinsler}, in Finsler geometry there are two essential invariants comprising the total obstruction to local flatness, namely the \textit{flag curvature} and the \textit{centro-affine cubic form} of the fibers~$\cC_x$ (also referred to as the Cartan torsion). The vanishing of the Cartan torsion implies that the Finsler structure is a Riemannian structure, in which case the flag curvature coincides with the sectional curvature.

Regarding causal structures, the main step is to realize that the $2n$-dimensional space $\cC$ carries a quasi-contact structure (also known as even contact structure), i.e., a 1-form $\omega^0$ satisfying $\omega^0\w(\exd\omega^0)^{n-1}\neq 0 $ and $ (\exd\omega^0)^n=0$.
The 1-form $\omega^0$ is referred to as the \emph{projective Hilbert form}.\footnote{This terminology is introduced in analogy with the Hilbert form in Finsler geometry and the calculus of variations~\cite{BCS-Finsler}.} It is shown that, for $n\geq 3$, there are two essential local invariants which form the total obstruction to local flatness. As one would expect, the \textit{Fubini cubic form} of each fiber~$\cC_x$ as a projective hypersurfaces in~$\PP T_xM$ constitutes one of these obstructions whose vanishing implies that the causal structure is a pseudo-conformal structure. The second essential invariant is a generalization of the sectional Weyl curvature of a pseudo-conformal structure restricted to the \emph{shadow space}\footnote{The terminology is due to a physical interpretation made by Sachs~\cite{Sachs-Gravitational} according to which if an object follows a congruence of null geodesics, its infinitesimal shadow casts onto the pull-back of this space. This term was initially used by Holland and Sparling in~\cite{HS2011Causal}.} as found by Holland and Sparling in~\cite{HS2011Causal} (see Section \ref{sec:essential-invariants}). By analogy with Finsler geometry, the second essential invariant is called the \textit{Weyl shadow flag curvature} (referred to as \wsf curvature.)

Furthermore, from the first order structure equations for causal geometries it follows that a~causal structure $(M,\cC)$ can be equivalently formulated as having a manifold $\cC^{2n}$ equipped with a distribution $D$ of rank $n$ which contains an involutive corank one subdistribution $\Delta\subset D$ and satisfies a certain non-degeneracy condition. The (small) growth vector of $D$ is $(n,2n-1,2n)$, i.e., $\rk [D,D]=2n-1$, and $\rk [D,[D,D]]=2n$\footnote{The sections of the distribution $[D,D]$ are given by $\Gamma([D,D])=[\Gamma(D),\Gamma(D)]$.} (see Section~\ref{sec:char-caus-struct}). In this case, $M^{n+1}$ is the space of integral leaves of $\Delta$.\footnote{In~\cite{ANNMonge}, these structures are referred to as the $B_n$ and $D_n$ classes of what they call parabolic geometries of Monge type.} This description identifies causal geometries as parabolic geometries modeled on $B_{n-1}/P_{1,2}$ and $D_n/P_{1,2}$ for $n\geq 4$ and $D_3/P_{1,2,3}$, with their natural $|3|$-grading. As will be shown, they have some similar features to path geometries~\cite{GrossmanPath} and contact path geometries~\cite{FoxPath}, which correspond respectively to parabolic geometries modeled on $A_n/P_{1,2}$ and $C_n/P_{1,2}$.

Another essential part of any causal structure is given by its set of \emph{characteristic curves.} Note that as a quasi-contact manifold, $\cC$ is equipped with the characteristic line field of the projective Hilbert form, i.e., the unique degenerate direction of $\exd\omega^0$ lying in $\Ker\,\omega^0$, defined by vector fields $\bfv$ satisfying
\begin{gather*} 
 \omega^0(\bfv)=0,\qquad\bfv\im\exd\omega^0=0.
\end{gather*}
The integral curves of a characteristic vector field $\bfv$ is called a characteristic curve. In the case of pseudo-conformal structures the projection of these curves to $M$ coincide with null geodesics.

The description above shows that, unlike Finsler geometry, the definition of characteristic curves does not depend on a choice of metric. Recall that in Finsler geometry geodesics can be defined as the extremals of the arc-length functional
\begin{gather*}\Phi(\gamma)=\int_I \gamma^*\phi,\end{gather*}
where \looseness=-1 $\gamma\colon (a,b)=:I\ra \cC$ is a curve in $\cC$ and $\gamma^*\phi=L(x^i;y^i)\exd t$. The function $L(x^i;y^i)$ is called the Finsler metric which in the calculus of variations is referred to as the Lagrangian. The Finsler metric is homogeneous in $y^i$'s and has non-degenerate vertical Hessian\footnote{It is usually assumed that the Finsler metric when restricted to $T_xM$ only vanishes at the origin and is positive elsewhere \cite{BCS-Finsler}. However, as was mentioned previously, the condition can be relaxed~\cite{Shen-Book}.}. When addressing the geodesics of a Finsler structure, or more generally extremal manifolds in the calculus of variations, it is generally assumed that the Lagrangian does not vanish along them~\cite{Griffiths-EDS-Book}. The study of extremal curves along which the Lagrangian vanishes seems to have received much less attention in the literature (see~\cite{RundNullExt, RundNullDist}). The exception is the case of Lagrangians that are quadratic in $y^i$'s. This is due to the fact that they arise in the context of pseudo-conformal structures. For instance, this problem is briefly discussed by Guillemin \cite{GuilleminCosmology}, where he studies the deformations of the product metric $\pi^*_1\exd t^2\oplus -\pi_2^* g$ on $S^1\times S^2$ where $\pi_1$ and~$\pi_2$ are projections onto~$S^1$, and $S^2$ and~$g$ is the round metric on~$S^2$, which is \textit{Zollfrei}, i.e., all of its null geodesics are closed with the same period. In Section~\ref{sec:jacobi-fields} it is discussed how causal structures provide a geometric setup for the variational problems involving certain classes of Lagrangians $L(x^i;y^i)$ which are homogeneous in $y^i$'s and have vertical Hessian of maximal rank over their vanishing set.

The foliation of $\cC$ defined by the characteristic curves allows one to define the quotient map $\tau\colon \cC\ra\cN$ where~$\cN$ is $(2n-1)$-dimensional and is called the space of characteristic curves of~$\cC$. A priori, $\cN$ is only locally defined around generic points and can be used to give a geometric description of the theory of correspondence spaces for causal structure in the sense of \v{C}ap~\cite{CapCorr} (see Section~\ref{sec:semi-integrable-lie}). Note that the space $\cN$, when globally defined, is of interest in Penrose's twistor theory~\cite{Lebrun-conformal-gravity} and a twistorial desciption of causal relations among the points of the space-time~\cite{Low-Causal}.

Finally, an interesting instance of studying causal structure can be found in \cite{HwangVMRTCodim1} where Hwang considers complex causal structures arising on certain uniruled projective manifolds, $M^{n+1}$, satisfying the property that through a generic point $x\in M$ there passes an $(n-1)$-parameter family of \textit{rational curves of minimal degree.} The hypersurface $\cC_x\subset \PP T_xM$ obtained from the tangent directions to such curves at the point $x$ is called the variety of minimal rational tangent (or VMRT). Under the assumption that~$\cC_x$ is smooth of degree $\geq 3$ for general $x\in M$, he shows that the causal structure defined by them is locally isotrivially flat (see Definition~\ref{def:isotrivially-flat})\footnote{Hwang refers to such structures as flat cone structures of codimension one.}. Furthermore, it turns out that the characteristic curves for such causal structures coincide with rational curves of minimal degree\footnote{Motivated by the notion of \emph{conformal torsion} introduced by LeBrun \cite{LebrunNullGeod} and \emph{contact torsion} introduced by Fox \cite{FoxPath}, one can define \emph{causal torsion} whose vanishing implies that the rational curves of minimal degree coincide with the characteristic curves of the causal structure they define. The precise definition of causal torsion is not given in this article. }. Finally, assuming that $M$ has Picard number 1 and $n\geq 2$, he shows that $M$ is biregular to a hyperquadric equipped with its natural causal structure (see Section~\ref{sec:flat-model-2}). This theorem is a generalization of what has been considered previously in \cite{BelgunNullgeod,Ye} for complex conformal structures.

More broadly, Hwang's study of causal structures is a part of the differential geometric characterization of uniruled varieties initiated by Hwang and Mok \cite{HwangSurveyVMRT} via geometric structures referred to as \emph{cone structures} (see Definition \ref{def:defin-notat-conv}). The starting point for Hwang and Mok is a~given family of rational curves of minimal degree whose tangent directions at a generic point $x\in M$ give rise to the VMRT $\cC_x\subset \PP T_xM$. For instance, if $\cC_x=\PP T_xM$ one obtains the complex analogue of path geometry \cite{GrossmanPath} and if $\cC_x$ is a hyperplane which induces a contact distribution on~$M$, one obtains a complex contact path geometry~\cite{FoxPath}. The program of Hwang and Mok seems interesting and calls for further exploration from the perspective of the method of equivalence, especially through the lens of flag structures as developed in~\cite{DZquasi}.

\subsection{Main results}\label{sec:main-results}

In all the statements of this article smoothness of underlying manifolds and maps is always assumed and the dimension of $M$ is always taken to be at least four unless otherwise specified.

In Section \ref{sec:con e-struct-proj}, after reviewing the flat model for causal structures, the first order structure equations for a causal geometry $(M^{n+1},\cC)$ of signature $(p+1,q+1)$ are derived. It is shown that the bundle of null cones $\cC$ is foliated by characteristic curves and is endowed with the conformal class of a quadratic form which is degenerate along the vertical directions of the fibration $\cC\ra M$. Subsequently, the first order structure equations are used in Appendix~\ref{cha:append-full-struct} to obtain structure equations leading to an $\{e\}$-structure.

 The essential theorem of this article that is stated in Section \ref{sec:an-e-structure-1} is the following.
\newtheorem*{thm:e-structure}{\bf Theorem \ref{thm:e-structure}}
\begin{thm:e-structure}
To a causal structure $\big(M^{n+1},\cC^{2n}\big)$, $n\geq 4$, one can associate an $\{e\}$-structure on the principal bundle $P_{1,2}\ra\cP\ra\cC$ of dimension $\frac{(n+3)(n+2)}{2}$ where \smash{$P_{1,2}\subset \mathrm{O}(p+2,q+2)$} is the parabolic subgroup defined in \eqref{eq:P_12}. If $n=3$, then $P_{1,2}$ is replaced by $P_{1,2,3}$. The essential invariants of a causal structure are the Fubini cubic form of the fibers and the \wsf curvature. The vanishing of the essential invariants implies that the $\{e\}$-structure coincides with the Maurer--Cartan forms of $\mathfrak{o}(p+2,q+2)$.
\end{thm:e-structure}
The proof of the above theorem takes up most of Appendix~\ref{cha:append-full-struct}.
Using the $\{e\}$-structure thus obtained, it is straightforward to deduce that the essential invariants are the Fubini cubic form and the \wsf curvature. Section \ref{sec:norm-cart-conn} deals with the parabolic nature of causal structures. In Section~\ref{sec:essential-invariants} some geometrical interpretations of the essential invariants are given. The first Section ends with examples of causal structures.

In Section \ref{sec:non-prem-cond}, the structure equations for causal geometries with vanishing \wsf curvature are examined. It is shown that if the Lie derivative of the Fubini cubic form along the vector fields tangent to the characteristic curves is proportional to itself, then the \wsf curvature has to vanish, unless the causal geometry is a pseudo-conformal structure. Such spaces are the causal analogues of Landsberg spaces in Finsler geometry\footnote{See \cite{BCS-Finsler} for an account.}.

Section \ref{sec:semi-integrable-lie} involves the induced structure on the space of characteristic curves, $\cN$, of a causal structure with vanishing \wsf curvature. The $(2n-1)$-dimensional manifold $\cN$ is shown to have a contact structure and is endowed with an $(n+1)$-parameter family of Legendrian submanifolds with the property that through each point $\gamma\in\cN$ there passes a 1-parameter family of them. Moreover, the sub-bundle of $T\cN$ defined by the contact distribution contains a Segr\'e cone of type $(2,n-1)$. The theorem below is proved in which the property of being \emph{$\beta$-integrable} means that at every point $\gamma\in\cN$, the Segr\'e cone is ruled by the tangent spaces of the corresponding 1-parameter family of Legendrian submanifolds passing through $\gamma$.
\newtheorem*{thm:Lie-contact-conf-flat}{\bf Theorem \ref{thm:Lie-contact-conf-flat}}
\begin{thm:Lie-contact-conf-flat}
Given a causal structure with vanishing \wsf curvature $(\cC,M)$, its space of characteristic curves admits a $\beta$-integrable Lie contact structure. Conversely, any $\beta$-integrable Lie contact structure on a manifold $\cN$ induces a causal structure with vanishing \wsf curvature on the space of its corresponding Legendrian submanifolds.
\end{thm:Lie-contact-conf-flat}

It should be noted that the theorem above is reminiscent of Grossman's work \cite{GrossmanPath} on torsion-free path geometries. He showed that given a torsion-free path structure on $N^{n+1}$, the space of paths, $\cS^{2n}$, has a Segr\'e structure, i.e., each tangent space is equipped with a Segr\'e cone of type~$(2,n)$. Furthermore, $\cS$ contains an $(n+1)$-parameter family of $n$-dimensional submanifolds such that through each point $\gamma\in\cS$ there passes a 1-parameter family of them. Grossman showed that at each point $\gamma\in\cS$, the Segr\'e cone is ruled by the tangent spaces of the 1-parameter family of submanifolds passing through $\gamma$. Such Segr\'e structures are called \emph{$\beta$-integrable.}
Furthermore, twistorial constructions such as that of Theorem \ref{thm:Lie-contact-conf-flat} fit exactly into \v{C}ap's theorem of correspondence spaces \cite{CapCorr}. However, \v{C}ap's theory does not contain the geometric interpretation given here in terms of existence of submanifolds that rule the null cone or the Segr\'e cone. Such interpretations are clear from the constructions presented here.

Appendix~\ref{cha:append-full-struct} contains the details of the equivalence problem calculations. Since causal structures are examples of parabolic geometries of type $(B_{n-1},P_{1,2})$ and $(D_n,P_{1,2})$ if $n\geq 4$ and $(D_3,P_{1,2,3})$, they admit a \textit{regular} and \textit{normal} Cartan connection \cite{ANNMonge, CS-Parabolic}. Using the machinery of Tanaka, the correct change of coframe can be given recursively, in order to introduce a Cartan connection on the structure bundle of causal geometries which is normal in the sense of Tanaka. However, for computational reasons, writing down the precise form of the Cartan connection leads to very long expressions which are not particularly enlightening (see \cite{Omid-Thesis}). Furthermore, the first order structure equations suffice for the purposes of this article.

\section{Local description of causal structures} \label{sec:con e-struct-proj}

In this section, after introducing the necessary definitions and notations, the first order structure equations of causal structures are derived. It is assumed that the reader is familiar with Cartan's method of equivalence for $G$-structures (see \cite{Gardner-Book,Olver-Equiv-Book} for an account). The essential torsion terms appearing in the first structure equations are interpreted geometrically. The derivation of the full structure equations is carried out in Appendix~\ref{cha:append-full-struct}. In Section~\ref{sec:essential-invariants} a discussion on the fundamental invariants of a causal structure is given. Finally, the section ends with a few examples of causal structures.

\subsection{Definitions and notational conventions}\label{sec:causal-structures}

Since the main purpose of this article is the local study of causal geometries, the smoothness of manifolds and maps that are considered is always assumed. In other words, the manifolds considered are assumed to be the maximal open sets over which the necessary smoothness conditions are satisfied.

Let $M$ be a smooth $(n+1)$-dimensional (real or complex) manifold, $n\geq 3$ and $\PP TM$ denote its projectivized tangent bundle with projection $\pi \colon \PP TM\ra M$, and fibers $\pi^{-1}(x)=\PP T_xM$. Given a tangent vector $y\in T_xM$, its projective equivalence class is denoted by $[y]\in\PP T_xM$. For $S\subset\PP T_xM$, the cone over $S$ is defined as
\begin{gather*}\widehat S:= \{y\in T_xM \,|\, [y]\in S \}\subset T_xM.\end{gather*}

Throughout this article, the indices range as
\begin{gather*}0\leq i, j, k\leq n,\qquad 1\leq a, b, c\leq n-1,\qquad 0\leq\alpha,\beta\leq n-1,\end{gather*}
and are subjected to the summation convention.
When a set of 1-forms is introduced as $\cI=I\{\omega^i\}_{i=0}^{n}$, it is understood that $\cI$ is the ideal algebraically generated by $\{\omega^0,\dots,\omega^n\}$ in the exterior algebra of $M$.

The symmetrization and anti-symmetrization operations for tensors are denoted using Penrose's abstract index notation. For example, consider a tensor with coefficients $A_{abcd}$. The tensor $A_{(ab|c|d)}$ is symmetric with respect to all indices except the third, and $A_{[a|bc|d]}$ is anti-symmetric in the first and last indices. When dealing with symmetric forms such as $g=2\omega^0\circ\omega^n-\ve_{ab}\omega^a\circ\omega^b$, or $F_{abc}\theta^a\circ\theta^b\circ\theta^c$ with $F_{abc}=F_{(abc)}$, the symbol $\circ$ denotes the symmetric tensor product. Bold-face letters such as $\bfi$, $\bfj$, $\bfa$, $\bfb$ are used to denote a specific index, e.g., the set $\{\omega^\bfa\}$ only contains the 1-form $\omega^\bfa$. The summation convention is not applied to bold-face letters.

At this stage, before stating the definition of a causal structure, a more general geometric structure, namely that of a \emph{cone structure}, is defined~\cite{HwangSurveyVMRT}. The reason for this is that in the process of coframe adaption for cone structures of codimension one, causal structures are distinguished as a sub-class of such geometries via the characterizing property of having \emph{tangentially non-degenerate} fibers (see Section~\ref{sec:tang-non-degen}).
\begin{df}\label{def:defin-notat-conv}
 A \textit{cone structure} of codimension $k\leq n$ on $M^{n+1}$ is given by an immersion $\iota \colon \cC\ra\PP TM$ where $\cC$ is a connected, smooth manifold of dimension $2n+1-k$ with the property that the fibration $\pi\circ\iota \colon \cC\ra M$ is a submersion whose fibers $\cC_x:=(\pi\circ\iota)^{-1}(x)$ are mapped, via the immersion $\iota_x \colon \cC_x\ra \PP T_xM$, to connected projective submanifolds of codimension~$k$ in~$\PP T_xM$.

 Two cone structures $\iota \colon \cC\ra \PP TM$ and $\iota' \colon \cC'\ra\PP TM'$ are locally equivalent around points $x\in M, x'\in M'$, if there exists a diffeomorphism $\phi \colon U\ra U'$ where $x\in U\subset M$ and $x'\in U'\subset M'$, such that $x'=\phi(x)$ and $\phi_*\left(\iota(\cC)\right)=\iota'(\cC'_{\phi(y)})$ for all $y\in M$.
\end{df}

In this article, only cone structures of codimension one are considered i.e., the fibers $\iota(\cC_x)\subset\PP T_xM$ are projective hypersurfaces. Since the map $\iota \colon \cC\ra \PP TM$ is assumed to be an immersion and this article concerns the local nature of geometric structures, it is safe to drop the symbol $\iota$ in $\iota(\cC)$ and $\iota(\cC_x)$ and denote them as $\cC$ and $\cC_x$, respectively. As a result, by restricting to small enough neighborhoods, one can safely assume $\cC\subset\PP TM$ is a sub-bundle defined in an open set $U\subset\PP TM$, where the fibers $\cC_x$ are smooth submanifolds. Consequently, a cone structure of codimension one over a manifold $M$ can be represented by $(\cC,M)$.

In local coordinates, a point of $TM$ is represented by $(x;y)$ where $x=(x^0,\dots,x^n)$ are the base coordinates and $y=(y^0,\dots,y^n)$ are the fiber coordinates. In a sufficiently small neighborhood of a point $p=(x;z)\in\cC$, the homogeneous coordinates $[y^0:\ldots:y^n]$ can be chosen so that $z=[0:\ldots:0:1]$ and the cone $\wh{\cC_x}$ is tangent to the hyperplane $\{y_0=0\}$ along~$\hat z$. The affine coordinates $(y^0,\ldots,y^{n-1})$ obtained from setting $y^n=1$, are used to derive the explicit expressions of the invariants. Using these preferred coordinates in a sufficiently small neighborhood of~$p$, it is assumed that $\cC$ is given in terms of homogeneous coordinates as a locus
 \begin{gather*} 
 L\big(x^0,\ldots,x^n;y^0,\ldots,y^n\big)=0,
 \end{gather*}
where $L$ is homogeneous of degree $r$ in $y^i$ and, in terms of affine coordinates, as a graph
 \begin{gather} \label{eq:ConeStr-DefGraph}
 y^0=F\big(x^0,\ldots,x^n;y^1,\ldots,y^{n-1}\big).
 \end{gather}

\begin{exa}\label{exa:contact-quasi-contact}
 A cone structure of codimension one on $M^{n+1}$ whose fibers are hyperplanes given by $\Ker \omega$ for some $\omega\in T^*M$, and are \emph{maximally non-integrable}, i.e., $\omega\w(\exd\omega)^{k-1}\neq 0$, is called a \textit{contact structure} if $n=2k-1$, and a \textit{quasi-contact structure} (or \textit{even contact structure}) if $n=2k$. These can be considered as ``degenerate'' examples of codimension one cone structures and have no local invariants. In these cases, the function $F$ in~\eqref{eq:ConeStr-DefGraph} is of the form $F(x;y)=c_a(x)y^a+c_n(x)$, with $\omega=\exd x^0+c_a(x)\exd x^a+c_n(x)\exd x^n$.
\end{exa}
\begin{exa}A pseudo-conformal structure, i.e., a conformal class of a pseudo-Riemannian metric $g_{ij}$ of signature $(p+1,q+1),\ p,q\geq 0$, can be introduced by its field of null cones, i.e., the set of quadrics
\begin{gather*}\cC_x:=\big\{ [y]\in\PP T_x \,|\, g_{ij}y^iy^j=0\big\},\end{gather*}
As a result, a pseudo-conformal structure is an example of a cone structure of codimension one, which, in contrast to the previous example, is ``non-degenerate'' (see section \ref{sec:tang-non-degen}). In terms of the local form \eqref{eq:ConeStr-DefGraph}, a representative $g_{ij}$ of the conformal class, at a point $s=(x;y)\in\cC$, can be put in the form $2y^0y^n-\ve_{ab}y^ay^b$, where
\begin{gather*} 
 [\varepsilon_{ab}]=\II_{p,q}:=
\begin{pmatrix}
 \II_p & 0\\
 0 & -\II_q
\end{pmatrix}
,\end{gather*}
 and therefore $F(s)=\varepsilon_{ab}y^ay^b$.
 \end{exa}

In Section \ref{sec:tang-non-degen}, when a coframe $(\omega^0,\ldots,\omega^n,\theta^1,\ldots,\theta^{n-1})$ over $\cC$ is introduced, the corresponding coframe derivatives are denoted by
 \begin{gather*}\exd G=G_{;i}\omega^i+G_{;\una}\theta^a.\end{gather*}

\subsection{The flat model}\label{sec:flat-model-2}
When setting up the equivalence problem for a given geometric structure, finding the flat model can be very helpful for identifying the structure bundle. Naturally, one expects the flat model for causal structures to coincide with the flat model for conformal geometry.
In causal geometry, the flat model is obtained from the natural pseudo-conformal structure on a quadric.
Consider the quadric $Q_{p+1,q+1}^{n+1}\subset\PP^{n+2}$, where $p+q=n-1$, given by ${}^t\!vSv=0$, where $v={}^t\!(y,x^0,x^1,\dots,x^{n},z)\in\RR^{n+3}$ and
\begin{gather*}S:=\begin{pmatrix}
 0 & 0 & 0 & 0 & 1\\
 0 & 0 & 0 & -1 & 0\\
 0 & 0 & \II_{p,q} & 0 & 0\\
 0 & -1 & 0 & 0 & 0\\
 1 & 0 & 0 & 0 & 0\\
\end{pmatrix}
\end{gather*}
has signature $(p+2,q+2)$. It is equipped with a pseudo-conformal structure of signature $(p+1,q+1)$, expressed as $2y^0 y^n-\ve_{ab} y^ay^b$ where $y^i$'s denote the fiber coordinates of $TQ_{p+1,q+1}^{n+1}$. The null cone bundle of this pseudo-conformal class is denoted by $\cC\subset\PP TQ_{p+1,q+1}^{n+1}$. Consider the \emph{flag manifold }
\begin{gather*}
 \NN\FF^{n+1}_{0,1}:=\big\{(V_1,V_2)\in \mathrm{Gr}^0_1(n+3)\times \mathrm{Gr}^0_2(n+3)\colon V_1\subset V_2\big\},
\end{gather*}
where $\mathrm{Gr}^0_k(m)$ denotes the space of $k$-dimensional subspaces of $\RR^{m}$ which are null with respect to the inner product $S$, e.g., $\mathrm{Gr}_1^0(n+3)\cong Q_{p+1,q+1}^{n+1}$. Taking a basis $\{e_1,\dots,e_m\}$ for $\RR^m$, the transitivity of the action of $G_{p,q}:=\mathrm{O}(p+2,q+2)$ on $\mathrm{Gr}^0_k(m)$ results in isomorphisms $\mathrm{Gr}^0_k(m)\cong G_{p,q}/P_k$ where $P_k$ is the parabolic subgroup isomorphic to the stabilizer of the subspace spanned by $\{e_1,\dots,e_k\}$.

Let $m=n+3$, $n\geq 4$ and define the parabolic subgroup $P_{k_1, \dots, k_i}:=P_{k_1}\cap\cdots\cap P_{k_i}$ of $G$. Consequently, if follows that
\begin{gather}
 \NN\FF^{4}_{0,1} \cong G_{p,q}/P_{1,2,3},\qquad p+q=2,\nonumber\\
 \NN\FF^{n+1}_{0,1} \cong G_{p,q}/P_{1,2},\qquad p+q=n-1\geq 3.\label{eq:P_12}
\end{gather}
The relevant fibrations are shown in the diagram below:
\begin{gather*}
\begin{diagram}
 & &G_{p,q} & \\
 & &\dTo^\upsilon &\\
 &Q_{p+1,q+1}^{n+1} \lTo^{\quad\pi\quad } & \NN\FF^{n+1}_{0,1} & \rTo^{\quad\tau\quad } \mathrm{Gr}^0_2(n+3). &
\end{diagram}
\end{gather*}
Note that $\NN\FF^{n+1}_{0,1}$ is isomorphic to $\cC$. The Maurer--Cartan form on $G_{p,q}$ determines a canonical Cartan connection over $\cC$ which, as will be shown, is associated to the flat causal structure. Moreover, the fibration of $G_{p,q}$ over $\mathrm{Gr}^0_2(n+3)$ gives the flat model for Lie contact structures studied in \cite{MiyaokaLieContact,SYLieContactStr1}. As will be seen, these structures arise on the space of characteristic curves of certain classes of causal geometries.

The Maurer--Cartan form of $G_{p,q}$ can be written as
\begin{gather} \label{eq:FlatModelCausal}
 \varphi= \begin{pmatrix}
 \phi_0+\phi_n & -\pi_n & -\pi_b & -\pi_0 & 0 \\
 \omega^n & \phi_0-\phi_n & \gamma_b & 0 & -\pi_0 \\
 \omega^a & \theta^a & -\phi^a{}_b & \gamma^a &\pi^a \\
 \omega^0 & 0 & \theta_b & -\phi_0+\phi_n & -\pi_n\\
 0 & \omega^0 & -\omega_b & \omega^n & -\phi_0-\phi_n
 \end{pmatrix},
\end{gather}
with $\exd\varphi+\varphi\w\varphi=0$. The symbol $\ve_{ab}$ is used to lower and raise indices. The set of semi-basic\footnote{A 1-form $\omega\in T^*\cC$ is semi-basic with respect to the fibration $\pi\colon \cC\rightarrow M $ if $\omega(v)=0$ for all $v\in \pi_*$.} 1-forms with respect to the projections $\upsilon,\upsilon\circ\pi$ and $\upsilon\circ\tau$ are respectively
\begin{gather*}
 \cI_1=\big\{\omega^0,\dots,\omega^n,\theta^1,\dots,\theta^{n-1}\big\},\qquad
 \cI_2=\big\{\omega^0,\dots,\omega^n\big\},\\
 \cI_3=\big\{\omega^0,\dots,\omega^{n-1},\theta^1,\dots,\theta^{n-1}\big\}.
 \end{gather*}

Each tangent space $T_p\cC$ is equipped with a filtration defined as follows. At $p\in\cC$, the line over~$p$, denote by $\hat p\subset T_{\pi(p)}Q_{p+1,q+1}^{n+1}$, is null. Let $\omega^0$ denote a representative of the dual projective class of $\hat p$ with respect to the conformally defined inner product $S$. For instance, if $p=[e_n]$, then take $\omega^0=e_n^\flat$. Define $K^2:=\Ann(\pi^*\omega^0)\subset T\cC$. Using the tautological bundle
\begin{gather*}K^1:=\big\{v\in T_p\cC\colon \left(\pi_*(p)\right)(v)\in \hat p\subset T_{\pi(p)}Q_{p+1,q+1}^{n+1}\big\},\end{gather*}
the filtration $K^1\subset K^2\subset K^3:=T\cC$ is obtained. At $p\in \cC$, let $K_{-1}(p):=K^1(p)$, $K_{-2}:=K^2(p)/K^1(p)$ and $K_{-3}:=K^3(p)/K^2(p)$. The graded vector space
\begin{gather} \label{eq:Flat-model-Grading}
 \fm(p)=K_{-1}(p)\oplus K_{-2}(p)\oplus K_{-3}(p)
\end{gather}
can be associated to the filtration of $T_p\cC$. Note that at each $p\in \cC$, $\fm(p)\cong\fm$. The Lie bracket on $T\cC$ induces a bracket on $\fm$ which is tensorial and makes $\fm$ a graded nilpotent Lie algebra (GNLA), i.e., $[K_i,K_j]=K_{i+j}$ with $K_i=0$ for all $i\geq 4$ and $K_{-1}$ the generating component. Let $G_0$ denote the subgroup of automorphisms of $T\cC$, i.e., $\mathrm{O}(p+2,q+2)$, which preserves the grading of $\fm$ via the adjoint action. Let $K_0$ denote the Lie algebra of $G_0$. Then, the Lie algebra of automorphisms of $\cC$, i.e., $\fg=\orth(p+2,q+2)$, admits a $\ZZ$-grading defined by
\begin{gather*}\fg_k:=\{g\in\fg\colon g\cdot K_i\subset K_{i+k} \mathrm{\ for\ } 0\geq i\geq -3 \}.\end{gather*}
This is a $|3|$-grading, i.e., $\fg=\oplus_{i=-3}^3\fg_i$ such that $[\fg_i,\fg_j]\subset\fg_{i+j}$ with $\fg_i=0$ for $|i|>3$, and $\fg_{-1}$ generates $\fg_-:=\fg_{-1}\oplus\fg_{-2}\oplus\fg_{-3}$.

Note that $K^1=E \oplus \mathrm{Ver}$ where $\mathrm{Ver}$ is the vertical distribution with respect to the fibration $\pi$ and $E=\pi_*(\fg_{-1})$. Moreover, $T\cC$ is equipped with a conformally defined quadratic form obtained by lifting the canonical one defined on $Q_{p+1,q+1}^{n+1}$,
\begin{gather*} g=2\omega^0\circ\omega^n-\varepsilon_{ab}\omega^a\circ\omega^b,\end{gather*} which is degenerate on $\mathrm{Ver}$. The group of automorphisms of $T\cC$, which preserves the grading of $\fm$, also preserves the conformal class of the lifted quadratic form $g$ on $T\cC$.
The Lie algebra $\fg_0\subset \fg$ that corresponds to the group of automorphisms of $\fm$ which preserves the grading, is isomorphic to $\mathfrak{co}(p,q)\oplus\RR$.

Using the Maurer--Cartan form $\varphi$, the subspaces $\fg_{-3},\fg_{-2}, $ and $\fg_{-1}$ are dual to the Pfaffian systems $\cI_{-3}:=\{\omega^0\}$, $\cI_{-2}:=\{\omega^a\}_{a=1}^{n-1}$, and $\cI_{-1}:=\{\omega^n,\theta^a\}_{a=1}^{n-1}$.

In Sections \ref{sec:norm-cart-conn-1} and \ref{sec:harmonic-curvature} the correspondence between causal structures and parabolic geo\-metries of type $(B_{n-1},P_{1,2})$ and $(D_n,P_{1,2})$ for $n\geq 4$ and $(D_3,P_{1,2,3})$ will be made more explicit.

\subsection{The first order structure equations}\label{sec:structure-equations}
In this section the first order structure equations for causal geometries are derived via the method of equivalence. The procedure involves successive coframe adaptations, as a result of which various important aspects of causal geometries are unraveled, e.g., the quasi-contact structure of~$\cC$, the projective invariants of the fibers $\cC_x$, and the \wsf curvature. Also, it is noted that a~causal structure can be equivalently formulated as a codimension one sub-bundle of~$\PP T^*M$, which is called a dual formulation of causal geometries. Such formulation is traditionally preferred in the context of geometric control theory \cite{AZ-JacobiCurves}.

The derivation of the first order structure equations should be compared to that of Finsler structures as presented in~\cite{BryantRemarksFinsler}.

\subsubsection{The projective Hilbert form}\label{sec:hilbert-form}

Given a cone structure of codimension one $(\cC,M)$, the subset $\cC_x\subset \PP T_xM$ is a projective hypersurface, for all $x\in M$. At a point $y\in \cC_x$, the affine tangent space $\widehat{T}_y\cC_x$ is defined as the hyperplane $T_z\wh \cC_x\subset T_xM$ tangent to the cone $\wh \cC_x\subset T_xM$, containing $z\in\hat y$. Note that this subspace is independent of the choice of $z$. It is important here to distinguish between the affine tangent space $\widehat{T}_y\cC_x\subset T_xM$ and the vertical tangent space $\mathrm{V}T_{(x;y)}\cC$ corresponding to the fibration $\pi\circ \iota\colon \cC\ra M$.

The 1-form $\omega^0\in T^*\cC$, at $(x;y)\in\cC$ is defined as a multiple of $(\pi\circ\iota)^*(\alpha)$ where $\alpha:=\Ann\big(\widehat{T}_y\cC_x\big)\in T_x^*M$. By definition, $\omega^0$ is defined up to a multiplication by a nowhere vanishing smooth function on $\cC$ and is therefore semi-basic with respect to the fibration $\pi\colon \cC\ra M$. In this article $\omega^0$ is referred to as the projective Hilbert form.

\subsubsection{The adapted flag}\label{sec:adapted-flag}

At each point $(x;y)\in\cC$, there is a natural basis $\big(\fomegaz,\dots,\fomegann\big)$ for $ T_xM$ adapted to the flag
\begin{gather*}
\hat{y}\subset \widehat{T}_{y}\cC_x\subset T_xM,
\end{gather*}
 where $\big\{\fomegann\big\} $ and $\big\{\fomegaone,\ldots,\fomegann\big\}$, respectively, span $\hat y$ and $\widehat{T}_{y}\cC_x$.

In terms of the graph \eqref{eq:ConeStr-DefGraph}, the local expressions are
\begin{gather*} 
 \frac{\partial}{\partial\omega^n}= F\frac{\partial}{\partial x^0}+y^a\frac{\partial}{\partial x^a}+\frac{\partial}{\partial x^n},\qquad
\frac{\partial}{\partial\omega^a} =\partial_\una F\frac{\partial}{\partial x^0}+\frac{\partial}{\partial x^a},\qquad
\frac{\partial}{\partial\omega^0} =\frac{\partial}{\partial x^0}.
\end{gather*}
The dual coframe $\big(\omega^0,\ldots,\omega^n\big)$ can be written as
\begin{gather} \label{eq:AdaptedCoframe01}
 \omega^n=\exd x^n,\qquad
 \omega^a=\exd x^a-y^a\exd x^n,\qquad
 \omega^0=\exd x^0-\partial_{\una}F\exd x^a+(y^a\partial_\una F-F)\exd x^n.
\end{gather}
The pull-back of these 1-forms to $\cC$ spans the semi-basic 1-forms with respect to the fibration $\pi\colon \cC\ra M$. At each point $p\in\cC$, given by $p=(x^i;y^i)$, these 1-forms are adapted to the flag
\begin{gather}\label{Conframe-adaptation-00}
 \mathrm{V} T_p\cC\subset K_p\subset H_p\subset T_p\cC,
\end{gather}
 where $\mathrm{V}T_pM=\Ker\{\omega^0,\dots,\omega^n\}$ is the vertical tangent bundle, $K_p=\Ker\{\omega^0,\dots,\omega^{n-1}\}$ is the fiber of the tautological vector bundle given by
\begin{gather*}K_p=\{v\in T_p\cC\,|\, \pi_*(p)(v)\in \hat y\subset T_{\pi(p)}M\},\end{gather*}
where $p=(x;y)$, the line $\hat y$ corresponds to the projective class of $y$ and $H_p=\Ker\{\omega^0\}$ is the quasi-contact distribution (see Section~\ref{sec:char-curv}). To clarify the notation, if $p=(x;y)$ then the fiber $\mathrm{V}T_p\cC$ is also denoted by $T_y\cC_x$ and the quasi-contact distribution can be expressed as the splitting $T_y\cC_x\oplus\wh T_y\cC_x$ which is not unique.

\subsubsection{Causal structures and the first coframe adaptation}\label{sec:tang-non-degen}

Recall that a projective hypersurface $\cV\subset\PP^n$ is called \emph{tangentially non-degenerate} if its projective second fundamental form has maximal rank (see \cite{AG-Projective,SasakiBook}).
Now the stage is set to define causal structures.
\begin{df} \label{def:causal-str-def} A \textit{causal structure} of signature $(p+1,q+1)$ on $M^{n+1}$ is a cone structure of codimension one, $(M,\cC)$, where the fibers $\cC_x$ are tangentially non-degenerate projective hypersurfaces and their projective second fundamental form has signature $(p,q)$ everywhere.
\end{df}
In Section~\ref{sec:2-adapted-coframe} it is shown that the tangential non-degeneracy assumption of the fibers $\cC_x$ implies that the signature of $h^{ab}$ at any point of $\cC$ determines its the signature of $h^{ab}$ everywhere.

Let $\exd_V$ denote the vertical exterior derivative, i.e., the exterior differentiation with respect to fiber variables $y^i$. Using the expressions~\eqref{eq:AdaptedCoframe01}, it follows that
\begin{gather} \label{eq:d_v-omega0}
 \exd_V\omega^0\equiv -\theta^0{}_a\w\omega^a,\qquad \operatorname{mod} \ I\big\{\omega^0\big\},
\end{gather}
where $\theta^0{}_a=\partial_\una\partial_\unb F\exd y^b$. As a result of \eqref{eq:d_v-omega0} one obtains
 \begin{gather*}
\exd\omega^0=-2\theta^0{}_0\w\omega^0-\theta^0{}_a\w\omega^a+T_{0i}\omega^0\w\omega^i+\half T_{ab}\omega^a\w\omega^b+T_{an}\omega^a\w\omega^n.
 \end{gather*}
By setting
$\phi_0:= \theta^0{}_0-T_{0i}\omega^i$, $ \theta_a:= \theta^0{}_a-\half T_{ab}\omega^b -T_{an}\omega^n, $
one arrives at
 \begin{gather} \label{eq:domega0}
\exd\omega^0=-2\phi_0\w\omega^0-\theta_a\w\omega^a,
 \end{gather}
 which implies that $\omega^0$ has rank $2n-1$, i.e., $\omega^0\w(\exd\omega^0)^{n-1}\neq 0$.

The 1-forms $\theta_a$ define vertical 1-forms with respect to the fibration $\pi\colon \cC\ra M$. The 1-forms~$\theta_a$ are linearly independent everywhere when the fibers $\cC_x$ are tangentially non-degenerate at every point, i.e., $\det(\partial_\una\partial_\unb F)\neq 0$.

Since the vertical 1-forms $\theta_1,\ldots,\theta_{n-1}$ are linearly independent, at each point $(x;y)\in\cC$, one obtains a coframe $\cI_{\mathsf{tot}}:=\{\omega^0,\ldots,\omega^n,\theta_1,\ldots,\theta_{n-1}\}$ for $T_{(x;z)}\cC$ adapted to the flag
\begin{gather} \label{eq:flag}
 \mathrm{V} T_p\cC\subset K_p\subset H_p\subset T_p\cC,
\end{gather}
defined previously in \eqref{Conframe-adaptation-00}. Such a coframe is called \textit{$1$-adapted}. For a quadric, this filtration coincides with the one introduced in \eqref{eq:Flat-model-Grading}. In analogy with the flat model, the following algebraic ideals are defined:
\begin{gather*} 
 \cI_{\mathsf{tot}}=I\big\{\omega^0,\dots,\omega^n,\theta^1,\dots,\theta^{n-1}\big\},\qquad
 \cI_{\mathsf{bas}}=I\big\{\omega^0,\dots,\omega^n\big\},\\
 \cI_{\mathsf{char}}=I\big\{\omega^0,\dots,\omega^{n-1},\theta^1,\dots,\theta^{n-1}\big\}.
\end{gather*}

Using the expressions~\eqref{eq:AdaptedCoframe01}, equation \eqref{eq:domega0} gives
\begin{gather} \label{eq:theta-a-coframe01}
 \theta_a=\partial_\una\partial_\unb F\exd y^b+\partial_0\partial_\una F\exd x^0+\partial_b\partial_\una F\exd x^b+(\partial_\una\partial_n F-\partial_aF-\partial_0F\partial_\una F)\exd x^n.
\end{gather}
 From the expression above it follows that the linear independence of $\theta_a$'s modulo $\omega^i$'s is equivalent to the non-degeneracy of the vertical Hessian $\partial_\una\partial_\unb F$.

The fact that such a coframing is adapted to the flag \eqref{eq:flag} and satisfies \eqref{eq:domega0} implies that different choices of 1-adapted coframes are related by
\begin{alignat*}{3}
& \omega^0 \lmt \bfa^2\omega^0,\qquad && \omega^a \lmt \bfE^a\omega^0+\bfA^a{}_b\omega^b,& \\
& \omega^n \lmt \bfe \omega^0+\bfB_a\omega^a+\bfb^2\omega^n,\qquad &&\theta_a \lmt \bfC_a\omega^0+\bfS_{ab}\big(\bfA^{-1}\big)^b{}_c\omega^c+\bfa^2\big(\bfA^{-1}\big)^b{}_a\theta_b.&
\end{alignat*}
In the language of $G$-structures, the structure group of the bundle of 1-adapted coframes, $G_1$, is of the form
 \begin{gather} \label{eq:G1}
\left\{
\begin{pmatrix}
 \bfa^2 & 0 & 0 & 0 \\
 \bfE & \bfA & 0 & 0 \\
 \bfe & \bfB & \bfb^2 & 0 \\
 \bfC & {}^t\!\bfA^{-1}\bfS & 0 & \bfa^2\,{}^t\!\bfA^{-1}
\end{pmatrix}
 \vl
\begin{array}{@{}l@{}}
 \bfA\in \GL_{n-1}(\RR),\quad \bfa, \bfb \in \RR\backslash\{0\} , \\
 {}^t\bfB, \bfC, \bfE\in \RR^{n-1},\quad \bfe\in \RR,\\
 \bfS\in M_{n-1}(\RR),\quad \bfS={}{}^t \bfS
\end{array}
 \right\}.
 \end{gather}
 \begin{rmk}\label{rmk:positive-causal-factor}
 The reason for using $\bfa^2$ and $\bfb^2$ in $G_1$ is as follows. Let $L(x^i;y^i)$ be a Lagrangian function for a causal structure, i.e., the vanishing set of $L$ defines $\cC$. As is conventional in the case of pseudo-conformal structures, it is desirable to restrict to the set of causal transformations preserving the sign of $L$ in a neighborhood of the null cones, i.e., restricting to the set of Lagrangians of the form~$S^2 L$ for some nowhere vanishing homogeneous function $S\colon TM\ra \RR$. Recall that the projective Hilbert form in terms of~$L$ is given by $\omega^0=\partial_\uni L\exd x^i$.
Replacing~$L$ by~$S^2 L$ and restricting to the zero locus of $L$, the projective Hilbert form transforms to~\smash{$S^2(L_\uni\exd x^i)=S^2\omega^0$}. This observation justifies the use of~$\bfa^2$ in~$G_1$.
Moreover, recall that a~line bundle is called orientable if it has a non-vanishing global section. In the case of contact manifolds, a~contact distribution is called \emph{co-orientable } or \emph{transversally orientable} if it has a~globally defined contact form~\cite{LMSympl87}. Similarly, a quasi-contact structure is called co-orientable if it has a globally defined quasi-contact form $\omega^0$. If the quasi-contact distribution on $\cC$ is given a co-orientation, it is preserved under the action of $G_1$. Similarly, if the characteristic field of a quasi-contact structure (see Section~\ref{sec:char-curv}) is given an orientation, then having~$\bfb^2$ in~\eqref{eq:G1} implies that the structure group $G_1$ preserves the orientation of the characteristic field. Moreover, the conformal scaling of the quadratic form $g$ defined in Section~\ref{sec:conf-class-quadr-1} by the action of~$G_1$ is given by~$\bfa^2\bfb^2$. As a result, it is possible to define space-like and time-like horizontal vector fields $v$ in $T\cC$ via the signature of $g(v,v)$.
 \end{rmk}

To implement Cartan's method of equivalence, one lifts the 1-adapted coframe to the $G_1$-bundle in the following way (see \cite{Gardner-Book, IL-Book, Olver-Equiv-Book} for the background).
Recall that the set of all coframes on $\cC$, defines a $\mathrm{GL}(2n)$-principal bundle $\tilde\varsigma\colon \tilde\cP\ra\cC$, with the right action
\begin{gather*}R_{g}( \vartheta_p)=g^{-1}\cdot \vartheta_p,\end{gather*}
where $\vartheta_p\in\varsigma^{-1}(p)$ is a coframe at $p\in \cC$, $g\in \mathrm{GL}(2n)$, and the action on the right hand side is the ordinary matrix multiplication.

Among the set of all coframes on $\cC$, let $\vartheta$ represent 1-adapted coframes ${}^t(\omega^0,\dots,\theta_{n-1})$, which are well-defined up to the action of $G_1$. As a result, the 1-adapted coframes are sections of a~$G_1$-bundle $\varsigma_1\colon \cP_1\ra\cC$. On $\cP_1\cong \cC\times G_1 $ one can define a~canonical set of 1-forms by setting
\begin{gather*}\underline \vartheta(p,g)={}^t\big(\underline\omega^0,\dots,\underline\theta_{n-1}\big):= g^{-1}\cdot \vartheta_p,\end{gather*}
where $\vartheta_p$ is a choice of 1-adapted coframe at $p\in\cC$.

The exterior derivatives of the 1-forms in $\underline \vartheta$ are given by
\begin{gather} \label{eq:str-eqns-strart}
 \exd\underline \vartheta=\exd g^{-1} \w \vartheta+g^{-1}\cdot \exd \vartheta=-g^{-1}\cdot\exd g\cdot \w g^{-1} \vartheta+g^{-1}\cdot\exd \vartheta=-\alpha\w \underline \vartheta+ T,
\end{gather}
where $\alpha(p,g)= g^{-1}\cdot\exd g$ is $\fg_1$-valued 1-form where $\fg_1$ denotes the Lie algebra of $G_1$. The 2-forms $T(p,g)$ can be expressed in terms of
$\underline\omega^i\w\underline\omega^j$, $\underline\theta_a\w\underline\omega^i$$,\underline\theta_a\w\underline\theta_b$. The terms constituting $T$ are called the torsion terms. They are semi-basic with respect to the projection $\varsigma_1$. Via $\alpha$, one can identify the vertical tangent spaces of the fibers of $\varsigma_1\colon \cP_1\ra\cC$ with $\fg_1$. The $\fg_1$-valued 1-form $\alpha$ is called a \emph{pseudo-connection} on $\cP_1$ which is not yet uniquely defined. Recall that the difference between a \emph{connection} and a pseudo-connection on $G_1$-bundle lies in the fact that a connection form $\tilde\alpha$ satisfies the equivariance condition $R^*_g\tilde\alpha=\mathrm{Ad}(g^{-1})\tilde\alpha$, where $g\in G_1$, while a pseudo-connection in general does not.

A choice of pseudo-connection in \eqref{eq:str-eqns-strart} gives rise to structure equations for the 1-forms $\underline\vartheta$. For 1-adapted coframes, the pseudo-connection forms are of the form,{\samepage
\begin{gather*} 
\begin{pmatrix}
 2\phi_0 & 0 & 0 & 0 \\
 \gamma^a{}_0 & \psi^a{}_b &0 &0\\
 \gamma^n{}_0 & \gamma^n{}_a & 2\phi_{n}& 0 \\
 \pi_a & \pi_{ab} & 0 & 2\phi_0 \delta^a_b-\psi^a{}_b
\end{pmatrix},
\end{gather*}
where $\pi_{ab}=\pi_{ba}$.}

 It follows that
\begin{subequations} \label{eq:StrEqns-01}
\begin{gather}
 \exd\underline\omega^0 = -2\phi_0\w\underline\omega^0-\underline\theta_a\w\underline\omega^a,\label{1-adapted-1}\\
 \exd\underline\omega^a=-\gamma^a{}_0\w\underline\omega^0 -\psi^a{}_b\w\underline\omega^b-h^{ab}\,\underline\theta_b\w\underline\omega^{\n},\label{1-adapted-2}\\
 \exd\underline\omega^{\n} =-\gamma^n{}_0\w\underline\omega^0 -\gamma^n{}_a\w\underline\omega^a-2\phi_{\n}\w\underline\omega^{\n},\label{1-adapted-3}\\
 \exd\underline\theta_a=-\pi_a\w\underline\omega^0 -\pi_{ab}\w\underline\omega^b+\psi^b{}_a\w\underline\theta_b-2\phi_0\w\underline\theta_a
+\sigma_a\w\underline\omega^\n+\sigma^b{}_a\w\underline\theta_b,\label{1-adapted-4}
\end{gather}
\end{subequations}
where the 1-forms $\sigma_a$, $\sigma^b{}_a$, are the torsion terms and vanish modulo~$\Itot$. Their components are determined in Section~\ref{sec:conf-flag-curv-fourth-order}.

The equation of $\exd\underline\omega^0$ follows directly from~\eqref{eq:domega0}. Moreover, the Pfaffian system $\Ibas$, is completely integrable with the fibers $\cC_x$ being its integral manifolds. It follows that $\exd\omega^a=\tau^a{}_i\w\omega^i$, and $\exd\omega^n=\tau^n{}_i\w\omega^i$, where $\tau^a{}_i$, $\tau^n{}_i\equiv 0$ modulo~$\Itot$. Using the fact that the action of~$G_1$ preserves~$\Ibas$, and that two pseudo-connection differ by a semi-basic $\fg_1$-valued 1-form, all the 1-forms $\tau^a_i$, $\tau^n{}_i$ can be absorbed into the pseudo-connection forms. The only torsion term left is $h^{ab}\underline\theta_b\w\underline\omega^n$, for some function~$h^{ab}$, which is discussed in Section~\ref{sec:2-adapted-coframe}. From now on, since structure equations are always written for the lifted coframe $\underline\vartheta$, the underline in $\underline\omega^i$ and $\underline\theta_a$'s will be dropped.

\begin{rmk} \label{rmk:oriented-cone-str}
The oriented projectivized tangent bundle $\PP_+TM$ is defined so that at $x\in M$ the fiber $\PP_+T_xM$ is isomorphic to the space of \textit{rays} in $T_xM$, or, equivalently, the set of oriented real lines in $T_xM$ passing through the origin. As a result, $\PP_+ T_xM$ is the nontrivial double cover of $\PP T_xM$.

The above definition of cone structures, and more specifically causal structures, can be changed so that $\cC$ is a sub-bundle of $\PP_+ TM$. In the case of causal structures immersed in~$\PP_+ TM$, the assumption that the fibers $\cC_x=\iota^{-1}(x)$ be connected is not desirable. This is due to the fact that $\cC_x$ should be thought of as the light cones which are comprised of the future and past light cones whose projectivizations are disjoint. Thus, one expects $\cC_x=\cC^+_x\sqcup\cC^-_x$. This way, in analogy with Finsler geometry, a cone structure is called \emph{reversible} if $\widehat\cC=-\widehat\cC$. Otherwise, it is called \emph{non-reversible}. It should be noted that among oriented causal structures, non-reversible ones can be of interest in certain problems due to possible differences in their future and past light cones. This is similar to the case of non-reversible Finsler structures containing the important class of Randers spaces~\cite{BRS-Zermelo}. In~\cite{KP-causal}, Kronheimer and Penrose use the term \emph{self-dual} for a~reversible causal structure. This is due to the fact that they define the dual of a causal structure to be the causal structure obtained from replacing all the causal relations between the points of the space by their inverses, i.e., interchanging the words ``future'' and ``past''.
\end{rmk}

\subsubsection{A dual formulation of causal structures}\label{sec:legendre-transf}

In the context of geometric control theory, geometric structures arising from distributions are typically formulated in the cotangent bundle~\cite{AZ-JacobiCurves}. Firstly, the definition of a Legendre transformation is recalled.
\begin{df} Given a sufficiently smooth real-valued function $L$, usually referred to as a~Lag\-rangian, on $TM$, the associated Legendre transformation $\Leg\colon TM\rightarrow T^*M$, is defined by
\begin{gather*}\Leg=\hat\pi\circ \exd_VL,\end{gather*}
where $\exd_VL$ defines a section of the product bundle $ TM\times T^*M$, with the projection map $\hat\pi\colon TM\times T^*M\ra T^*M$. In local coordinates, $\Leg(x^i;y^i)=(x^i;p_i)$ where $p_i=\frac{\partial L}{\partial y^i}$, and $(p_0,\dots,p_n)$ are local coordinates for the fibers of~$T^*M$.
\end{df}

 The Legendre transformation is a local diffeomorphism provided that $\Leg_*=\frac{\partial^2 L}{\partial y^i\partial y^j}$ is non-degenerate everywhere. The inverse Legendre transformation $\Leg^{-1}$ can be obtained by the same procedure using the Hamiltonian function $Q:=\cA\circ\Leg^{-1}$ where $\cA:=\exd L(\underline A)-L$ with~$\underline A$ being the Liouville vector field, i.e., $\cA=y^i\frac{\partial L}{\partial y^i}-L$ in local coordinates. When the degree of homogeneity of $L$ is $r\neq 1$, Euler's theorem gives $Q=(r-1)L\circ\Leg^{-1}$.

Changing the defining function to $L'(x,y)=S(x,y)L(x,y)$, the function $\Leg$ remains invariant over $\cC$, hence the vanishing set of $Q$ in $T^*M$ remains invariant.
Consequently, one can equivalently study the \textit{dual causal structure} $\cC^*:=\Leg(\cC)\subset \PP T^*M$ whose fibers $\cC^*_x\subset\PP T^*_xM$ are the projective dual of the projective hypersurfaces $\cC_x\subset \PP T_xM$.

Via the Legendre transformation the projective Hilbert form $\omega^0=\partial_\uni L\,\exd x^i$ is the pull-back of the canonical contact form $\lambda=p_i\exd x^i$ restricted to $\cC^*\subset \PP T^*M$, i.e., \begin{gather*}\Leg^*\big(p_i\exd x^i\big)=\frac{\partial L}{\partial y^i}(x,y)\exd x^i.\end{gather*}
Consequently, the coframe adaptation can be carried out using $\lambda$ for the dual cone structure. Hence, this dual formulation can be taken as the starting point for studying causal structures.

\subsubsection{Characteristic curves and a quasi-Legendrian foliation}\label{sec:char-curv}
A causal structure comes with a distinguished set of curves which can be defined using the quasi-contact nature of $\cC$. Recall from Example~\ref{exa:contact-quasi-contact} that a quasi-contact structure (or even contact structure) on a $2n$-dimensional manifolds was defined locally in terms of a 1-form of rank $2n-1$. As a result, there is a unique direction along which $\omega^0$ and $\exd\omega^0$ are degenerate, i.e., $[\bfv]\subset T\PP TM$ such that $\bfv\im \omega^0=0$ and $\bfv\im\exd\omega^0=0$, which is called the \emph{characteristic field} of the quasi-contact structure and its integral curves are called \emph{characteristic curves} (cf.~\cite{HsuCalculus92}). A~section of the characteristic field is called a \emph{characteristic vector field}.

 In local coordinates using the defining function $F$ in~\eqref{eq:ConeStr-DefGraph}, one obtains that the vector field
\begin{gather*}\textstyle F\frac{\partial}{\partial x^0}+y^a\frac{\partial}{\partial x^a}+\frac{\partial}{\partial x^n}+ A^a\frac{\partial}{\partial y^a},\end{gather*}
where
\begin{gather*}A^b\partial_\una\partial_\unb F=-F\partial_0\partial_\una F-y^b\partial_b\partial_\una F+\partial_aF+\partial_0F\partial_\una F-\partial_\una\partial_n F\end{gather*}
represents a characteristic vector field on $\cC$.

\begin{df} Let $M$ be a $2n$-dimensional manifold with a quasi-contact 1-form $\omega\in\Gamma(T^*M)$. A submanifold $N\subset M$ is called \emph{quasi-Legendrian} if $T_xN\in \Ker\omega$ and $T_xN\im\exd\omega=0$ for all $x\in N$, with the property that $T_xN$ is transversal to the characteristic field of $\omega$.
\end{df}

 The bundle $\cC$ is $2n$-dimensional and $\omega^0\in T^*\cC$ has maximal rank $2n-1$, therefore, $\cC$~has a~quasi-contact structure. The fibers $\cC_x$ form a quasi-Legendrian foliation of $\cC$ since the vertical tangent spaces $T_{y}\cC_x$ are annihilated by $\cI_{\mathsf{bas}}=\{\omega^0,\dots,\omega^n\}$.

Moreover, using \eqref{eq:domega0}, the characteristic field of $\omega^0$ is the kernel of the Pfaffian system
$\cI_{\mathsf{char}}=\{\omega^0,\dots,\omega^{n-1},\theta_1,\dots,\theta_{n-1}\}$, which justifies the subscript $\mathsf{char}$.
\begin{rmk} Recall that the geodesics of a Finsler structure are defined as the projection of integral curves of the Reeb vector field associated to the contact structure on the indicatrix bundle induced by the Hilbert form~\cite{BryantRemarksFinsler}. Here, the projection of the characteristic curves gives an analogue of null geodesics in pseudo-conformal structures.
\end{rmk}

\subsubsection{Projective second fundamental form and the second coframe adaptation}\label{sec:2-adapted-coframe}
The next coframe adaptation involves the projective second fundamental form of the fibers $\cC_x$.

Let $\exd\omega^a\equiv -h^{ab}\theta_b\w\omega^n$ modulo $I\{\omega^0,\dots,\omega^{n-1}\}$. Using the identity $\exd^2\omega^0=0$, it follows that
\begin{gather*}h^{ab}\theta_a\w\theta_b\w\omega^n=0,\end{gather*}
i.e., $h^{ab}=h^{ba}$.
Using the local expressions of $\omega^a$ in \eqref{eq:AdaptedCoframe01}, one obtains that
\begin{gather*}\exd_V\omega^a\equiv -\exd y^a\w\omega^n \qquad \operatorname{mod} \ I\big\{\omega^0,\dots,\omega^{n-1}\big\}.\end{gather*}
 Recall from \eqref{eq:d_v-omega0} that $\exd y^a=h^{ab}\theta^0{}_b$ where $(h^{ab})=(\partial_\una\partial_\unb F)^{-1}$.

Using the structure equations \eqref{eq:StrEqns-01}, the infinitesimal action of the structure group $G_1$ on~$h^{ab}$ is obtained from the identity $\exd^2\omega^a=0$, and is given by
\begin{gather} \label{eq:InfGpAction-hab}
\exd h^{ab}+h^{ac}\psi^{b}{}_{c}+h^{bc}\psi^a{}_c-2h^{ab}(\phi_0+\phi_\n) \equiv 0 \qquad \operatorname{mod} \ \Itot.
\end{gather}

The relation above implies that the functions $h^{ab}(p,g_0)$ and $\tilde h^{ab}(p,g_1)$ where $g_1=g g_0$, with $g\in G_1$ represented by the matrix \eqref{eq:G1}, are related by
\begin{gather}
 \label{eq:FullGpAction-hab}
\tilde h^{ab}= {\bfa^2\bfb^2}(\bfA)^{-1}{}^a{}_{c}(\bfA)^{-1}{}^b{}_dh^{cd}.
\end{gather}
It follows that the bilinear form $\phi_2=h^{ab}\theta_a\circ\theta_b$ when pulled-back to the fibers $\cC_x$ coincides with its projective second fundamental form (see \cite{AG-Projective, SasakiBook}).

Note that by definition \ref{def:causal-str-def}, tangential non-degeneracy of the fibers $\cC_x$ implies that $h^{ab}$ is non-degenerate everywhere in $\cC$. As a result, the signature of $h^{ab}$, as a well-defined tensor over~$\cC$, remains the same everywhere.

Assuming that $h^{ab}$ has signature $(p,q)$, the relation~\eqref{eq:FullGpAction-hab} implies that $\bfA^a{}_b$ can be chosen so that $h^{ab}$ is normalized to~$\ve^{ab}$. By restricting the $G_1$-bundle to the coframes for which $h^{ab}=\varepsilon^{ab}$, and using the relation~\eqref{eq:InfGpAction-hab}, it follows that
\begin{gather}\label{eq:psi-phi-1-adaptation}
\varepsilon^{ac}\phi^b{}_c+\varepsilon^{bc}\phi^a{}_c\equiv 0 \qquad \operatorname{mod} \ \Itot,
\qquad\mathrm{where}\quad \phi^a{}_b=\psi^a{}_b-(\phi_0+\phi_\n)\delta^a{}_b.
\end{gather}
As a result, for such coframes one can replace $\psi^a{}_b$ by $\phi^a{}_b+ (\phi_0+\phi_\n)\delta^a{}_b$ where the matrix-valued 1-forms $\phi^a{}_b$ take values in $\orth(p,q)$ modulo $\Itot$

Coframings with normalized $h^{ab}$ are \textit{$2$-adapted}, and they are local sections of a $G_2$-bundle over $\cC$ where
\begin{gather*} 
G_2=\left\{
\begin{pmatrix}
 \bfa^2 & 0 & 0 & 0 \\
 \bfE & \bfa\bfb\bfA & 0 & 0 \\
 \bfe & \bfB & \bfb^2 & 0 \\
 \bfC & \bfA\bfS & 0 & \frac{\bfa}{\bfb} \bfA
\end{pmatrix} \vl
\begin{array}{@{}l@{}}
 \bfA\in \Orth(p,q),\quad \bfa, \bfb \in \RR\backslash\{0\}, \\
 {}^t\bfB, \bfC, \bfE\in \RR^{n-1},\quad \bfe\in \RR,\\
 \bfS\in M_{n-1}(\RR),\quad \bfS= {}^t \bfS
\end{array}
 \right\}.
\end{gather*}

\subsubsection{The third coframe adaptation}\label{sec:3-adapted-coframe}
The 1-forms $\phi^a{}_b$ are $\orth(p,q)$-valued, i.e., $\phi_{ab}+\phi_{ba}\equiv 0$ modulo $\Itot$, where $\phi_{ab}=\varepsilon_{ac}\phi^c{}_b$. From~\eqref{eq:psi-phi-1-adaptation} it follows that
\begin{gather} \label{eq:F-abc-E-aic}
 \psi_{ab}=\phi_{ba}+(\phi_0+\phi_n)\delta_{ab}+E_{aib}\omega^i+F_a{}^{c}{}_b\theta_c,
\end{gather}
 where $E_{aib}=E_{bia}$ and $F_a{}^{c}{}_b=F_b{}^{c}{}_a$. Hence, with respect to 2-adapted coframes, the structure equations are
 \begin{subequations} \label{eq:StrEq-2-adapted-A}
\begin{gather}
 \exd\omega^0= -2\phi_0\w\omega^0-\theta_a\w\omega^a,\label{2-adapted-1}\\
 \exd\omega^a= -\gamma^a{}_0\w\omega^0-\phi^a{}_b\w\omega^b-(\phi_0 +\phi_\n)\w\omega^a -\ve^{ab} \theta_b\w\omega^{\n}\nonumber\\
\hphantom{\exd\omega^a=}{} -E^a{}_{ic} \omega^i\w\omega^c-F^{ab}{}_c \theta_b\w\omega^c,\label{2-adapted-2}\\
 \exd\omega^{\n} = -\gamma^n{}_0\w\omega^0-\gamma^n{}_a\w\omega^a-2\phi_{\n}\w\omega^{\n},\label{2-adapted-3}\\
 \exd\theta_a= -\pi_a\w\omega^0-\pi_{ab}\w\omega^b+\phi^b{}_a\w\theta_b -(\phi_0-\phi_\n)\w\theta_a +\sigma_a\w\omega^\n+\sigma^b{}_a\w\theta_b ,\label{2-adapted-4}
\end{gather}
 \end{subequations}
where $\ve^{ab}$ is used to raise the indices in $E_{aib}$ and $F_a{}^{c}{}_b$. By replacing
\begin{gather*}
 \phi^a{}_b \lmt \phi^a{}_b+\frac{1}{2}\big( E^a{}_{bc} \omega^c-\ve^{ae}E^{d}{}_{ec}\varepsilon_{da} \omega^c+\varepsilon_{cd}E^c{}_{ab} \omega^d\big),\\
\phi_n \lmt \phi_n+E^b{}_{nb} \omega^{\n},\qquad\gamma^a{}_0\lmt \gamma^a{}_0+E^a{}_{0b} \omega^b,
\end{gather*}
it can be assumed that $E^a{}_{bc}$ and $E^a{}_{0c}$ vanish, and $E^a{}_{nb}$ is trace-free.

Differentiating $\exd\omega^a$, and collecting the terms of the form $\tau^a{}_b\w\omega^b\w\omega^n$ for some 1-form sa\-tisfying $\tau_{ab}=\tau_{ba}$ and $\tau^a{}_a\equiv 0$ modulo $\Itot$, one obtains
\begin{gather*}
 \big[\big(\exd E_{anb}+2E_{anb}\phi_n+2E_{cn(a}\phi^c{}_{b)} +\pi_{ab}\big)\w\omega^n-\exd(\phi_0+\phi_n)\ve_{ab}\big]\w\omega^b = 0,
\end{gather*}
where $E_{anb}=\varepsilon_{ac}E^c{}_{nb}$. Contracting the above equations with $\ve^{ab}$, and using the properties $E_{anb}=E_{bna}, E^a{}_{na}=0$, and $\phi_{ab}=-\phi_{ba}$, it follows that $\exd(\phi_0+\phi_n)\equiv \ts\frac{1}{n-1}\pi_{ab}\ve^{ab}\w\omega^n$ modulo~$\Ichar$. As a result, the infinitesimal action of the structure group on~$E^a{}_{nb}$ is given by
\begin{gather} \label{eq:E_aib-InfGpAct}
 \exd E_{anb}+2E_{anb}\phi_n+2E_{cn(a}\phi^c{}_{b)} +\big(\pi_{ab}-\ts\frac{1}{n-1}\pi_{cd}\ve^{cd}\ve_{ab}\big)\equiv 0 \qquad \operatorname{mod} \ \Itot.
\end{gather}

Similarly to Section~\ref{sec:2-adapted-coframe}, it is possible to restrict to coframes with vanishing $E^a{}_{nb}$, and consequently, by relation \eqref{eq:E_aib-InfGpAct}, the 1-forms $\pi_{ab}$ are pure trace modulo $\Itot$, i.e., $\pi_{ab}=\pi_n\ve_{ab}$, modulo $\Itot$, for some 1-form $\pi_n$. Such coframes are called \textit{3-adapted}.

 The 3-adapted coframes are local sections of a $G_3$-bundle, where $G_3$ consists of matrices
 \begin{gather*}
 \left\{ \begin{pmatrix}
 \bfa^2 & 0 & 0 & 0 \\
 \bfE & \bfa\bfb\bfA & 0 & 0 \\
 \bfe & \bfB & \bfb^2 & 0 \\
 \bfC & \bfc{}^t\!\bfA\varepsilon & 0 & \frac{\bfa}{\bfb}{}^t\!\bfA
\end{pmatrix}
 \vl
\begin{array}{@{}l@{}}
 \bfA\in \Orth(p,q),\quad \bfa, \bfb \in \RR\backslash\{0\}, \\
 {}^t\bfB, \bfC, \bfE\in \RR^{n-1},\quad \bfc,\bfe\in \RR
\end{array}\right\}.
 \end{gather*}

\subsubsection[Symmetries of $F^{abc}$]{Symmetries of $\boldsymbol{F^{abc}}$}
In order to show $F^{abc}$ is completely symmetric, one first differentiates equation \eqref{2-adapted-2} to obtain
\begin{gather*}0\equiv \varepsilon^{ab}\big(\exd\theta_b-\phi^c{}_b\w\theta_c+(\phi_0-\phi_n)\theta_b +F^{cd}{}_b \theta_d\w\theta_c\big)\w\omega^\n \qquad \operatorname{mod} \ \omega^0, \ldots, \omega^{n-1}.\end{gather*}
Differentiating equation \eqref{2-adapted-1} yields,
\begin{gather*}0=\big(\exd\theta_c-\phi^a{}_c\w\theta_a+(\phi_0-\phi_\n)\theta_c +F^{ab}{}_c \theta_a\w\theta_b\big)\w\omega^c \qquad \operatorname{mod} \ I\big\{\omega^0,\omega^\n\big\}.\end{gather*}
Comparing the two equations above, one obtains, $F^{ab}{}_c\,\theta_a\w\theta_b=0$, i.e., $F^{ab}{}_c=F^{ba}{}_c$. Recall, in~\eqref{eq:F-abc-E-aic} the quantities $F^{ab}{}_c$ satisfy $F^{abc}=F^{cba}$, where $F^{abc}=\ve^{cd}F^{ab}{}_d$. As a result, $F^{abc}$ is totally symmetric.

\subsubsection{The fourth coframe adaptation, Fubini cubic form and the \wsf curvature} \label{sec:conf-flag-curv-fourth-order}
The next coframe adaptation reveals the essential invariants of a causal structure, which are the Fubini cubic forms of the fibers~$\cC_x$ and the Weyl shadow flag curvature (also referred to as the \wsf curvature).

Differentiating \eqref{2-adapted-2} and collecting terms of the form $\tau^{ab}{}_c\w\theta_b\w\omega^c$ satisfying $\tau_a{}^{b}{}_c=\tau_c{}^{b}{}_a$, one has
\begin{gather}
0\equiv \big[\big(\exd F^{abc}+F^{abd}\phi^c{}_d+F^{adc}\phi^b{}_d+F^{dbc}\phi^a{}_d-F^{abc}(\phi_0-\phi_n) \nonumber\\
\hphantom{0\equiv}{} -\big(\gamma^n{}_d\varepsilon^{dc}\varepsilon^{ab}-\gamma^{a}{}_0\ve^{bc} \big)\big)\w\theta_b -\exd(\phi_0+\phi_n)\ve^{ac} \big]\w\omega_c \qquad \operatorname{mod} \ \omega^0,\omega^n.\label{eq:InfGpAction-F-abc-01}
\end{gather}

Contacting by $\ve_{ab}$, and defining
\begin{gather*}F^a:=\varepsilon_{bc}F^{abc},\end{gather*}
 one arrives at
\begin{gather*}\exd(\phi_0+\phi_n)=\ts\frac{1}{n-1}\big(\gamma^b{}_0-\gamma^n{}_d\ve^{db}+\exd F^b\big)\w\theta_b.\end{gather*}
Putting this back in \eqref{eq:InfGpAction-F-abc-01}, and contracting with $\ve_{bc}$ gives
\begin{gather*}
\exd F^a \equiv -F^b\phi^a{}_b+F^a(\phi_0-\phi_n)-\ts\frac{n(n-2)}{n-1}\big(\gamma^a{}_0-\gamma^n{}_b\varepsilon^{ba}\big) \qquad \operatorname{mod} \ \Itot. \end{gather*}
As before, the infinitesimal action above makes it possible to translate $F^a$ to zero. Therefore, the class of \textit{$4$-adapted} coframes can be defined by the \textit{apolarity relation}
\begin{gather} \label{eq:apolarity-causal}
 F^a=0.
\end{gather}
 The resulting reduction in the structure group is realized via the relation $\gamma^a{}_0\equiv\varepsilon^{ab}\gamma^n{}_b$ mo\-du\-lo~$\Itot$.

It can be shown that the cubic form
\begin{gather*}F_3=F^{abc}\theta_a\circ\theta_b\circ\theta_c,\end{gather*}
 when pulled-back to the fiber $\cC_x$, coincides with Fubini cubic form of $\cC_x$ (see \cite{AG-Projective, SasakiBook}). This becomes evident in the derivation of the local expression of~$F^{abc}$ in~\eqref{eq:explicit-Fubini-cubic-form}.

 It is noted that because the second fundamental form is normalized to~$\ve_{ab}$, for any choice of 4-adapted coframe, the Fubini cubic form is well-defined tensor on~$\cC$.

The structure group $G_4$ is reduced to the group of matrices of the form
 \begin{gather*} 
 \left\{ \begin{pmatrix}
 \bfa^2 & 0 & 0 & 0 \\
 \bfE & \bfa\bfb\bfA & 0 & 0 \\
 \bfe & \bfB & \bfb^2 & 0 \\
 \bfC & \bfc{}^t \bfA\varepsilon & 0 & \frac{\bfa}{\bfb}{}^t\!\bfA
\end{pmatrix}
 \vl
\begin{array}{@{}l@{}}
 \bfA\in \Orth(p,q),\quad \bfa, \bfb \in \RR\backslash\{0\}, \\
 {}^t\bfB, \bfC \in \RR^{n-1},\quad \bfc,\bfe\in \RR,\\
 \bfE=\frac{a}{b}\bfA\,{}^t \bfB
\end{array}\right\},
 \end{gather*}
 where ${}^t\bfB$ denotes the transpose of $\bfB$ with respect to $\ve$.

By setting $\gamma^a{}_0-\varepsilon^{ab}\gamma^n{}_b=K^a{}_{i0}\,\omega^i+K^{ab}\theta_b$, the structure equations after the fourth coframe adaptation are of the form
 \begin{subequations} \label{eq:StrEq-4-adapted}
\begin{gather}
 \exd\omega^0= -2\phi_0\w\omega^0-\theta_a\w\omega^a,\label{4-adapted-1}\\
 \exd\omega^a= -\ve^{ab} \gamma_b\w\omega^0-\phi^a{}_b\w\omega^b-(\phi_0 +\phi_n)\w\omega^a -\varepsilon^{ab} \theta_b\w\omega^{n}\nonumber\\
\hphantom{\exd\omega^a=}{} -K^a{}_{i0} \omega^i\w\omega^0-K^{ab} \theta_b\w\omega^0-F^{ab}{}_c \theta_b\w\omega^c,\label{4-adapted-2}\\
 \exd\omega^{n} = -\gamma^n{}_0\w\omega^0-\gamma_a\w\omega^a-2\phi_{n}\w\omega^{n},\label{4-adapted-3}\\
 \exd\theta_a= -\pi_a\w\omega^0-\ve_{ab} \pi_{n}\w\omega^b+\phi^b{}_a\w\theta_b -(\phi_0-\phi_n)\w\theta_a+A^b{}_a\theta_b\w\omega^\n \nonumber\\
 \hphantom{\exd\theta_a=}{}+W_{anbn}\omega^b\w\omega^\n+B^{bc}{}_a\theta_b\w\theta_c +E_{an}{}^{c}{}_{b}\theta_c\w\omega^b +W_{anbc}\omega^b\w\omega^c,\label{4-adapted-4}
\end{gather}
 \end{subequations}
where $\gamma^n{}_a$ is being replaced by $\gamma_a$.

Differentiating \eqref{4-adapted-1} yields
\begin{gather*}
A^a{}_b{}=0, \qquad W_{anbn}{}=W_{bnan},\qquad E_{anbc}{}=E_{ancb},\\ K^{ab}=K^{ba},\qquad B^{ab}{}_c=0, \qquad W_{[a|n|bc]} =0.
\end{gather*}
By making the following changes
\begin{gather*}
 \gamma_a\lmt\gamma_a-K_{cn0} \omega^n-K_{(ab)0}\omega^a,\qquad \phi_n\lmt\phi_n-K_{cn0}\,\omega^c,\\
\phi^a_b\lmt\phi^a_b +\half\big( T^a{}_{bc} \omega^c-\varepsilon^{ae}T^{d}{}_{ec}\varepsilon_{da} \omega^c +\varepsilon_{cd}T^c{}_{ab} \omega^d\big) +\half \big( K^a{}_b-\ve^{ac}K^d{}_c\ve_{eb}\big) \omega^0,
\end{gather*}
where $T^a{}_{bc}=\delta^a{}_bK_{cn0}-\delta^a{}_cK_{bn0}$, and $K_{an0}=\ve_{ab}K^b{}_{n0}$, it can be assumed that $K^a{}_{i0}$ vanishes.

By replacing \begin{gather*}\pi_n\lmt\pi_n+ \ts\frac{1}{n-1}\ve^{ac}E_{anbc}\,\theta^b-\nfrac\ve^{cd}W_{cndn}\,\omega^n,\end{gather*} the traces of $W_{anbn}$ and $E^a{}_{nbc}$ are absorbed in $\pi_n$, and therefore, $W_{anbn}$ and $E_{ancb}$ can be chosen to be trace-free in $(a,b)$.

The symmetric trace-free tensor $W_{anbn}$ is called the \textit{Weyl shadow flag curvature} (the \wsf curvature), which, as will be shown, is the second essential invariant of causal structures.

 By differentiating \eqref{4-adapted-2}, and collecting 3-forms $A_{ab}\,\theta^a\w\omega^b\w\omega^n$ such that $A_{ab}=A_{ba}$, and $A^a{}_a=0$, it follows that
\begin{gather*}E_{anbc}=-f_{abc},\end{gather*}
where $f_{abc}=F_{abc;n}$, i.e., the derivative of $F_{abc}$ along the characteristic curves.

\subsubsection{The fifth coframe adaptation}\label{sec:5-adapted-coframe}

There is one more coframe adaptation that can be imposed on any causal structure, which reduces the structure group by one dimension. First, note that the derivative of \eqref{4-adapted-2} gives
\begin{gather} \label{eq:InfActionK-ab-01}
\big(\big(\exd K^{ab}+2K^{c(a}\phi^{b)}{}_c-2K^{ab}(\phi_0-\phi_n )-\ve^{ab}\gamma^n{}_0 -F^{abc}\gamma_c\big)\w\theta_b+\exd\gamma^a\big) \w\omega^0 \equiv 0
\end{gather}
modulo $ I\{\omega^1,\dots,\omega^n\}$. To find the infinitesimal group action on $K^{ab}$ via the relation above one needs to find the 2-form components $\nu^{ab}\w\theta_b$ of $\exd\gamma^a$ where $\nu^{ab}\not\equiv 0$ modulo $\Itot$. Differentia\-ting~\eqref{4-adapted-3}, it follows that
\begin{gather*}\big(\exd\gamma_a+F^{bc}{}_a\gamma_b\w\theta_c+\Lambda_a\big)\w\omega^a\equiv 0,\end{gather*}
where $\Lambda_a$ is of the form $\nu^b{}_a\w\theta_b$ with $\nu^b{}_a\equiv 0$ modulo $\Itot$. Inserting the above expression of $\exd\gamma_a$ into~\eqref{eq:InfActionK-ab-01}, one sees that the infinitesimal action of the structure group $G_5$ on $K^{ab}$ is given by
\begin{gather*}\exd K^{ab}+2K^{c(a}\phi^{b)}{}_c-2K^{ab}(\phi_0+\phi_n)+\ve^{ab}\gamma^n{}_0 -2F^{abc}\gamma_c \equiv 0 \qquad \operatorname{mod} \ \Itot.\end{gather*}
Note that $K^{ab}$ is not a tensor as its infinitesimal transformation rule in terms of the structure group involves~$F^{ab}{}_c$.

Using the apolarity relation \eqref{eq:apolarity-causal}, and the fact that $K_{ab}$ and $\phi_{ab}$ are symmetric and skew-symmetric respectively, the infinitesimal action of the structure group $G_4$ on $K:=\ve_{ab}K^{ab}$ is found to be
\begin{gather*}\exd K+K(\phi_0-\phi_n)+(n-1)\gamma^n{}_0\equiv 0 \qquad \operatorname{mod} \ \Itot.\end{gather*}

Now, the \textit{$5$-adapted} coframes can be defined by the property that $K=0$. Hence, one obtains
 \begin{gather*}\gamma^0{}_n=M_{ni} \omega^i+L^a \theta_a.\end{gather*}
By replacing
\begin{gather*}\gamma_a\lmt\gamma_a-M_{na}\,\omega^0-\ve_{ab}M_{nn}\,\omega^b,\qquad \phi_n\lmt\phi_n-\half M_{nn}\,\omega^0,\end{gather*}
it can be assumed that $M_{ni}=0$.

The 5-adapted coframes are local sections of a $G_5$-bundle where $G_5$ is the matrix group defined by
\begin{gather} \label{eq:G5-StrGroup}
\left\{\begin{pmatrix}
 \bfa^2 & 0 & 0 & 0 \\
 \bfE & \bfa\bfb\bfA & 0 & 0 \\
 \bfe & \bfB & \bfb^2 & 0 \\
 \bfC & \bfc\,{}^t\!\bfA\varepsilon & 0 & \frac{\bfa}{\bfb}\,{}^t\!\bfA
\end{pmatrix}\vl
\begin{array}{@{}l@{}}
 \bfa, \bfb\in \RR\backslash\{0\},\quad \bfc\in \RR, \\
 \bfA\in \Orth(p,q),\quad {}^t\bfB, \bfC\in\RR^{n-1}, \\
 \bfE=\frac{\bfa}{\bfb}\bfA\,{}^t\bfB,\quad \bfe=\frac{1}{2\bfb^2}\bfB\,{}^t\bfB
\end{array}
 \right\}.
\end{gather}

\subsubsection{Invariant conformal quadratic form}\label{sec:conf-class-quadr-1}
According to \eqref{eq:G5-StrGroup}, the structure group $G_5$ preserves the conformal class of the quadratic form
\begin{gather} \label{eq:HorizQuadraticFrom}
 g:= 2 \omega^0\circ\omega^n-\ve_{ab} \omega^a\circ\omega^b,
\end{gather}
which is semi-basic with respect to the fibration $\pi\colon \cC\ra M$. At $p=(x;[y])\in \cC$ the null cone of~$g_p$ in~$T_xM$ gives the osculating quadric of second order for $\cC_x$ at $[y]\in\cC_x$.

As will be shown in Section~\ref{sec:conformal-structures}, the vanishing of the Fubini form implies that the conformal class of~$g$ is well-defined on~$M$ and that the null cones of the causal structure coincide with those of~$g$. In that case, the $\orth(p,q)$-valued part of the associated conformal connection is given by
 \begin{gather*}
 \begin{pmatrix}
 \phi_n-\phi_0 & \theta & 0\\
 {}^t\!\gamma & \phi & {}^t \theta\\
 0 & \gamma & \phi_0-\phi_n
 \end{pmatrix},
\end{gather*}
 where ${}^t\!\gamma$ and ${}^t\!\theta$ are column vectors of 1-forms $\gamma^a:=\ve^{ab}\gamma_b$ and $ \theta^a:=\ve^{ab}\theta_b$, respectively. From now on, $\ve_{ab}$ will be used to raise and lower the indices $a$, $b$, $c$.

\subsubsection{ The first order structure equations}\label{sec:first-order-struct}
Having carried out five coframe adaptations, the first structure equations for causal structures can now be expressed as follows.

The pseudo-connection form for the first order structure group is
\begin{gather} \label{eq:MC-Forms-first-order-str-eqns}
 \begin{pmatrix}
2\phi_0 & 0 & 0 & 0 \\
\gamma^{a} & \phi^a{}_b+\left(\phi_0+\phi_{n}\right)\delta^a{}_b & 0 & 0 \\
0 & \gamma_b & 2\phi_{\n} & 0\\
\pi_a & \ve_{ab} \pi_{n} & 0 & -\phi^a{}_b+ (\phi_0-\phi_{\n} )\delta^a{}_b
\end{pmatrix}.
\end{gather}

Consequently, the structure equations can be written as
\begin{subequations}\label{eq:1st-order-str-eqns}
 \begin{gather}
\exd\omega^0 = -2\phi_0 \w\omega^0-{} \theta_a\w\omega^a,\label{1st-causal-I}\\
\exd\omega^a = -\gamma^a \w\omega^0- \phi^a{}_b \w\omega^b- (\phi_0+\phi_{\n} ) \w\omega^a-\theta^a\w\omega^{\n}\nonumber\\
\hphantom{\exd\omega^a =}{} -K^{a}{}_b \theta^{b}\w\omega^0-F^{a}{}_{bc} \theta^{b}\w\omega^c,\label{1st-causal-II}\\
\exd\omega^{\n} = -\gamma_a\w\omega^a-2\phi_{\n}\w\omega^{\n}- L^a \theta_a\w\omega^0,\label{1st-causal-III}\\
\exd\theta_a = -\pi_a\w\omega^0-\pi_{\n}\w\omega_a+\phi^b{}_a \w\theta_b-\left(\phi_0-\phi_{\n}\right)\w\theta_a\nonumber\\
\hphantom{\exd\theta_a =}{} - f_{abc} \theta^{b}\w\omega^c+ W_{anbn} \omega^b\w\omega^{\n}+\half W_{anbc} \omega^b\w\omega^c,\label{1st-causal-IV}
 \end{gather}
\end{subequations}
with the following symmetries
\begin{gather}
F_{abc}=F_{(abc)},\qquad F^{a}{}_{ab}=0,\qquad K^{ab}=K^{ba},\qquad K^a{}_a=0,\qquad f_{abc}=F_{abc;n},\nonumber\\
 W_{anbn}=W_{bnan},\qquad W^a{}_{nan}=0,\qquad \quad W_{anbc}=-W_{0acb},\qquad W_{[a|n|bc]}=0.\label{eq:symmetris-of-1st-order-StrEqns}
\end{gather}

The infinitesimal group action on the torsion coefficients is equivalent to saying that the following 1-forms are semi-basic with respect to the fibration $\pi\colon \cC\ra M$:
\begin{subequations}\label{eq:infinit-gp-action-1st-order-torsions}
\begin{gather}
\D F_{abc} :=\exd F_{abc} -F_{abd}\phi^d{}_c-F_{adc}\phi^d{}_b -F_{dbc}\phi^d{}_a-F_{abc}(\phi_0-\phi_\n),\label{infin-I}\\
\D K_{ab}:=\exd K_{ab}-K_{ac}\phi^{c}{}_b-K_{bc}\phi^{c}{}_a-2K_{ab}(\phi_0-\phi_\n)-2F_{abc}\gamma^c,\label{infin-II}\\
\D L_a :=\exd L_a-L_b\,\phi^b{}_a-3L^a (\phi_0-\phi_\n) +K^{ab} \gamma_b,\label{infin-III}\\
\D W_{anbn} :=\exd W_{anbn}-W_{ancn} \phi^c{}_{b}-W_{bncn} \phi^c{}_{a}-4W_{anbn} \phi_\n,\label{infin-V}\\
\D W_{anbc} :=\exd W_{anbc}-W_{dnbc} \phi^d{}_a -W_{andc} \phi^d{}_b-W_{anbd} \phi^d{}_c\nonumber\\
\hphantom{\D W_{anbc} :=}{} - W_{anbc} (\phi_0-3\phi_n)+( W_{anbn} \ve_{cd}+ W_{ancn} \ve_{bd})\gamma^d.\label{infin-VI}
\end{gather}
\end{subequations}

By the infinitesimal actions \eqref{eq:infinit-gp-action-1st-order-torsions}, it follows that no more reduction of the structure group is possible, unless certain non-vanishing conditions on the torsion coefficients are assumed. Such special cases are considered in subsequent sections.

\subsubsection[An $\{e\}$-structure]{An $\boldsymbol{\{e\}}$-structure}\label{sec:an-e-structure-1}
Using the first order structure equations, an $\{e\}$-structure can be associated to causal structures. The necessary computations are carried out in Appendix~\ref{cha:append-full-struct}.

\begin{thm}\label{thm:e-structure} To a causal structure $\big(M^{n+1},\cC^{2n}\big)$, $n\geq 4$, one can associate an $\{e\}$-structure on the principal bundle $P_{1,2}\ra\cP\ra\cC$ of dimension $\frac{(n+3)(n+2)}{2}$ where \smash{$P_{1,2}\subset \mathrm{O}(p+2,q+2)$} is the parabolic subgroup defined in \eqref{eq:P_12}. If $n=3$, then $P_{1,2}$ is replaced by $P_{1,2,3}$. The essential invariants of a causal structure are the Fubini cubic form of the fibers and the \wsf curvature. The vanishing of the essential invariants implies that the $\{e\}$-structure coincides with the Maurer--Cartan forms of~$\mathfrak{o}(p+2,q+2)$. The structure equations can be written as
 \begin{gather} \label{eq:e-str-pseudo-conn-form}
 \exd\phi+\phi\w\phi=
 \begin{pmatrix}
 -\Phi_0-\Phi_n & -\Pi_n & -\Pi_b & -\Pi_0 & 0 \\
 \Omega^n & \Phi_n-\Phi_0 & \Gamma_b & 0 & -\Pi_0 \\
 \Omega^a & \Theta^a & \Phi^a{}_b & \Gamma^a &\Pi^a \\
 \Omega^0 & 0 & \Theta_b & \Phi_0-\Phi_n & -\Pi_n\\
 0 & \Omega^0 & -\Omega_b & \Omega^n & \Phi_0+\Phi_n
 \end{pmatrix},
 \end{gather}
 where $\phi$ is defined in \eqref{eq:FlatModelCausal} and the right hand side of \eqref{eq:e-str-pseudo-conn-form} is expressed in~\eqref{eq:full-stru-equns-causal}.
\end{thm}
The proof of the above theorem will be continued in Appendix~\ref{cha:append-full-struct} and involves lots of tedious calculations. Of course, it is desirable to find the right change of basis so that the $\{e\}$-structure defines a Cartan connection. This issue is discussed in Section \ref{sec:norm-cart-conn-1}. It is noted the number of torsion elements on the right hand side of the first order structure equations is minimal due to the special absorption process that was undertaken here. When obtaining the $\{e\}$-structure in Appendix~\ref{cha:append-full-struct}, to simplify computations, one term will be added to equation~\eqref{1st-causal-IV} (see Section~\ref{sec:an-e-structure}).

An immediate corollary of Theorem \ref{thm:e-structure} is the following.
\begin{cor}
 The maximal dimension of the symmetry algebra of causal structures on a~\smash{$(n+1)$}-dimensional manifolds is $\frac{(n+3)(n+2)}{2}$ which occurs only in the flat model.
\end{cor}
\begin{proof}
 As explained in Section \ref{sec:flat-model-2} the flat model has the maximal symmetries. Now, suppose there is non-flat model that is maximally symmetric. Then, either the Fubini form or the \wsf have a non-zero entry. However, by the infinitesimal group actions \eqref{eq:infinit-gp-action-1st-order-torsions}, the structure group acts by scaling on both of the essential invariants. As a result, the symmetry group of the causal structure has to be a proper sub-group of $\mathrm{O}(p+1,q+1)$ obtained after normalizing the value of that non-zero entry of the essential invariants. Thus, it cannot be maximally symmetric.
\end{proof}

The corollary above is clear from the parabolic point of view, since any parabolic geo\-met\-ry~$(G,P)$ has maximal symmetry dimension equal to $\dim G$ which can only be reached in the flat case~\cite{CS-Parabolic}.

\subsubsection{The case of conformal structures}\label{sec:conformal-structures}
Using the first order structure equations, an observation is made regarding (pseudo-)conformal structures viewed as causal structures.

\begin{prop}\label{prop:causal-vs-conformal}
 A causal structure $(M,\cC)$ of signature $(p+1,q+1)$ with vanishing Fubini cubic form induces a pseudo-conformal structure of the same signature on $M$. The resulting pseudo-conformal structure is flat if the initial causal structure has vanishing \wsf curvature, and hence flat. Conversely, any pseudo-conformal structure defines a causal structure of the same signature on its bundle of null cones. The resulting causal structure is flat if the initial pseudo-conformal structure is flat.
\end{prop}
\begin{proof}
To show the first part, note that if the Fubini form vanishes then, by the Bianchi identities \eqref{eq:causal-bianchi-K}, the quantities $K_{ab}$ and $L_a$ vanish as well. The first order structure equations~\eqref{eq:MC-Forms-first-order-str-eqns} can be used to show that the Lie derivatives of the quadratic form~$g$, defined in Section~\ref{sec:conf-class-quadr-1}, along vertical vector fields $\frac{\partial}{\partial\theta_1},\dots,\frac{\partial}{\partial\theta_{n-1}}$ are given by
\begin{gather} \label{eq:Lie-deriv-of-g-conformal}
 \cL_{\frac{\partial}{\partial \theta_a}}g= \lambda g,
\end{gather}
for some non-vanishing function $\lambda$ on $\cC$. The vector fields $\frac{\partial}{\partial \omega^i}$, and $\frac{\partial}{\partial \theta_a}$ are defined as dual to the coframe $\omega^0,\dots,\theta_{n-1}$. The vertical vector fields are integrable with the fibers $\cC_x$ as their integral manifolds, and $M$ as the quotient space. Consequently, the conformal class of the quadratic form $g$ descends to $M$ and because it is non-degenerate and of maximal rank on~$M$, a~pseudo-conformal structure on~$M$ is obtained with the same signature as that of $g$. If the causal structure is flat, i.e., the $\{e\}$-structure is given by the Maurer--Cartan forms of $\mathrm{O}(p+2,q+2)$, then the structure bundle fibers over~$M$ as well, implying that the resulting pseudo-conformal structure has to be flat too.

The converse part is shown similarly. Given a pseudo-conformal structure of signature \smash{$(p+1,q+1)$}, choose a coframe $(\omega^0,\dots,\omega^n)$ so that a representative of the pseudo-conformal structure takes the form $g=2\omega^0\circ\omega^n-\ve_{ab}\omega^a\circ\omega^b$. On its null cone bundle $\pi\colon \cC\ra M$, the metric~$g$ pulls back to define a quadratic form on $\cC$ satisfying~\eqref{eq:Lie-deriv-of-g-conformal}. The 1-forms $\omega^i$ define a set of semi-basic 1-forms on $\cC$. Using conformal transformations, $\omega^i$'s can be chosen to become adapted to the flag \eqref{Conframe-adaptation-00}. This can be seen as follows. The affine tangent space $\wh T_y\cC_x\subset T_xM$ is a null hyperplane and by a conformal transformation $\omega^0$ can be taken to be its annihilator. Moreover, the dual of~$\omega^0$ with respect to the metric~$g$ lies on~$\hat y$. More explicitly, let $ y\in\hat y\subset \wh\cC$. Because, $g_{ij}y^iy^j=0$, the 1-form $\omega^0$ as the annihilator of $\wh T_y\cC_x$ can be expressed as $\omega^0=g_{ij}y^j\exd x^i$. Now, it is apparent that the dual of~$\omega^0$ lies in $\hat y$. Since the metric is expressed as $g=2\omega^0\circ\omega^n-\ve_{ab}\omega^a\circ\omega^b$, it follows that $\omega^a(\hat y)=0$, and $\omega^n(\hat y)\neq 0$. Hence, $\omega^i$'s are adapted to the flag~\eqref{Conframe-adaptation-00}.

 Since $g$ satisfies \eqref{eq:Lie-deriv-of-g-conformal}, the resulting causal structure have vanishing Fubini form. If the conformal structure is flat, i.e., isomorphic to the flat model, then as discussed in Section \ref{sec:flat-model-2}, the corresponding causal structure is flat as well.
Because the construction is reversible, the result follows from the first part of the proposition.
 \end{proof}
The following corollary follows immediately.
\begin{cor}
 If a pseudo-conformal structure lifts to a causal structure with vanishing \wsf curvature, then the resulting causal structure is flat.
\end{cor}
\begin{proof}
 By Proposition \ref{prop:causal-vs-conformal}, a pseudo-conformal structure defines a causal structure with va\-nishing Fubini cubic form. If the resulting causal structure has vanishing \wsf curvature, then all of its essential invariants are zero and is therefore flat. By Proposition \ref{prop:causal-vs-conformal}, one obtains that the initial pseudo-conformal structure has be flat as well.
\end{proof}
\begin{rmk}
 As was mentioned previously, the discussion above goes through in the case of conformal structures of positive definite signature. The extra step would be to complexify the inner product and the tangent bundle and work with complex causal structures endowed with a reality condition. See \cite{LebrunNullGeod} for a treatment of complex conformal structures.
\end{rmk}

\subsection{Causal structures as parabolic geometries}\label{sec:norm-cart-conn}
 This section provides the necessary definitions and theorems needed to characterize causal structures as parabolic geometries. A main reference on parabolic geometries is \cite{CS-Parabolic}, but, for the purposes of this section the articles \cite{ANNMonge,KT-gap,Yamaguchi-SimpleDiffSystems,Zelenko-Tanaka} suffice.

The main step is the characterization of causal structures by a bracket generating distribution with certain properties, described in Section~\ref{sec:char-caus-struct}. As a result, causal structures on an \smash{$(n+1)$}-dimensional manifold can be viewed as regular normal infinitesimal flag structures of type $(B_{n-1},P_{1,2})$ or $(D_n,P_{1,2})$ if $n\geq 4$, or $(D_3,P_{1,2,3})$ if $n=3$. By a result of Tanaka--Morimoto--\v{C}ap--Schichl, such structures can be canonically identified as regular normal parabolic geometries of the same type.

 \subsubsection{Parabolic geometries and Tanaka theory} \label{sec:parabolic-geometries}
A parabolic geometry is a Cartan geometry $(\cP,\omega)$ of type $(G,P)$ where $G$ is a real or complex semi-simple Lie group and $P\subset G$ is a parabolic subgroup. Recall that a Cartan geometry $(\cP,\omega)$ of type $(G,H)$ is a principle bundle $H\longrightarrow \cP\longrightarrow M$, endowed with a Cartan connection $\omega$.
The 1-form $\omega$ takes value in $\fg$, the Lie algebra of $G$, and satisfies the following conditions:
 \begin{enumerate}\itemsep=0pt
 \item[1)] $\omega_u\colon T_u\cP\lra\fg$ is an isomorphism of vector space for all $u\in\cP$,
 \item[2)] $V^\dag\im\omega=V$ for all $V\in\fh$ where $V^\dag$ is the fundamental field of $V$,
 \item[3)] $(R_h)^*\omega=\mathrm{Ad}(h^{-1})\omega$ for all $h\in H$.
 \end{enumerate}

A canonical way of identifying a geometric structure as a parabolic geometry is by characterizing it in terms of certain \emph{bracket generating distributions} possibly equipped with a tensor field, e.g., conformal structures or semi-Riemannian structures. Recall the following definitions. To any distribution $D\subset T N$, one can associate a Pfaffian system
\begin{gather*}D^\perp:=\{\omega\in T^*\cC\,|\, \omega(u)=0\ \forall\, u\in D\}.\end{gather*}
Given a distribution $D\subset T\cC$, its space of sections is denoted by $\Gamma(D)$ and its \emph{derived distribution} is defined by $ D^2=D+[D,D]$. Here, $D+[D,D]$ means a distribution whose sections are obtained from $ \Gamma(D)+[ \Gamma(D), \Gamma(D)]$. Define $D^1=D$. The\emph{ $k$-th weak derived distribution} of $D$ is defined by
\begin{gather*} 
 D^{k+1}=D^k+\big[D^k,D\big].
\end{gather*}
One obtains that
\begin{gather*} 
\big( D^{k+1}\big)^\perp=\big\{\omega\in \big(D^k\big)^\perp\, \big|\, \exd\omega\equiv 0 \ \operatorname{mod} \ \big(D^k\big)^\perp\big\}.
\end{gather*}
A distribution $D\subset TN$ is called bracket generating if there exists $k>1$ such that $D^k=TN$. The small growth vector of a bracket generating distribution $D$ is $(d_1,\dots,d_k)$ where $d_k=\dim \big(D^k\big)$. From now on, by restricting to a sufficiently small open set, it is assumed that the growth vector of the distribution is constant.

On a manifold $M$, a bracket generating distribution, $D$, of depth $\mu$ induces a filtration of~$T_xM$ for all $x\in M$, by
\begin{gather*}D(x)=:D^1(x)\subset D^2(x)\subset \cdots\subset D^\mu(x)=T_xM. \end{gather*}
Let $\fm(x)$ denote the grading of $T_xM$ that corresponds to the filtration above, i.e.,
\begin{gather*}\fm(x)=\bigoplus_{i=-1}^{-\mu} \fg_{i}(x),\qquad \fg_i(x):=D^{-i}(x)/D^{-i-1}(x).\end{gather*}
At each $x\in M$, the Lie bracket induces a bracket on $\fm(x)$ called the Levi bracket which is tensorial and turns $\fm(x)$ into a graded nilpotent Lie algebra (GNLA). The Lie algebra $\fm(x)$ is called the \emph{symbol algebra} of $D(x)$. A distribution $D$ of \emph{constant type} has the property that $\fm(x)\cong\fm$ for some GNLA $\fm$ for all $x\in M$.

An interesting class of GNLA arises from pairs $(G,P)$ where $G$ is a real or complex semi-simple group and $P\subset G$ is a parabolic subgroup in the following way. A $|\mu|$-grading for the Lie algebra of $G$, denoted by $\fg$, is a vector space decomposition
\begin{gather*}\fg=\fg_{-\mu}\oplus\cdots\oplus\fg_{\mu},\end{gather*}
satisfying (i)~$[\fg_i,\fg_j]\subset\fg_{i+j}$; (ii)~$[\fg_i,\fg_{-1}]=\fg_{i-1}$ for all $i\leq -1$, i.e., the negative part of the grading,
\begin{gather*}\fg_-:=\fg_{-\mu}\oplus\cdots\oplus\fg_{-1},\end{gather*}
is generated by $\fg_{-1};$ and (iii)~$\fg_{\pm\mu}\neq 0$.

The subalgebra $\fg_-$ is GNLA and
\begin{gather*}\fp:=\fg_0\oplus\fg_1\oplus\cdots\oplus\fg_\mu\end{gather*}
is always a parabolic subalgebra of $\fg$. Its corresponding subgroup, $P\subset G$, is a subgroup that lies between the normalizer $N_G(\fp)$ and its identity component. The subgroup $G_0\subset P$ defined as
\begin{gather*}G_0=\{g\in P\,|\, \mathrm{Ad}_g(\fg_i)\subset\fg_i\ \mathrm{for\ all\ }i=-\mu,\dots,\mu\}\end{gather*}
preserves the grading of $\fg$.

For such a choice of $(G,P)$, there is a canonical construction of a bracket generating distribution whose symbol algebra is $\fg_-$. Let $N_\fg$ be the simply-connected Lie group with Lie algebra~$\fg_-$. The left-invariant vector fields obtained from~$\fg_{-1}$ induces a distribution of constant type on~$N_\fg$ whose symbol algebra is isomorphic to~$\fg_-$.

For a semi-simple Lie algebra $\fg$ every derivation is inner. Therefore given a $|\mu|$-grading of $\fg$ it admits a unique \emph{grading element}, i.e., $e\in\fg_0$ such that $[e,g]=jg$ for all $g\in\fg_j$ and $-\mu\leq j\leq \mu$.

Consequently, the space of $q$-forms on $\fg_-$ with values in $\fg$, denoted by $\Lambda^{q}(\fg_-,\fg)$,
 admits a decomposition $\oplus_{p}\Lambda^{p,q}(\fg_-,\fg)$ where $\Lambda^{p,q}(\fg_-,\fg)$ is the set of $q$-forms $\omega$, with weight $p$, i.e., $\cL_e\omega=p\omega$.\footnote{In \cite{Yamaguchi-SimpleDiffSystems}, the space $\Lambda^{p,q}$ is referred to as $C^{p-q+1,q}$. } The decomposition can be written as
\begin{gather*}\Lambda^{p,q}(\fg_-,\fg)=\bigoplus_{k}\fg_{k+p}\otimes\wedge^q_{k}(\fg_-^*),\qquad\text{where}\quad \wedge^q_k(\fg_-^*):=\sum_{\begin{subarray}{l} r_1+\cdots+r_q=k\\ r_1,...,r_q<0 \end{subarray}}\fg^*_{r_1}\w\cdots\w\fg^*_{r_q}.\end{gather*}

According to a theorem of Tanaka \cite{Yamaguchi-SimpleDiffSystems}, if $H^{1,p}(\fg_-,\fg)=0$, for all $p\geq 0$, then via a process called the \emph{Tanaka prolongation} the Lie algebra of infinitesimal symmetries of a distribution with symbol algebra $\fg_-$ is isomorphic to $\fg$, i.e., $\mathfrak{inf}(D)=\fg$.

The above theorem can be applied to \emph{regular infinitesimal flag structures} of type $(G,P)$, which is defined as follows. Consider a manifold~$M$ endowed with a bracket generating distribution of constant type whose symbol algebra is~$\fm$. The graded frame bundle is a right principal bundle $\operatorname{Aut}_{\rm gr}(\fm)\ra\cF_{\rm gr}(M)\ra M$, where $\operatorname{Aut}_{\rm gr}(\fm)$ is the group of automorphisms of~$\fm$ that preserve the grading.\footnote{In case $M$ is endowed with additional structure, e.g., conformal structures or semi-Riemannian structures, this principal bundle can be reduced further.} If the symbol algebra of the distribution is isomorphic to~$\fg_-$ and $\operatorname{Aut}_{\rm gr}(\fm)$ is isomorphic\footnote{If there exists additional structure on the distribution then $\operatorname{Aut}_{\rm gr}(\fm)$ reduces to~$G_0$.} to $G_0\subset\operatorname{Aut}_{\rm gr}(\fg_-)$, then one says that the distribution induces a~regular infinitesimal flag structure of type $(G,P)$.

By a result of Tanaka \cite{Tanaka-79}, Morimoto \cite{Morimoto-Filtered} and \v{C}ap--Schichl \cite{CS-Cartan}, there is an equivalence of categories between \emph{regular, normal} parabolic geometries and regular infinitesimal flag structures via Tanaka prolongation. Regularity and normality of a parabolic geometry $(G,P)$ is defined in terms of its curvature function of its Cartan connection $(\cP,\omega)$ as follows. At each point $s\in\cP$, the curvature $\Omega=\exd\omega+[\omega,\omega]$ of the Cartan connection can be used to define the curvature function $K\colon \cP\rightarrow \wedge^2\fg_-^*\otimes\fg $ as
\begin{gather*}K(s)(X,Y)=\Omega\big(\omega^{-1}(X)(s),\omega^{-1}(Y)(s)\big),\qquad X,Y\in\fg_-,\end{gather*}
where the space $\fg_-^*$ can be identified with $\fg_+:=\oplus_{i>0}\fg_i$ via the Killing form.
 The Spencer cohomology groups $H^{p,q}\subset \fg\otimes\wedge^q(\fg_-^*)$, are defined
using the Kostant operators $\partial\colon \Lambda^{p,q}\rightarrow \Lambda^{p-1,q+1}$, and $\partial^*\colon \Lambda^{p,q}\rightarrow \Lambda^{p+1,q-1}$, given by
\begin{gather*}
 \partial c(X_1\w\cdots\w X_{q+1})=\sum_i(-1)^{i+1}[X_i,c(X_1\w\cdots\w\hat{X}_i\w\cdots\w X_{q+1})]\\
 \hphantom{\partial c(X_1\w\cdots\w X_{q+1})=}{} +\sum_{i<j} (-1)^{i+j}c([X_i,X_j]\w X_1\w\cdots\w \hat{X}_i\w\cdots\w\hat{X}_j\w\cdots\w X_{q+1}),\\
\partial^* c(X_1\w\cdots\w X_{q-1}) =\sum_i[e^*_i,c(e_i\w X_1\w\cdots\w X_{q-1})]\\
\hphantom{\partial^* c(X_1\w\cdots\w X_{q-1}) =}{} +\frac{1}{2}\sum_{i,j} (-1)^{j+1}c([e^*_i,X_j]_-\w e_i\w X_1\w\cdots\w\hat{X}_j\w\cdots\w X_{q-1}),
\end{gather*}
where $c\in \Lambda^{p,q}(\fg)$ and $[e^*_i,X_j]_-$ denotes the $\fg_-$-component of $[e^*_i,X_j]$, with respect to the decomposition $\fg=\fg_-\oplus\fp$. The operators above result in the orthogonal decomposition
\begin{gather*}
 \Lambda^{p,q}=H^{p,q}+\square \Lambda^{p,q},
\end{gather*}
where the operator $\square:=\partial^*\partial+\partial\partial^*$ is referred to as the Kostant Laplacian and $\square H^{p,q}=0$.

The $\Lambda^{p,2}$-component of the curvature function $K$ is denoted by $K^{(p)}$. Consequently, $K$ can be decomposed as
\begin{gather} \label{eq:curvature-homogeneity}
 K=\sum_{i=-\mu+2}^{3\mu}K^{(i)},
\end{gather}
where $K^{(p)}\colon \fg_i\times\fg_j\rightarrow \fg_{i+j+p}$ is the component of $K$ with homogeneity $p$. Additionally, one can decompose the curvature function as
 \begin{gather*}K=\sum _{i=-\mu}^\mu K_p,\end{gather*}
where $K_p$ takes value in $\fg_p$. The $\fg_-$-component of $K$, $\sum\limits_{i=-\mu}^{-1} K_i$, is denoted by $K_-$. The components of the curvature lying in $H^{p,2}$ are called \emph{harmonic curvature} which constitute the set of essential invariants.

 A Cartan connection is called \textit{torsion-free} if $K_-=0$, \textit{normal} if $\partial^*K=0$, and \textit{regular} if it is normal and $K^{(i)}=0$, for all $i\leq 0$.

\subsubsection{Characterizing causal structures by a bracket generating distribution}\label{sec:char-caus-struct}
The fact that $[\ve^{ab}]$ is non-degenerate allows one to give a characterization of causal structures in terms of a bracket generating distribution with a certain small growth vector obtained as follows.

Define the distributions $\Delta$ and $D$ such that $\Delta^\perp=\{\omega^0,\dots,\omega^{n}\}$ and $D^\perp=\{\omega^0,\dots,\omega^{n-1}\}$. The rank of $D$ is $n$ and $\Delta$ is a completely integrable sub-distribution of $D$ of corank $1$ whose integral manifolds are given by the fibers $\cC_x$. Let $\bfv$ be a characteristic vector field. It follows that $\bfv\in D$ and $\bfv\notin \Delta$.

Knowing $\exd\omega^a\equiv -\ve^{ab}\theta_b\w\omega^n$, modulo $I\{\omega^0,\dots,\omega^{n-1}\}$ and that $\Delta$ is completely integrable implies that $( D^2)^\perp=\{\omega^0\}$. Similarly, the equation $\exd\omega^0\equiv-\theta_a\w\omega^a$, modulo $I\{\omega^0\}$, implies $(D^3)^\perp=\{0\}$, i.e., $D^3=T\cC$. As a result, the small growth vector of $D$ is $(n,2n-1,2n)$.

Let $\{f_1,\dots,f_{n-1}\}$ be a dual basis for $\Delta$ and $e_n$ be a basis for $D/\Delta$. Since the growth vector is $(n,2n-1,2n)$, it follows that $\rk(D^2/D)=n-1$. Because $\Delta$ is Frobenius, one obtains that the set of vectors $e_a=[e_n,f_a]$ are linearly independent and span $D^2/D$. This relation among~$e_a$,~$e_n$ and $f_a$ is equivalent to the structure equation $\exd\omega^a\equiv-\ve^{ab}\theta_b\w\omega^n$ modulo $I\{\omega^0,\dots,\omega^{n-1}\}$.

Furthermore, the equation $\exd\omega^0\equiv -\theta_a\w\omega^a$ modulo $I\{\omega^0\}$, implies that there is a non-degenerate pairing between $e_a$'s and $f_a$'s, i.e., $[e_\bfa,f_\bfa]=e_0$ (no summation involved) for all $1\leq \bfa\leq n-1$, where $e_0$ spans $D^3/D^2$.

One can investigate the action induced by the structure group $G_5$ in \eqref{eq:G5-StrGroup} on the \emph{symbol algebra} of $D$ as follows.
 Each tangent space $T_x\cC$ admits a filtration $D(x)\subset D^2(x)\subset D^3(x)$ which results in the grading $\fm(x)=\fg_{-1}(x)\oplus \fg_{-2}(x)\oplus\fg_{-3}(x)$ where $\fg_i(x)=D^{-i}/D^{-i-1}$ with $D^0=\{0\}$. Using the Lie bracket, $\fm(x)$ becomes a GNLA which would be the symbol algebra of~$D$. Since $\fg_{-1}$ is given by the kernel of $\cI_1:=I\{\omega^0,\dots,\omega^{n-1}\}$, the action of the structure group~$G_5$ on $\Itot=I\{\omega^0,\dots,\omega^n,\dots,\theta^{n-1}\}$, modulo $\cI_1$, is given by scaling on $\omega^n$ and by conformal transformations on $\theta^a$, i.e.,
\begin{gather*}\omega^n\ra \bfb^2\omega^n,\qquad \theta^a\ra \frac{\bfa}{\bfb}\bfA^a{}_b\theta^b.\end{gather*}
This action is isomorphic to the action of $G_0$ as a subgroup of the parabolic subgroup $P_{1,2}$, if $n\geq 4$ or $P_{1,2,3}$ if $n=3$, of $\mathrm{O}(p+2,q+2)$ where $p+q=n-1$. In other words, in causal structures the group of automorphisms of the coframe bundle which preserves the grading of~$T_x\cC$ is reduced to $G_0$.

\subsubsection{Normal Cartan connection}\label{sec:norm-cart-conn-1}

Using the flat model in Section~\ref{sec:flat-model-2}, it follows that the symbol algebra arising from the bracket generating distribution $D$, defined in Section \ref{sec:char-caus-struct} as $\{\omega^0,\dots,\omega^{n-1}\}^\perp$, is isomorphic to the negative part of the $|3|$-grading for $\fg=\fo(p+2,q+2)$, i.e.,
\begin{gather*}\fm\cong\fg_-=\fg_{-3}\oplus\fg_{-2}\oplus\fg_{-1}.\end{gather*}
The parabolic subalgebra that corresponds to this $|3|$-grading is denoted as $\fp_{1,2}$ when $n\geq 4$ and $\fp_{1,2,3}$ when $n=3$. The corresponding parabolic subgroup is given in Section \ref{sec:flat-model-2}, i.e., the parabolic subgroup of $\mathrm{O}(p+2,q+2)$ that is isomorphic to the stabilizer of the subspace spanned by $\{e_1,e_2\}$ where $\{e_1,\dots,e_{n+3}\}$ is some basis for $\RR^{n+3}$ and $p+q=n-1$. Alternatively, $\fp_{1,2}$ can be described by crossing the first two nodes in the Dynkin diagrams $B_{n-1}$ and $D_n$ when $n\geq 4$ and all three nodes when $n=3$.

Furthermore, the distribution $D$ induces a regular infinitesimal flag structure. This is due to the fact that the condition $H^{p,1}(\fg_-,\fg)=0$ is satisfied for all $p\geq 0$.
More directly, by the discussion in Section \ref{sec:char-caus-struct} and the first order structure equations, one obtains that $\operatorname{Aut}_{\rm gr}(\fg_-)$ is reduced to~$G_0$. Thus, by Tanaka prolongation one can view causal structures as regular, normal parabolic geometries of type $(B_{n-1},P_{1,2}) $ or $(D_n,P_{1,2})$ if $n\geq 4$, or $(D_3,P_{1,2,3})$. In Section \ref{sec:harmonic-curvature} it is discussed that when $n\geq 4$, causal structures are canonically isomorphic to regular, normal parabolic geometries of type $(B_{n-1},P_{1,2}) $ or $(D_n,P_{1,2})$. When $n=3$, causal structures are \emph{``subgeometries''} of regular, normal parabolic geometries of type $(D_3,P_{1,2,3})$, i.e., the integrability of the fibers $\cC_x$, implies the vanishing of one of the harmonic curvature components of such parabolic geometries.

Knowing the existence of a normal Cartan connection for causal structures, the normality conditions of Tanaka are straightforward to obtain and can be used to recursively define the correct change of basis which transforms the $\{e\}$-structure to a Cartan connection (see \cite{Omid-Thesis}). Since finding the explicit expressions for the change of basis is very tedious and the result is not particularly illuminating, finding the normal Cartan connection is not carried out here.

\subsection{Essential invariants and the Raychaudhuri equation}\label{sec:essential-invariants}
In this section the essential invariants whose vanishing imply the flatness of the causal structure are introduced. Their geometric interpretation is given. A variational problem is naturally associated to a causal geometry which allows one to define the notion of null Jacobi fields in the causal setting. Finally, the section ends with a discussion on associated Weyl connections of a~causal structure and a derivation of the Raychaudhuri equation for causal structures. In this section there are references to Appendix~\ref{cha:append-full-struct}, particularly to the $\{e\}$-structure obtained in Section~\ref{sec:an-e-structure}.
\subsubsection{The harmonic curvature}\label{sec:harmonic-curvature}

According to the $\{e\}$-structure obtained in Section \ref{sec:an-e-structure}, obstructions to the flatness of a causal structure, as listed in Sections \ref{sec:tors-coeff-2nd-strEq-1} and \ref{sec:bianchi-identities}, are generated by $F_{abc}$ and $W_{anbn}$. For instance, differentiating \eqref{1st-causal-IV} gives
\begin{gather*}
W_{anbc} = \tfrac{2}{3} \big(W_{an[b|n;|\unc]}+F_{a[c}{}^dW_{b]ndn}\big) +\tfrac{4}{3(n+2)}\big(\ve_{a[c|}W^{d}{}_{n|b]n;\und} +\ve_{a[b}F^{de}{}_{c]}W_{dnen}\big).
\end{gather*}
In Section \ref{sec:tors-coeff-2nd-strEq-1}, the quantities $W_{ijkl}$ are expressed as the coframe derivatives of $W_{anbn}$ with respect to $\frac{\partial}{\partial\theta^a}$. As a result, in the case $F_{abc}=0$, the vanishing of~$W_{nabn}$ implies flatness of the causal structure.

For an understanding of $W_{anbn}$, consider a 2-plane in $T_xM$ given by $\lambda\w\mu$ where $\lambda=\lambda^i\bfe_i$, $\mu=\mu^i\bfe_i\in TM$ for some basis $\{\bfe_0,\dots,\bfe_n\}$. The \emph{sectional conformal curvature} of a pseudo-conformal structure evaluated at $\lambda\w\mu$ is defined by
\begin{gather*}W(\lambda\w\mu)=W_{ijkl}\lambda^i\mu^j\lambda^k\mu^l.\end{gather*}
As a result, at the point $(x,[y])\in\cC$, the invariant $W_{nabn}$ can be interpreted as the part of the conformal sectional curvature acting on 2-planes $\xi^a=\frac{\partial}{\partial\omega^n}\w\frac{\partial}{\partial\omega^a}$. According to the construction of the coframes in \eqref{Conframe-adaptation-00}, $\xi$ is a subset of $\widehat T_y\cC_x$ and contains $\hat y$ where $[\hat y]=y\in \cC_x$. It follows that $\xi^a=\mu^a\w y^i\partial_i$, where $\mu^a$ is an element of the \emph{shadow space}
\begin{gather}
 \label{eq:shadow-space}
S_{(x,[y])}:=\widehat T_y\cC_x/\hat y.
\end{gather}
As a result, $W_{anbn}$ can be thought of as the restriction of the bilinear form $W_{ijkl}y^jy^l$ to the shadow space. This restriction is well-defined up to scale. Infinitesimally, this scaling action is realized by the 1-form $\phi_n$ as is apparent from \eqref{infin-V}.

Finally, it can be argued that the Fubini cubic form and the \wsf curvature coincide with the harmonic curvatures of the parabolic geometries that correspond to causal structures. By the discussion in Section \ref{sec:norm-cart-conn-1}, causal structures on an $(n+1)$-dimensional manifold correspond to parabolic geometries of type $(B_{n-1},P_{1,2})$ and $(D_n,P_{1,2})$ if $n\geq 4$, and $(D_3,P_{1,2,3})$. As discussed in Section \ref{sec:parabolic-geometries}, a grading is induced on the cohomology classes of parabolic geometries. Moreover, the set of essential invariants for a parabolic geometry is given by the harmonic curvature. It turns out that for parabolic geometries of type $(B_{n-1},P_{1,2})$ and $(D_n,P_{1,2})$ if $n\geq 4$, the infinitesimal flag structure is the same as causal structures and the harmonic curvature has two irreducible components. Using the grading \eqref{eq:curvature-homogeneity} of the curvature, one component has homogeneity one whose $\fg_0$-irreducible representation is given as a totally symmetric and trace-free cubic tensor. The other irreducible component has homogeneity two whose $\fg_0$-irreducible representation is given as a symmetric and trace-free bilinear form (see \cite{ANNMonge,KT-gap,Yamaguchi-SimpleDiffSystems}). The structure equations of causal geometries imply that the essential invariant $F_{abc}$ corresponds to a cohomology class of homogeneity one and the \wsf curvature corresponds to a cohomology class of homogeneity two. Also, from the infinitesimal group actions~\eqref{infin-I} and~\eqref{infin-V} one obtains that the Fubini cubic form and the \wsf curvature are $\fg_0$-irreducible. Moreover, one can impose the Tanaka normality conditions on the pseudo-connection form of the $\{e\}$-structure to derive a normal Cartan connection (see~\cite{Omid-Thesis}). It turns out that the necessary change of coframe does not affect the Fubini cubic form and the \wsf curvature, i.e., the essential invariants of the $\{e\}$-structure coincide with the harmonic curvature.

The situation is different in four dimensions. Four dimensional causal structures are a special class of parabolic geometries of type $(D_3,P_{1,2,3})$. The harmonic curvature of these parabolic geometries is five dimensional, three of which have homogeneity one and two have homogeneity two. Consider a four dimensional causal structure of split signature. The Fubini cubic form of projective surfaces and the \wsf curvature are comprised of two entries.It turns out that all these four entries are irreducible under the action of $\fg_0$. As a result, four dimensional causal structures are special cases of parabolic geometries of type $(D_3,P_{1,2,3})$ in which one of the harmonic curvatures of homogeneity one is zero. The vanishing of this piece of harmonic curvature amounts to the fibration $\pi\colon \cC\rightarrow M$ with 2-dimensional fibers.

\subsubsection{The \ssf curvature}\label{sec:density-bundle}
The action of the structure group $G_5$ defined in \eqref{eq:G5-StrGroup} on the bilinear form \eqref{eq:HorizQuadraticFrom} is a scaling action $g\ra \bfa^2\bfb^2g$. Fix a representative $g$ and consider the sub-bundle $\cP_1=\{p\in\cP\,|\, \bfa^2\bfb^2=1\}$, where $\cP$ denotes the structure bundle associated to the derived $\{e\}$-structure (see Section \ref{sec:an-e-structure}). It follows that over $\cP_1$ one has
\begin{gather*} 
 0=\exd\big(\bfa^2\bfb^2\big)=2(\phi_0+\phi_n)+ s_i\omega^i+s_\una\theta^a,
\end{gather*}
for some functions $s_i$ and $s_\una$. Differentiating this equation again, it follows that
\begin{gather*}
 \exd s_i\equiv\pi_i+\delta^0_i\big(s_a\gamma^a+s_a\pi^a\big) +\delta^a_i\big(s_b\phi^b{}_a+s_n\gamma_a +s_\una\pi_n\big) \qquad \operatorname{mod} \ \Itot.
\end{gather*}
 Consequently, $s_i$ can be translated to zero by the action of the structure group. It follows that
\begin{gather} \label{eq:Schouten-01}
 \pi_i=p_{ij}\omega^j+p_{i\una}\theta^a.
\end{gather}
 Substituting \eqref{eq:Schouten-01} and $\phi_n=-\phi_0+s_\una\theta^a$ into the structure equations, equation \eqref{1st-causal-IV} becomes
 \begin{gather*} 
 \exd\theta^a=\phi^a{}_b\w\theta^b-2\phi_0\w\theta^a+ \half R^a{}_{nij}\omega^i\w\omega^j+ E^a{}_{nbi}\theta^b\w\omega^i+ \half C^a{}_{nbc}\theta^b\w\theta^c,
 \end{gather*}
where
\begin{gather}
 R^a{}_{nbn}=W^a{}_{nbn}+p_{nn}\delta^a{}_b,\qquad R^a{}_{nbc}=W^a{}_{nbc}+\delta^a{}_cp_{nb}+\delta^a{}_bp_{nc},\nonumber\\
 R^a{}_{n0b}=-p^a{}_b+\delta^a{}_bp_{n0},\qquad R^a{}_{n0n}=p^a{}_n.\label{eq:shadow-flag-curvature-with-trace-01}
\end{gather}
In conformal geometry, the symmetric part of $p_{ij}$ is referred to as the \emph{Schouten tensor}. The quantities $R^a{}_{nbn}$ and $ p_{nn}$ are called the \emph{shadow flag curvature} and the \emph{scalar shadow flag curvature} (or the \ssf curvature) of the causal structure respectively.

\subsubsection{Null Jacobi fields and the tidal force}\label{sec:jacobi-fields}

As was mentioned in Section \ref{sec:char-curv}, the characteristic curves of causal structures is intrinsically defined by the quasi-contact structure of $\cC$. Nevertheless, one can associate a \emph{variational problem} to causal geometries via the Griffiths formalism \cite{Griffiths-EDS-Book} (see also \cite{HsuCalculus92}). As a result, the notion of null Jacobi fields naturally arises. The corresponding Jacobi equations result in another interpretation of the \wsf curvature. The fact that Jacobi equations of null Jacobi fields involve the \wsf curvature has to do with the second Lie derivative of $g$ in~\eqref{eq:HorizQuadraticFrom}, which is referred to as the \emph{tidal force}~\cite{O'neill-Semi-Riemannian}.

Recall that for the Pfaffian system $\Ichar=\{\omega^0,\dots,\omega^{n-1},\theta^1,\dots,\theta^{n-1}\}$, after fixing an interval $I=(a,b)$, one can define the space of smooth immersions of its integral curves as
\begin{gather*}\widehat\sU(I)=\{\gamma\colon I\ra \cC \,|\, \gamma^*(\Itot)=0\}.\end{gather*}
The space $\widehat\sU(I)$ can be given a Whitney $C^\infty$-topology. After making a choice of orientation for the characteristic curves, the space of the characteristic curves of $\cC$, $\sU(I)$, is defined as $\widehat\sU(I)/\mathrm{Diff}(I)$, i.e., the space obtained from $\widehat\sU(I)$ by identifying the characteristic curves that are reparametrizations of each other.

 Consider the functional $\Phi\colon \sU\ra\RR$, which at a point $[\gamma]\in\sU(I)$ is defined by
\begin{gather*}\Phi([\gamma])=\int_I\gamma^*\omega^0.\end{gather*}
The triple $(\cC,I,\omega^0)$ is referred to as a \emph{variational problem} on $\cC$, and is intrinsic to any causal structure.
Since $\omega^0(\dot{\gamma})=0$ everywhere, the above definition of $\Phi$ is independent of the choice of representative $\gamma\in[\gamma]$.
This is in contrast with Finsler geometry where the geodesics are defined as the extremal curves of the arc-length functional which is always positive and are equipped with a specific parametrization.

Intuitively, the infinitesimal variations of a characteristic curve $\gamma$ can be identified with the tangent space $T_\gamma\sU(I)$ which is the space of smooth variational vector fields along $\gamma$. To give a~more precise definition consider a parametrization for $I=(a,b)$ and a map $\Gamma\colon (-\epsilon,\epsilon)\ra\sU(I)$. As a result of such parametrization, the curves given by $\Gamma(s)$ can be parametrized. Using this parametrization, one can define $\widehat\Gamma\colon (-\epsilon,\epsilon)\times I\ra\cC$, as $\widehat\Gamma(s,t)= (\Gamma(s) )(t)$. The map $\Gamma$ is called a compact variation of the curve $\gamma$ if the corresponding map $\widehat\Gamma$ is smooth, $\widehat\Gamma(0,t)=\gamma(t)$, $\wh\Gamma^*(s_0,t)(\Ichar)=0$, for any fixed $s_0$, and $\wh\Gamma$ is compactly supported in $I$, i.e., for all values of~$s$, the curves $\wh\Gamma(s,t)$ and $\wh\Gamma(0,t)$ coincide outside of a compact subset of~$I$. Then, $T_\gamma\sU(I)$ is expressed as
\begin{gather*}T_\gamma\sU(I):=\left\{\frac{\partial \wh\Gamma}{\partial s}(0,t) \vl \Gamma\colon (-\epsilon,\epsilon)\ra \sU(I) \mathrm{\ is\ a\ compact \ variation\ of\ }\gamma\right\}.\end{gather*}
The \emph{variational equations} of a characteristic curve can be thought of as the first order approximation of $T_\gamma\sU(I)$. A practical derivation of the variational equations can be given in terms of the Pfaffian system $\Ichar$.

Let $\eta\in\Ichar$, and $\Gamma$ a variation of the characteristic curve $\gamma$. By the definition of $\wh\Gamma(s,t)$, one has $\wh\Gamma^*(s_0,t)\eta=0$, and therefore $\cL_{\frac{\partial}{\partial s}}\wh\Gamma^*(s,t)\eta\equiv 0$ modulo $I\{\exd s\}$. Setting $s=0, $ the latter expression implies that
\begin{gather*} \gamma^*\big(\hat J\im\exd\eta+\exd(\hat J\im\eta)\big)=0,\end{gather*}
where $\hat J(t)=\frac{\partial}{\partial s}\widehat\Gamma(0,t)$ is the infinitesimal variations associated to $\Gamma(s)$.

In the expression of the variational equations one can drop the pull-back $\gamma^*$ and write them as
\begin{gather*}\hat J\im\exd\eta+\exd\big(\hat J\im\eta\big)\equiv 0\qquad \operatorname{mod} \ \Ichar.\end{gather*}
The tangential component of the vector field $\hat J(t)$ along $\gamma(t)$ has no effect on the variational equations, hence $\frac{\partial}{\partial\omega^n}$ will be dropped from the expressions. Therefore, $\hat J$ along the curve $\gamma(t)$ can be expressed as
\begin{gather*}\hat J(t)=V^a(t)\frac{\partial}{\partial \omega^a}+V^{\una}(t)\frac{\partial}{\partial\theta^a}.\end{gather*}
The vector field $J=\pi_*\hat J$ is called a \emph{null Jacobi field} along the null geodesics $\pi(\gamma(t))\subset M$.

The variational equations for $\hat J(t)$ can be written as
\begin{subequations}\label{eq:jacobi-eqns01}
\begin{gather}
 \hat J\im \exd\omega^0+\exd \big(\hat J\im\omega^0\big)=0, \label{eq:jacobi-eqns01-I}\\
 \hat J\im \exd\omega^a+\exd \big(\hat J\im\omega^a\big)=0 \qquad\operatorname{mod} \ \Ichar:=I\big\{\omega^0,\dots,\omega^{n-1},\theta_1,\dots,\theta_{n-1}\big\}. \label{eq:jacobi-eqns01-II}\\
 \hat J\im \exd\theta_a+\exd \big(\hat J\im\theta_a\big)=0, \label{eq:jacobi-eqns01-III}
 \end{gather}
\end{subequations}
equation \eqref{eq:jacobi-eqns01-I} gives no new information. The last two equations give
\begin{gather}
\exd V^a+V^b\phi^a{}_b-V^\una\omega^n\equiv 0,\nonumber\\
 \exd V^\una +V^\unb\phi^a{}_b+W^a{}_{nbn}V^b\omega^n\equiv 0 \qquad \operatorname{mod} \ \Ichar.\label{eq:jacobi-eqns02}
\end{gather}
Using the pseudo-connection from Section~\ref{sec:first-order-struct}, the covariant derivative along characteristic curves of vector fields lying within the shadow space \eqref{eq:shadow-space} is given by{\samepage
\begin{gather*}D_\bfv \hat J\equiv \bfv\im\left(\exd V^a\frac{\partial}{\partial\omega^a}+V^b\phi^a{}_b\frac{\partial}{\partial\omega^a}\right) \qquad \operatorname{mod} \ \frac{\partial}{\partial\theta^1},\dots,\frac{\partial}{\partial\theta^{n-1}}, \end{gather*}
where $\bfv$ is the characteristic vector field satisfying $\omega^n(\bfv)=1$.}

Using equations \eqref{eq:jacobi-eqns02}, it follows that
 \begin{gather} \label{eq:null-jacobi-equation-weyl-flag}
 D_\bfv D_\bfv \hat J \equiv \begin{cases} -W^a{}_{nbn}V^b\dfrac{\partial}{\partial\omega^a} & \operatorname{mod} \ \dfrac{\partial}{\partial\theta^1},\dots,\dfrac{\partial}{\partial\theta^{n-1}},\\
 -\overset{\mathrm{sf}}W \hat J & \operatorname{mod} \ \dfrac{\partial}{\partial\theta^1},\dots,\dfrac{\partial}{\partial\theta^{n-1}},
 \end{cases}
 \end{gather}
where $\overset{\mathrm{sf}}W$ is viewed as an automorphism of the shadow space. Recall from Section~\ref{sec:harmonic-curvature} that $W^a{}_{nbn}$ is given by $\omega^a(W(\frac{\partial}{\partial\omega^b},\bfv)\bfv$. Hence, defining
\begin{gather*}J'=\pi_*\big(\D_\bfv\hat J\big),\end{gather*}
the null Jacobi equation can be expressed as
\begin{gather} \label{eq:Jacobi-equation-final-form}
 J''+W(J,\bfv)\bfv=0,
\end{gather}
where $J=\pi_{*}\hat J$ and $W(J,\bfv)\bfv$ denotes $\pi_*\big(W(\hat J,\bfv)\bfv\big)$. For convenience, in Section~\ref{sec:raych-equat} equation~\eqref{eq:Jacobi-equation-final-form} is expressed as
\begin{gather} \label{eq:matrix-form-jacobi-eq}
 J''+\overset{\mathrm{sf}}W J=0.
\end{gather}
In some references \cite{O'neill-Semi-Riemannian} (cf.~\cite{HS2011Causal}) the automorphism of the shadow space $W^a{}_{nba}\colon S_{(x,[y])}\ra S_{(x,[y])}$ is called the trace-free part of the \emph{tidal force} in analogy with the terminology from general relativity.
Equation~\eqref{eq:null-jacobi-equation-weyl-flag} shows that variational problem for null geodesics only depends on~\wsf curvature and not the flag curvature.

\subsubsection{The Raychaudhuri equation}\label{sec:raych-equat}

In this section, a causal analogue of the Raychaudhuri equation along null geodesics is derived. The approach that is taken here is presented in \cite[Chapter~12]{BEE-Global}.

Taking a section of the structure bundle, $s\colon \cC\ra\cP$, define the null Jacobi fields $\hat J_1(t),\dots$, $\hat J_{n-1}(t)$ along the characteristic curve $\gamma(t)\subset \cC$, $t\in[a,b]$ by the initial conditions
\begin{gather*}\hat J_a(t_0)=0,\qquad \mathrm{D}_{\bfv(t_0)}\hat J_a(t_0)=\frac{\partial}{\partial (s^*\omega^a)},\end{gather*}
where $t_0\in[a,b]$ and $\bfv(t)$ is tangent vector to $\gamma(t)$. Define the $(1,1)$ tensor field called the \emph{null Jacobi tensor} as
\begin{gather*}A(t)=[J_1(t),\dots,J_{n-1}(t)],\end{gather*}
which is a $(n-1)\times (n-1)$ matrix whose $b$th column is the null Jacobi field $J_b(t)=\pi_*(\hat J_b(t))$.

The optical invariants along null geodesics are defined as follows.
\begin{df}
 Let $A$ be a null Jacobi tensor along a null geodesic and define $B=A' A^{-1}$, where $A$ is invertible. The \emph{expansion}, $\theta$, \emph{vorticity tensor,} $\omega$ and \emph{shear tensor}, $\sigma$, are defined as
\begin{gather*}\theta=B^a{}_a,\qquad \omega^a{}_b=\half \big(B^a{}_b-(B^*)^a{}_b\big),\qquad \sigma^a{}_b=\half \big(B^a{}_b+ (B^* )^a{}_b\big)-\frac{\theta}{n-1}\delta^a{}_b,\end{gather*}
where $B^*$ denotes the adjoint matrix with respect to the bilinear form $s^*g$.
\end{df}

Following the discussion in Section \ref{sec:jacobi-fields}, after taking a section of the structure bundle one obtains that a null Jacobi field satisfies the equation
\begin{gather*}J''+\overset{\mathrm{sf}}R J=0,\end{gather*}
rather than equation \eqref{eq:matrix-form-jacobi-eq}, where $\overset{\mathrm{sf}}R$ represents the matrix form of the shadow flag curvature $R^a{}_{nbn}=W^a{}_{nbn}+p_{nn}\delta^a{}_b$ defined in~\eqref{eq:shadow-flag-curvature-with-trace-01}.

It follows that \begin{gather*}A''+\overset{\mathrm{sf}}R A=0.\end{gather*}
Moreover, one has
\begin{gather*}B'=\big(A' A^{-1}\big)'=A'' A^{-1}-A'A^{-1}A'A^{-1}=-\overset{\mathrm{sf}}R -B^2.\end{gather*}
Since $B=\omega+\sigma+ \frac{\theta}{n-1}\mathrm{Id}$, one obtains
\begin{gather*}\theta'=\operatorname{tr}(B')=-\operatorname{tr}(\overset{\mathrm{sf}}R)-\operatorname{tr}(B^2)= -\operatorname{tr}(\overset{\mathrm{sf}}R)-\operatorname{tr}\left[\left(\omega+\sigma+ \frac{\theta}{n-1}\mathrm{Id}\right)^2\right].\end{gather*}
Using the fact that $\operatorname{tr}(\omega)=\operatorname{tr}(\sigma)=\operatorname{tr}(\omega\sigma)=0$ and $\operatorname{tr}(\overset{\mathrm{sf}}R)=(n-1)p_{nn}$ one arrives at
\begin{gather*} 
 \theta'=-(n-1)p_{nn}-\operatorname{tr}\big(\omega^2\big) -\operatorname{tr}\big(\sigma^2\big)-\frac{\theta^2}{n-1},
\end{gather*}
which is the causal analogue of the Raychaudhuri equation along null geodesics.

\subsection{Examples}\label{sec:examples}
In this section three examples of causal structures are given. The first one is obtained from the product of two Finsler structures and the second and third ones are given in terms of a defining function.

\subsubsection{Product of Finsler structures}\label{sec:prod-finsl-struct}

A Finsler manifold is denoted by $(N,\Sigma)$, where $N$ is $n$-dimensional and $\Sigma\subset TN$ is of codimension one whose fibers $\Sigma_p\subset T_pN $ are strictly convex hypersurfaces. The fiber bundle $\Sigma$ is often referred to as the indicatrix bundle over which the Finsler metric has unit value. It is known \cite{BryantRemarksFinsler, CartanFinsler} that $\Sigma$ is equipped with unique coframing $\eta=\big(\eta^0,\eta^1,\dots,\eta^{n-1},\zeta^1,\dots,\zeta^{n-1}\big)$ satisfying
\begin{gather}
 \exd \eta^0= -\zeta_a\w\eta^a,\nonumber\\
 \exd\eta^a=\phantom{-}\zeta^a\w\eta^0-\psi^a{}_b\w\eta^b-I^{ab}{}_{c}\zeta_b\w\eta^c,\label{eq:FinslerStrucure}\\
 \exd\zeta_a=\phantom{-}\psi^b{}_a\w\zeta_b+R_{0a0b}\eta^0\w\eta^b+\half R_{0abc}\eta^b\w\eta^c-J_a{}^b{}_c\zeta_b\w\eta^c\nonumber
\end{gather}
for some invariants $I_{abc}$, $J_{abc}$, $R_{0a0b}$ and $R_{0abc}$. The 1-form $\eta^0$ is called the Hilbert form of the Finsler structure which, at the point $v\in\Sigma_p$, is uniquely characterized by the conditions
 \begin{gather*}\eta^0(H_v\Sigma_p)=0,\qquad \eta^0(v)=1,\end{gather*}
where $H_v\Sigma_p\subset T_pN$ is the tangent hyperplane to $\Sigma_p$ at the point $v$. As a result of the structure equation, $\eta^0$ induces a contact structure on $\Sigma$ whose associated Reeb vector field defines the geodesic flow of the Finsler structure on $\Sigma$. Moreover, the quadratic form $\omega^0\circ\omega^0+\delta_{ab}\omega^b\circ\omega^c$ is well-defined for the Finsler structure. It turns out that the quantity $I_{abc}=\delta_{ad}\delta_{be}I^{de}{}_c$ can be used to define the cubic form $I_{abc}\zeta^a\circ\zeta^b\circ\zeta^c$, which, when restricted to a fiber $\Sigma_p$, coincides with the centro-affine cubic form of the hypersurface $\Sigma_{p}\subset T_pN$. The quantities $J_{abc}$ measure the rate of change of $I_{abc}$ along the geodesics, i.e., $J_{abc}=\frac{\partial}{\partial \eta^0}I_{abc}$. The quantity $R_{0a0b}$ is symmetric using which the \textit{flag curvature} of the Finsler structure can be represented by
\begin{gather*}R_{0a0b}\eta^a\circ\eta^b.\end{gather*}

Let $(N,\Sigma)$ and $(\tilde N,\tilde\Sigma)$ be two Finsler manifolds.
 Given a strictly positive function $f$ on their product space, at a point $p\in M=N\times{}_f\tilde N$, one can define a projective hypersurface whose associated cone is given by
 \begin{gather*}
\widehat \cC_p=\big\{(v,\tilde v)\in T_p \big(N\times\tilde N\big)\,|\, \lambda f v\in\Sigma_{\pi(p)},\lambda\tilde v\in\tilde\Sigma_{\tilde\pi(p)} \mathrm{\ for\ some\ } \lambda\in\RR\big\}. \end{gather*}
As a result, $M$ is equipped with a causal structure of signature $(n,\tilde n)$ where $\tilde n=\dim \tilde N$. Note that it is assumed that the Finsler structures are reversible.

If $G$ and $\tilde G$ represent the Finsler metrics on $N$ and $\tilde N$, respectively, then the causal structure on $M$ is locally given as the zero locus of $ G- f\tilde G$.
If $f$ only depends on $N$, then the above causal structure corresponds to the warped product of the Finsler structures on $N$ and $\tilde N$.

More specifically, let $M=N\times{}_f\RR$, where $N$ is a 3-dimensional Finsler manifold. On $N$ there is the unique coframing $\big(\eta^0,\eta^1,\eta^2,\zeta^1,\zeta^2\big)$ satisfying \eqref{eq:FinslerStrucure}. If~$\eta^3$ represents a 1-form on $\RR$, one has
\begin{gather*}\exd f=f_{,0}\eta^0+f_{,1}\eta^1+f_{,2}\eta^2+f_{,3}\eta^3,\end{gather*}
where $f_{,i}=\frac{\partial}{\partial\eta^i}f$ are called the coframe derivatives of $f$. The projective Hilbert form is a~multiple of
\begin{gather*}\omega^0=\sqr \big(\eta^0-f\eta^3\big),\end{gather*}
where, by abuse of notation, the pull-backs are dropped and $\eta^0$ and $\eta^3$ represent $\pi^*\big(\eta^0\big)$ and~$\tilde\pi^*\big(\eta^3\big)$.

A choice of 3-adapted coframing for the causal structure on $M$ is given by
\begin{gather*} 
 \omega^1=\eta^1,\qquad \omega^2=\eta^2,\qquad \omega^3=\sqr\big({-}\eta^0-f\eta^3\big),\\
 \theta_1=\sqr\zeta_1-\sqr f_{,1}\eta^3,\qquad \theta_2=\sqr\zeta_2-\sqr f_{,2}\eta^3,\\
 \gamma_1=-\sqr \zeta_1+\sqr f_{,1}\eta^3,\qquad \gamma_2=-\sqr \zeta_2+\sqr f_{,2}\eta^3,\\
\phi_0=-\phi_3=-\half f_{,0}\eta^3,\qquad \phi=\pi_3=0.
\end{gather*}
The quantities $F^{ab}{}_c$ in the structure equations \eqref{eq:StrEq-2-adapted-A} coincide with $I^{ab}{}_c$ in \eqref{eq:FinslerStrucure}. By absorbing the trace of $F_{abc}$ and $K_{ab}$, a 5-adapted coframing $(\tilde\omega^i,\tilde\theta_a)$ is given by
\begin{gather*}
\tilde\omega^0=\omega^0,\qquad \tilde\omega^a= \omega^a+A^a\omega^0,\qquad \tilde\omega^3= \omega^3+ B\omega^0,\qquad \tilde\theta_a=\theta_a,\end{gather*}
where
\begin{gather*}
A^1=-\half(I_{111}+I_{122}),\qquad A^2=\half(I_{112}+I_{222}),\\
 B=\quar\big(\big(A^1\big)^2+\big(A^2\big)^2\big)-\textstyle{\frac{1}{2\sqrt 2}} \big(A^1{}_{,\underline{1}}+A^2{}_{,\underline{2}}\big)-A^1F_{111}-A^2F_{222},
\end{gather*}
where $F_{111}=\quar (I_{111}-3I_{122} )$, $F_{222}=\quar (I_{222}-3I_{211} )$ are the components of the Fubini cubic form of the causal structure and $A^a{}_{,\underline{b}}$ denotes the coframe derivative $\frac{\partial}{\partial\zeta^b}A^a$. A choice of connection forms for the first order structure equations is given by
\begin{gather*}
 \tilde\gamma^1=\gamma^1+ B_{,1}\omega^0-\big(\sqr B_{,0}+A^1{}_{,1}\big)\tilde\omega^1 -\half\big(A^1{}_{,2}+A^2{}_{;1}\big)\tilde\omega^2 +\sqr A^1{}_{,0}\tilde\omega^3+B\theta_1,\\
 \tilde\gamma^2=\gamma^2+ B_{,2}\omega^0-\half\big(A^1{}_{,2}+A^2{}_{,1}\big)\tilde\omega^1 -\big(\sqr B_{,0}+A^2{}_{,2}\big)\tilde\omega^2 +\sqr A^2{}_{,0}\tilde\omega^3+B\theta_2,\\
\tilde\phi_0=-\half A^1\theta_1-\half A^2\theta_2,\\
\tilde\phi_3=\textstyle{\frac{1}{2\sqrt 2}}
\big({-}B_{,0}\omega^0+A^1{}_{,0}\tilde\omega^1+A^2{}_{,0}\tilde\omega^2\big),\\
\tilde\phi\phantom{{}_{3}}=\half \big(A^1{}_{,2}-A^2{}_{,1}\big)\omega^0+\textstyle{\frac{1}{2\sqrt 2}} A^2{}_{,0}\tilde\omega^1-\textstyle{\frac{1}{2\sqrt 2}}A^1{}_{,0}\tilde\omega^2-\half A^2\theta_1+\half A^1\theta_2,\\
\tilde\pi_3= -\sqr A^1{}_{,0}\theta_1-\sqr A^2{}_{,0}\theta_2-\half\big(R_{0101}+R_{0202}+(\ln f)_{,11}+(\ln f)_{,22}\big)\tilde\omega^3.
\end{gather*}
Consequently, the \wsf curvature is given by the trace-free part of
\begin{gather} \label{eq:Causal-Finsler-Wsf}
 \half R_{0a0b}+\half (\ln f)_{,ab},
\end{gather}
where $(\ln f)_{;ab}:=\frac{\partial}{\partial\eta^b}\frac{\partial}{\partial\eta^a}\ln f$. Moreover, the quantities $K_{ab}$ and $L_a$ are given by
\begin{gather*}K_{12}=K_{21}=-\sqrt 2 A^1{}_{,\underline 2}-F_{112}A^1-F_{122}A^2,\qquad K_{11}=-K_{22}=-\sqr \big(A^1{}_{,\underline 1}-A^2{}_{,\underline 2}\big),\\
L_a=-\sqrt 2 B_{,\underline a}.\end{gather*}
To see the symmetric relation $K_{12}=K_{21}$ one needs to realize $A^1{}_{,\underline 2}=A^2{}_{,\underline 1}$.
It is rather straightforward to generalize this example to obtain a causal structure on $M=N\times{}_f \RR$ where $N$ is an $n$-dimensional Finsler manifold. The corresponding \wsf is given by expression \eqref{eq:Causal-Finsler-Wsf} and the Fubini cubic form is obtained from the centro-affine cubic form of indicatrices of the Finsler structure.
\subsubsection{Cayley's cubic scroll}\label{sec:cayleys-cubic-scroll}
Using the description \eqref{eq:ConeStr-DefGraph}, consider the causal structure on a 4-dimensional manifold given locally as the graph of the function
\begin{gather*}F\big(x^i;y^a\big)=\textstyle\frac{1}{3}c_{111}\big(y^1\big)^3+c_{12}y^1y^2+c_1y^1+c_2y^2+c_0,\end{gather*}
where the coefficients $c_{i\cdots k}$ are functions of $x^i$'s.

A choice of 1-adapted coframing is given by
\begin{gather*}
 \omega^0=\exd x^0-\textstyle\frac{\partial}{\partial y^a}F\exd x^a+\big(y^a\textstyle\frac{\partial}{\partial y^a}F-F\big)\exd x^3,\qquad
 \omega^a=\exd x^a-y^a\exd x^3,\qquad \omega^3=\exd x^3,\\
 \theta_a= \ts\frac{\partial^2}{\partial x^i\partial y^a}F\exd x^i- \big(\ts\frac{\partial}{\partial x^a}F+\ts\frac{\partial}{\partial x^0}F\ts\frac{\partial}{\partial y^a}F\big)\exd x^n+\ts\frac{\partial^2}{\partial y^a\partial y^b}F\exd y^b.
\end{gather*}
A 2-adapted coframing is given by
\begin{gather*}\omega^1\lra\omega^1-\ts\frac{c_{111}}{c_{12}}y^1\omega^2,\qquad\omega^3\lra \ts\frac{1}{c_{12}}\omega^3,\qquad\theta_1\lra\theta_1-\ts\frac{c_{111}}{c_{12}}y^1\theta_2,\end{gather*}
with respect to which the second fundamental form is normalized to $h^{ab}=\left[
\begin{smallmatrix}
 0 & 1\\
 1 & 0
\end{smallmatrix}\right].
$ The details of higher coframe adaptations won't be discussed.

It should be noted that when the only non-zero coefficients are $c_{12}=1$ and $c_{0}=(x^1)^2$, then one obtains a local description of the split signature pp-wave metric. Moreover, when $c_0=c_1=c_2=0$ one obtains a local description of causal structures whose null cones are at each point projectively equivalent to Cayley's cubic scroll\footnote{A scroll is a ruled surface that is not developable \cite{Struik}.} whose graph can be expressed as $y^0=\third (y^1)^3+y^1y^2$ in $\PP^3$.

\subsubsection{Locally isotrivially flat causal structures}
This class of causal structures are the analogue of locally Minkowskian Finsler structures.
\begin{df}\label{def:isotrivially-flat} A causal structure $\big(\cC,M^{n+1}\big)$ is called locally isotrivially flat if it is locally equivalent to $(M\times \cV,M)$ where $\cV\subset \PP^{n}$ is a projective hypersurface.
\end{df}

Alternatively, a locally isotrivially flat causal structure has the property that locally it can be expressed as the graph of a function $F(y)$ (see equation~\eqref{eq:ConeStr-DefGraph}). Note that an isotrivially flat causal structure does not induce a conformal structure on $M$ unless $\cV$ is a quadric, in which case the obtained conformal structure is flat.

The local expressions~\eqref{eq:AdaptedCoframe01} and~\eqref{eq:theta-a-coframe01} of the 1-forms $\omega^i$ and $\theta_a$ can be used to express the Fubini cubic form and \wsf curvature. Firstly, note that $\theta_a=\partial_\una\partial_\unb F\exd y^b$. To express $F_{abc}$, consider the following transformation
\begin{gather} \label{eq:isotrivial-transfor}
 \tilde\omega^0=\bfa^2\omega^0,\qquad \tilde\omega^a=\bfE^a\omega^0+\omega^a,\qquad \tilde\omega^n=\bfe\omega^0+\omega^n,\qquad
 \tilde\theta_a=\bfa^2\theta_a.
\end{gather}
Differentiating $\tilde\omega^0$, it follows that
\begin{gather*}\exd\tilde\omega^0=-2\phi_0\w\tilde\omega^0-\tilde\theta_a\w\tilde\omega^a,\end{gather*}
where $\phi_0=-\frac{\exd\bfa}{\bfa}-\frac{\bfE^a}{2\bfa^2}\tilde\theta_a$. Consequently, in equation~\eqref{1-adapted-2} for 1-adapted coframes one obtains
\begin{gather*}\gamma^a{}_0=\ts\frac{1}{\bfa^2}\exd\bfE^a+\frac{1}{\bfa^4}\big(\bfE^a\bfE^b+H^{ab}\big)\theta_b, \qquad \psi^a{}_b=\frac{1}{\bfa^2}\bfE^a\theta_b,\qquad h^{ab}=\frac{1}{\bfa^2}H^{ab},\end{gather*}
where $[H^{ab}]$ denotes the inverse matrix of the vertical Hessian $[\partial_\una\partial_\unb F]$.

By composing equations \eqref{eq:InfGpAction-hab} and \eqref{eq:F-abc-E-aic}, it follows that the Fubini cubic form $F^{abc}$ is defined via the equation
\begin{gather*}\exd h^{ab}+h^{ac}\psi^{b}{}_{c}+h^{bc}\psi^a{}_c-2h^{ab} (\phi_0+\phi_\n) \equiv F^{abc}\theta_c\qquad \operatorname{mod} \ \Ibas.
\end{gather*}
Transformations \eqref{eq:isotrivial-transfor} imply that $\phi_n=0$. Using the above expressions for $\phi_0$ and $\psi^a{}_b$, without normalizing $h^{ab}$ to $\ve^{ab}$, one can impose the apolarity condition $h^{ab}F_{abc}=0$, as in~\eqref{eq:apolarity-causal}, by setting
\begin{gather*}\bfa=\big(\det (\partial_\una\partial_\unb F )\big)^{\frac{1}{2(n+1)}},\qquad \bfE^d=-\frac{\bfa^2}{n+1}H^{ab}H^{cd}\partial_\una\partial_\unb\partial_\unc F.\end{gather*}
Finally, using $\partial_{\una}\partial_\unb F$ to lower indices, it follows that
\begin{gather} \label{eq:explicit-Fubini-cubic-form}
F_{abc}=\bfa^{-2}\left(\partial_\una\partial_\unb\partial_{\unc}F -\frac{1}{n+1}\big( \partial_{\una}\partial_\unb F F_c+\partial_{\unb}\partial_\unc F F_a+\partial_{\unc}\partial_\una F F_b\big)\right),
\end{gather}
where $F_a=H^{bc}\partial_\una\partial_\unb\partial_{\unc}F$.

The above computations go through for any causal structure. The expression~\eqref{eq:explicit-Fubini-cubic-form} (see also \cite{AG-Projective,SasakiBook}) shows that the Fubini cubic form is a third order invariant. Using the Bianchi identities~\eqref{eq:causal-bianchi-K}, it follows that the jet order of the invariants $K^{ab}$ and $L^a$ over $\cC$ are four and five, respectively.

To find the \wsf curvature, note that without normalizing $h^{ab}$, the quantities $E_{anb}$ in equation~\eqref{eq:F-abc-E-aic} are defined via
\begin{gather*}\exd h^{ab}+h^{ac}\psi^{b}{}_{c}+h^{bc}\psi^a{}_c-2h^{ab}(\phi_0+\phi_\n) =E^a{}_n{}^b\omega^n \!\!\!\qquad \operatorname{mod} \! \ I\big\{\omega^0,\dots,\omega^{n-1},\theta_1,\dots,\theta_{n-1}\big\},
\end{gather*}
which by isotrivially flatness is zero. Hence, the reduction of $\pi_{ab}$ via \eqref{eq:E_aib-InfGpAct}, gives $\pi_{ab}-\frac{1}{n-1}\ve^{cd}\pi_{cd}\ve_{ab}\equiv 0$ modulo $\Itot$. Combining this with the fact that $\exd \theta_a\equiv 0$ mod $I\{\theta_1,\dots,\theta_{n-1}\}$, it follows that $W^a{}_{nbn}=0$, i.e., locally isotrivially flat causal structures have vanishing \wsf curvature.

\section{Causal structures with vanishing \wsf curvature}\label{sec:horiz-flat-causal}

In this section causal structures with vanishing \wsf curvature are studied. For such classes, a compatibility condition is derived which can be used to give a causal analogue of Landsberg spaces in Finsler geometry (see~\cite{BCS-Finsler} for an account of various types of Finsler structures). The section ends with a discussion on Lie contact structures on the locally defined space of characteristic curves of a causal structure with vanishing \wsf curvature.

\subsection{A compatibility condition}\label{sec:non-prem-cond}
In this Section a class of causal structures which can be regarded as analogues of Landsberg spaces in Finsler geometry is introduced. Their similarity is the content of Proposition~\ref{prop:Landsberg-analog}.

Since $\mathrm{D}F_{abc}$ in \eqref{infin-I} is semi-basic,
 \begin{gather} \label{eq:F_abc-coframe-derivatives-infin-action}
 \exd F_{abc} -F_{abd} \phi^d{}_c-F_{adc} \phi^d{}_b -F_{dbc} \phi^d{}_a -F_{abc} (\phi_0-\phi_\n)=F_{abc;i}\omega^i+F_{abc;\und}\theta^d.
 \end{gather}
After taking the Lie derivative of \eqref{eq:F_abc-coframe-derivatives-infin-action} along characteristic curves and using the structure equations, one arrives at
\begin{gather*}\exd f_{abc}-f_{abd} \phi^d{}_c-f_{adc} \phi^d{}_b -f_{dbc} \phi^d{}_a
-f_{abc} (\phi_0+\phi_\n) +F_{abc}\pi_n\equiv 0\qquad \operatorname{mod} \ \Itot.\end{gather*}
 Assume $f_{abc}=\lambda F_{abc}$, for some function $\lambda. $ Then, if $F_{abc}\neq 0$, the infinitesimal group action on~$\lambda$ is given by
\begin{gather*}\exd\lambda-2\lambda\phi_n+\pi_n\equiv 0\qquad \operatorname{mod} \ \Itot.\end{gather*} By translating~$\lambda$ to zero, the 1-form $\pi_n$ reduces to
 \begin{gather} \label{eq:pi_n-reduced-Landsberg}
\pi_n=p_{ni}\omega^i+p_{n\una}\theta^a.
 \end{gather}
for some coefficients $p_{ni}$ and $p_{n\una}$. Replacing $\pi_n$ with the above expression in equation~\eqref{1st-causal-IV}, one finds the infinitesimal action of the group $G_5$ on $p_{n\una}$ to be given by
\begin{gather*}dp_{n\una}-p_{n\una}(\phi_0+\phi_n)-p_{n\und}\phi^d{}_a+\pi_a\equiv 0 \qquad \operatorname{mod} \ \Itot.\end{gather*}
Translating $p_{n\una}$ to zero results in the reduction
\begin{gather} \label{eq:pi_a-reduced-Landsberg}
 \pi_a=p_{ai}\omega^i+p_{a\unb}\theta^b.
\end{gather}
Consequently, equation \eqref{1st-causal-IV} changes to
\begin{gather} \label{eq:d-theta-landsberg-1}
 \exd\theta_a=\big(\phi^b{}_a-(\phi_0-\phi_n)\delta^b{}_a\big)\w\theta_b+ \half W_{anij}\omega^i\w\omega^j+E_{anb0}\theta^b\w\omega^0,
\end{gather}
where
\begin{gather*}
 W_{anbn}=p_{nn},\qquad W_{anbc}=-p_{nb}\ve_{ac}+p_{nc}\ve_{ab},\\
W_{an0n}=p_{a0},\qquad W_{an0b}=p_{ab},\qquad E_{anb0}=-p_{a\unb}.
\end{gather*}
Similarly, one finds that the normalization $\ve^{ab}p_{a\unb}=0$ results in $\pi_0\equiv 0$ modulo $\Itot$. It follows that
\begin{gather} \label{eq:pi-0-reduced-landsberg}
 \pi_0=p_{0i}\omega^i+p_{0\una}\theta^a.
\end{gather}

Assume furthermore that $p_{nn}\neq0$. Using the relation \eqref{infin-V}, the infinitesimal group action on $p_{nn}$ is found to be
\begin{gather} \label{eq:conformal-group-action-on-pnn}
 \exd p_{nn}-4p_{nn}\phi_n\equiv 0 \qquad \operatorname{mod} \ \Itot.
\end{gather}
Hence, the normalization $p_{nn}=\pm 1$ can be made, resulting in $\phi_n$ reduced to
\begin{gather} \label{eq:phi_n-reduced-landsberg}
 \phi_n=r_i\omega^i+r_\una\theta^a.
\end{gather}
Furthermore, the infinitesimal group action of the reduced structure group on $r_\una$ is given by
\begin{gather*}\exd r_{\una}+r_{\unb}\phi^b{}_a-\gamma_a\equiv0\qquad \operatorname{mod} \ \Itot.\end{gather*}
As a result, by translating $r_{\una}$ to zero, the 1-forms $\gamma_a$ are reduced, i.e.,
\begin{gather} \label{eq:gamma_a-reduced-landsberg}
 \gamma_a=q_{ai}\omega^i+q_{a\unb}\theta^b.
\end{gather}
After these reductions, the proposition below can be proved.
\begin{prop}\label{prop:Landsberg-analog} A causal structure with vanishing \wsf curvature satisfying $f_{abc}=\lambda F_{abc}$ has vanishing \ssf curvature.
\end{prop}
\begin{proof} As was discussed, the condition $f_{abc}=\lambda F_{abc}$ results in the reduction of the 1-forms $\pi_i$ according to equations \eqref{eq:pi_n-reduced-Landsberg}, \eqref{eq:pi_a-reduced-Landsberg} and \eqref{eq:pi-0-reduced-landsberg}. Assuming $p_{nn}\neq 0$, equation \eqref{eq:conformal-group-action-on-pnn} is used to reduce~$\phi_n$, as in~\eqref{eq:phi_n-reduced-landsberg}, by normalizing $p_{nn}$ to a non-zero value~$c$. Consequently, the 1-forms $\gamma_a$ can be reduced to~\eqref{eq:gamma_a-reduced-landsberg}.

Now that all the 1-forms $\pi_i$, $\phi_n$ and $\gamma^a$ are reduced equation \eqref{eq:d-theta-landsberg-1} can be written as
\begin{gather*}
 \exd\theta_a= \big(\phi^b{}_a-\phi_0\delta^b{}_a\big)\w\theta_b +c \omega^a\w\omega^n+(p_{ab}-p_{n0}\ve_{ab})\omega^0\w\omega^b +p_{an}\omega^0\w\omega^n\\
\hphantom{\exd\theta_a=}{} -(p_{a\unb}+r_0\ve_{ab})\theta^b\w\omega^0-p_{nb}\ve_{ac}\omega^b\w\omega^c-r_c\ve_{ab}\theta^b\w\omega^c -r_n\theta^a\w\omega^n,
\end{gather*}
where $c$ denotes the normalized value of $p_{nn}$. It is noted that the translations of~$r_\una$, $p_{n\una}$, $\ve^{ab}p_{a\unb}$ to zero are used in the above equation. Similarly, express $\exd\omega^a$ and $\exd\omega^n$ by replacing $\gamma^a$ and $\phi_n$ with their reduced forms in~\eqref{eq:1st-order-str-eqns}.

The reduced structure equations can be used to express the identity $\exd^2\theta_a=0$. Collecting the coefficients of the 3-form $\theta^a\w\omega^b\w\omega^n, $ it follows that
\begin{gather} \label{eq:F_abc-vanish-landsberg}
 cF_{abc}+\big(r_{c;n}-r_{n;c}+\half p_{cn} -p_{nc}-r_nq_{cn}-r_cr_n\big)\ve_{ab}+p_{nb}\ve_{ac}+p_{an}\ve_{bc}+W_{abcn}=0,
\end{gather}
where $W_{abcn}=-W_{bacn}$ is a torsion term for $\Phi_{ab}$ as defined in \eqref{eq:curvatures1} and \eqref{eq:causal-StrEqns-incomplete-1}. According to~\eqref{eq:2nd-order-W-ijkl} and~\eqref{eq:2nd-order-biachies}, the value of $W_{abcn}$ is zero before the reduction. After reducing $\pi_a$ and $\gamma_a$, it is found that
\begin{gather*}W_{abcn}=p_{bn}\ve_{ac}-p_{an}\ve_{bc}.\end{gather*}
If $c\neq 0$, then equation \eqref{eq:F_abc-vanish-landsberg} and the fact that $F_{abc}$ is totally symmetric and trace-free imply $F_{abc}=0$. This is a contradiction with the assumptions $f_{abc}=\lambda F_{abc}$ and $F_{abc}$ being non-vanishing. As a result, $c$ has to vanish.
\end{proof}

\begin{rmk} The proposition above has an analogue in Finsler geometry. Recall that a Finsler metric is called Landsberg if the derivative of its Cartan torsion along geodesics is zero, i.e.,
\begin{gather*}\dot{I}_{ijk}=0.\end{gather*}
It can be shown that if the flag curvature of a Landsberg metric with non-vanishing Cartan torsion is constant, then it has to be zero \cite[Section~12.1]{BCS-Finsler}. By the proposition above, the condition $F_{abc;n}=\lambda F_{abc}$ can be used to translate $F_{abc;n}$ to zero. As a result of this translation, if the causal structure has vanishing \wsf curvature, then its \ssf curvature has to be zero. This observation suggests that causal structures satisfying $f_{abc}=\lambda F_{abc}$ are causal analogues of Landsberg spaces.
\end{rmk}

\subsection[$\beta$-integrable Lie contact structures]{$\boldsymbol{\beta}$-integrable Lie contact structures} \label{sec:semi-integrable-lie}

In this section it is shown how a causal structure with vanishing \wsf curvature descends to a~Lie contact structures on $\cN$. The construction is analogous to Grossman's notion of $\beta$-integrable Segr\'e structures arising from torsion-free path geometries \cite{GrossmanPath} (see also \cite{Mettler}).

Let $p\in \cC$ be a generic point, that is a point admitting an open neighborhood $U\subset \cC$ foliated by characteristic curves. As a result, the quotient map $\tau\colon U\ra \cN^{2n-1}$ defines the space of characteristic curves in $U$.

Recall from Section \ref{sec:char-curv} that a characteristic vector field $\bfv$ spans the unique degenerate direction of $\omega^0$ and $\exd\omega^0$, at every point, i.e.,
\begin{gather*}\bfv\im\omega^0=0,\qquad \bfv\im\exd\omega^0=0.\end{gather*}
It follows that $\cL_\bfv\omega^0=0$. Therefore, there exists a 1-form $\tilde\omega^0\subset \Gamma(T^*\cN)$ of rank $2n-1$ such that $\tau^*\tilde\omega^0=\omega^0$. As a result, the quasi-contact structure on $\cC$ defined by $\omega^0$ induces a contact distribution on $\cN$. By abuse of notation, $\omega^0$ is used to denote $\tilde\omega^0$ as the distinction is clear from the context.

Because
\begin{gather*}\bfv\im\omega^a=\bfv\im\theta^a=0,\end{gather*}
 the 1-forms $\omega^0,\dots,\theta^{n-1}$ introduce a coframe on $\cN$, and the $\{e\}$-structure $\cP$ fibers over~$\cN$. To understand the resulting structure on $\cN$, consider the flat model. Recall from Section~\ref{sec:flat-model-2} that the flat causal geometry corresponds to the null cone bundle of a hyperquadric $Q_{p+1,q+1}\subset\PP^{n+2}$, and its space of characteristic curves is the space of isotropic 2-planes denoted by $G_0(2,n+4)$. The resulting structure was first studied by Sato and Yamaguchi in~\cite{SYLieContactStr1} referred to as Lie contact structures. As a parabolic geometry, this geometry is modeled on $(B_{n-1},P_2)$ or $(D_n,P_2)$, for $n\geq 4$ or $(D_3,P_{1,2})$ (see Section~\ref{sec:norm-cart-conn-1}).

In this section it is shown how a causal structure with vanishing \wsf curvature descends to a Lie contact structures on~$\cN$. As was discussed in Section~\ref{sec:an-e-structure-1}, no attempt is made to introduce a Cartan connection for the Lie contact structure from the derived $\{e\}$-structure.

First, a definition of Lie contact structures in terms of a field of Segr\'e cones of type $(2,\sN)$ is given.
Let $\VV$ be a $2\sN$-dimensional vector space with a decomposition $\VV\cong\RR^2\otimes\RR^\mathsf{N}. $ An element $p\in \VV$ is said to be of rank one if and only if $p=a\otimes b$. A Segr\'e cone of type $(2,\sN)$, denoted by $S(2,\sN)$, is defined as the set of rank one elements of this decomposition. In terms of local coordinates $\big(z^1_{1},\dots,z^1_{\mathsf{N}},z^2_{1},\dots,z^2_{\mathsf{N}}\big)$ adapted to this decomposition, the Segr\'e cone is given by the points satisfying the homogeneous quadratic equations \begin{gather*}z^1_{\mu}z^2_{\nu}-z^2_{\mu}z^1_{\nu}=0,\end{gather*}
where $1\leq \mu,\nu\leq\mathsf{N}$. As a result, a Segr\'e cone of type $(2,\mathsf{N})$ has a double ruling. A $(\mathsf{N}-1)$-parameter ruling by 2-planes $E_v\cong \RR^2\otimes v$ where $v\in \RR^\mathsf{N}$. Such 2-planes are called $\alpha$-planes~\cite{AG-Conformal}. Another ruling of $S(2,\sN)$ is given by the 1-parameter family of $\mathsf{N}$-planes $u\otimes\RR^{\mathsf{N}}$ where $u\in\RR^2$. Such $\mathsf{N}$-planes are called $\beta$-planes.

The proposition below is used to show how a Lie contact structure is induced on $\cN$.
\begin{prop}\label{prop:segre-cone}
Given a causal structure with adapted coframe $(\omega^i,\theta_a)$, the span of the bilinear forms
\begin{gather} \label{eq:Segre-bilinear-form-Upsilon}
 \Upsilon^{ab}=\theta^a\circ\omega^b-\theta^b\circ\omega^a
\end{gather}
restricted to $H=\operatorname{Ker} \omega^0$ is invariant under the action of the structure group $G_5$. Moreover, $\cC$ has vanishing \wsf curvature, if and only if $\operatorname{span}\big\{\Upsilon^{ab}\big\}$ is invariant along characteristic curves.
\end{prop}
\begin{proof}By the matrix form of the structure group $G_5$ in \eqref{eq:G5-StrGroup}, it follows that the span of $\Upsilon^{ab}$ modulo $I\{\omega^0\}$ is invariant under the action of~$G_5$.

Let $\bfv$ denote a characteristic vector field. To verify the second claim it suffices to note that~\eqref{eq:1st-order-str-eqns} implies
\begin{gather*}\cL_{\bfv}\Upsilon^{ab}=A^{ab}{}_{cd}\Upsilon^{cd}+\big(\bfv\im\omega^n\big)\big(W^b{}_{ncn}\omega^a-W^a{}_{ncn}\omega^b\big)\circ\omega^c,\end{gather*}
where $A^{ab}{}_{cd}=\bfv\im \big(2\phi_0\delta^a_c\delta^b_d+\delta^b_d\phi^a{}_c-\delta^a_c\phi^b{}_d\big)$. Thus, if the condition $W^a{}_{nbn}=0$ is satisfies, i.e., causal structure has vanishing \wsf curvature, then $\operatorname{span}\big\{\Upsilon^{ab}\big\}$ is invariant along the characteristic curves.
\end{proof}

As a result of the proposition above, the set of tangent vectors $\bfu\in H$ satisfying $\Upsilon^{ab}(\bfu,\bfu)=0$ is well-defined and gives a Segr\'e cone of type $(2,n-1)$ in the contact distribution $H\subset T\cN$.

\begin{df}\label{def:Lie-contact-str}
 A manifold $N^{2\sN+1}$ with a contact distribution $H$ is said to have a Lie contact structure of signature $(p,q)$, $p+q=\sN$ if there is a smoothly varying family of Segr\'e cones $S(2,\sN)\subset H_x$ such that at every point $x\in N$, its $\beta$-planes are endowed with conformally invariant inner product of signature $(p,q)$ and are Lagrangian with respect to the symplectic 2-form that is induced by the contact 1-form on the contact distribution.

 A Lie contact structure is called \emph{$\beta$-integrable} if there exists a 1-parameter family of foliations by Legendrian submanifolds with the property that given a point $x\in N$ and a $\beta$-plane in $S(2,\sN)\subset H_x$, there passes a unique member of the family through $x$, tangent to that $\beta$-plane.
\end{df}
There are other equivalent definitions of a Lie contact structure discussed in \cite{ZadnikLieContact}.

\begin{rmk}\label{rmk:bundle-of-beta-planes} $\beta$-integrability condition gives rise to a foliation of the bundle $ \tilde \rho\colon \cE_\beta\ra\cN$ where~$\cE_\beta$ denotes the bundle of $\beta$-planes of the Segr\'e cones. Because $\beta$-planes are Lagrangian subspaces of the contact distribution, $\cE_\beta$ is a sub-bundle of $\mathrm{LG}(\cN)$ where $\rho\colon \mathrm{LG}(\cN)\ra\cN$ is the bundle of Lagrangian--Grassmannians of the contact distribution $H$. Since, the bundle map~$\tilde\rho$ is the restriction of~$\rho$ to the sub-bundle $\cE_\beta$, by abuse of notation, it is denoted by $\rho$ as well.
\end{rmk}

The theorem below gives a characterization of causal structures with vanishing \wsf curvature in terms of $\beta$-integrable Lie contact structures induced on their space of characteristic curves.
\begin{thm}\label{thm:Lie-contact-conf-flat} Given a causal structure $(\cC,M)$ with vanishing \wsf curvature of signature \smash{$(p+1,q+1)$} its space of characteristic curves admits a $\beta$-integrable Lie contact structure of signature~$(p,q)$. Conversely, any $\beta$-integrable Lie contact structure of signature~$(p,q)$ on a~space~$\cN$ induces a causal structure of signature $(p+1,q+1)$ with vanishing \wsf curvature on its space of ruling Legendrian submanifolds.
\end{thm}
\begin{proof}To see how $\cN$ is endowed with a Lie contact structure, firstly note that by Proposi\-tion~\ref{prop:segre-cone} the contact distribution carries a Segr\'e structure of type $(2,n-1)$. By the definition of~$\Upsilon^{ab}$, it follows that the $\beta$-planes, i.e., the 1-parameter ruling of the Segr\'e cone, are given by
 \begin{gather} \label{eq:beta-plane-parameter}
\Ker\big\{A\theta^a+B\omega^a\,|\, a=1,\dots,n-1\big\},
\end{gather}
$[A:B]\in\PP^1$. Moreover, the inner product $\ve_{ab}\theta^a\circ\theta^b$ is of signature $(p,q)$ and its conformal class is well-defined on each $\beta$-plane, i.e., it is conformally invariant under the action of $G_5$ modulo $\Ibas=I\{\omega^0,\dots,\omega^n\}$. Finally, to show $\beta$-planes are Lagrangian consider a $\beta$-plane~$E_{\beta_1}$ corresponding to $[A:B]\in\PP^1$ in~\eqref{eq:beta-plane-parameter}. The tangent vectors $(v_1,\dots,v_{n-1})$ where $v_a=B\frac{\partial}{\partial\theta^a}-A\frac{\partial}{\partial\omega^a}$ define a~basis for $E_{\beta_1}$. Because $\exd\omega^0\equiv \ve_{ab}\theta^a\w\omega^b$, modulo $I\{\omega^0\}$, it follows that $\exd\omega^0(v_a,v_b)=0$. Hence, $E_{\beta_1}$ is Lagrangian with respect to the conformal symplectic form~$\exd\omega^0$.

 As for the $\beta$-integrability of these Lie contact structures, the following lemma will be useful.
 \begin{lem}\label{lem:canonical-diffeo-Lie-cont-conf-flat}
 Given a causal structure with vanishing \wsf curvature, there is a canonical bundle diffeomorphism $\varpi\colon \cC\ra\cE_\beta$ between the fiber bundles $\tau\colon \cC\ra \cN$ and $\rho\colon \cE_\beta\ra\cN$,.
 \end{lem}
 \begin{proof} Let $\gamma_t=(x(t);y(t))$ be a characteristic curve of $\cC$ which corresponds to the point $[\gamma]\in\cN$. Let $\cD$ be the completely integrable distribution of $T\cC$ given by the vertical tangent spaces $T_{y}\cC_{x}:=\Ker\cI_{\mathsf{bas}}$ with $\cD_t:=T_{y(t)}\cC_{x(t)}$. By the discussion above, it follows that, at $t=t_0$, $\tau_*\cD_{t_0}\subset H_{[\gamma]}$ is isotropic, $(n-1)$-dimensional and lies in the zero locus of $\Upsilon^{ab}$, i.e., $\tau_*\cD_{t_0}$ is a~$\beta$-plane. By defining
 \begin{gather*} 
 \varpi(\gamma_{t_0})=\tau_*\cD_{t_0},
 \end{gather*}
one obtains the desired diffeomorphism.
 \end{proof}

Continuing the proof of Theorem \ref{thm:Lie-contact-conf-flat}, note that the fibers $\cC_x$, as the integral manifolds of $\cD$, give a foliation of $\cE_\beta$ via the diffeomorphism~$\varpi$. Because of the transversality of the fibers $\cC_x$ and characteristic curves $\gamma$ passing through $x$, it follows that the subsets $\tau(\cC_x)\subset \cN$ are immersed as $(n-1)$-dimensional submanifolds of~$\cN$. These submanifolds are Legendrian because it was shown that their tangent spaces, $\tau_*\cD$, are $\beta$-planes, i.e., Lagrangian subspaces of the contact distribution which rule the Segr\'e cone.

The converse part of the theorem can be shown as follows. Let $\omega^0$ be a~contact form. Since a~Segr\'e cone of type $(2,n-1)$ admits a 1-parameter ruling by $(n-1)$-planes, one can find 1-forms $\{\eta^a,\zeta^a \,|\, a=1,\dots,n-1\}$ complementing $\omega^0$ in $T^*\cN$ such that the Segr\'e cone can be expressed as the vanishing set of the bilinear forms
\begin{gather*}\Upsilon^{ab}=\eta^a\circ\zeta^b-\eta^b\circ\zeta^a.\end{gather*}
As a result, the 1-parameter family of ruling $(n-1)$-planes of the Segr\'e cone can be expressed as $\Ker\{t\eta^a+(1-t)\zeta^a\,|\,a=1,\dots,n-1\}$ where $t\in (0,1)$. Let $\rho\colon \cE_\beta\ra\cN$ denote the bundle of $\beta$-planes of the Segr\'e cone. A point $p\in\cE_\beta$ represents a $\beta$-plane of the contact distribution $H$ which is Lagrangian with respect to the conformal symplectic structure on the contact distribution. Thus, the tautological bundle $K\subset T\cE_\beta$ can be defined with fibers
\begin{gather*}K_p=\big\{v\in T_p\cE_\beta \,|\, \rho_*|_p(v)\in p\subset T_{\rho(p)}\cN\big\}.\end{gather*}
As a result, the subspaces $\rho_*(K_p)$ rule the Segr\'e cone $S(2,n-1)\subset H_{\rho(p)}\subset T_{\rho(p)}\cN$. The fibers of $\rho\colon \cE_\beta\ra \cN$ are one dimensional and parametrize the ruling subspaces $\rho_*(K_p)$. Let $\bfv$ be a~vertical vector field of the fibration $\rho\colon \cE_\beta\ra \cN$ satisfying $\rho_*(\bfv)=\frac{\partial}{\partial t}$, and by abuse of notation, let a multiple of $\rho^*(\omega^0)$ be denoted by $\omega^0$ as well. Hence, $\omega^0\subset T^*\cE_\beta$ is a quasi-contact 1-form.
Defining $\omega^a=\rho^*(t\eta^a+(1-t)\zeta^a)$, it follows that $\cL_\bfv\omega^a=\rho^*(\eta^a-\zeta^a)$, and, consequently,
\begin{gather*}\omega^1\w\cdots\w\omega^{n-1}\w\cL_\bfv\omega^1\w\cdots\w\cL_\bfv\omega^{n-1}\neq 0\qquad\operatorname{mod} \ I\big\{\omega^0\big\}.\end{gather*}
The reason that it is non-zero, modulo $I\{\omega^0\}$, is that the Segr\'e cone is a subset of the contact distribution.

Hence, locally, the tautological bundle is given by $K=\Ker\{\omega^0,\dots,\omega^{n-1}\}$.
As a result, the 1-forms $\omega^0, \omega^a$, $\theta_a=\cL_\bfv\omega^a$ which are semi-basic with respect to the fibration $\rho\colon \cE_\beta\rightarrow\cN$ and the 1-form $\omega^n$, which is vertical, span $T^*_p\cE_\beta$, and one has
\begin{gather*}\exd\omega^a\equiv -\theta^a\w\omega^n+T^a{}_{bc}\theta^a\w\theta^b \qquad\operatorname{mod} \ I\big\{\omega^0,\dots,\omega^{n-1}\big\}.\end{gather*}
Since the Lie contact structure is assumed to be $\beta$-integrable, the lift of its associated 1-parameter family of Legendrian submanifolds foliate $\cE_\beta$ by quasi-Legendrian submanifolds. The tangent distribution to this foliation is a corank sub-bundle of the tautological subbundle $K$ given by $\Ker\{\omega^0,\dots,\omega^n\}$. As a result, $\omega^i$'s form a Pfaffian system which implies $T^a{}_{bc}=0$.
Moreover, knowing that $K=\Ker\{\omega^0,\dots,\omega^{n-1}\}$ and that $\omega^0$ is semi-basic, one arrives at
\begin{gather*}\exd\omega^0\equiv -h_{ab}\theta^a\w\omega^b \qquad\operatorname{mod} \ I\big\{\omega^0\big\}.\end{gather*}
In the identity $\exd^2\omega^0=0$, the vanishing of the 3-form $\theta^a\w\theta^b\w\omega^n$ implies that $h_{ab}=h_{ba}$.

Let $M$ be the leaf space of the quasi-Legendrian foliation of $\cE_\beta$, via the quotient map $\pi\colon \cE_\beta\rightarrow M$. Because the tangent planes of the leaves are given by $\Ker\{\omega^0,\dots,\omega^n\}$, the 1-forms $\{\omega^0,\dots,\omega^n\}$ are semi-basic with respect to the fibration $\pi\colon \cE_\beta\ra M$, and the 1-forms $\theta^a$ are vertical. Moreover, the coframe on $\cE_\beta$ is adapted to the flag~\eqref{Conframe-adaptation-00}. Note that the tautological bundle $K$ is isomorphic to the tautological bundle introduced in Section~\ref{sec:adapted-flag}.

According to Section \ref{sec:structure-equations}, the complete integrability of the Pfaffian system $\Ibas$ and the relations
\begin{gather*}
 \exd\omega^0\equiv -\theta_a\w\omega^a\qquad\operatorname{mod} \ I\big\{\omega^0\big\},\\
 \exd\omega^a\equiv -h^{ab}\theta_b\w\omega^n \qquad\operatorname{mod} \ I\big\{\omega^0,\dots,\omega^{n-1}\big\},
\end{gather*}
where $\theta_a:=h_{ab}\theta^b$, suffice to carry out the coframe adaptations and obtain the $\{e\}$-structure that characterizes a causal structure. Note that the distribution given by $\Ker\{\omega^0,\dots,\omega^{n-1}\}$ characterizes a causal structure according to Section~\ref{sec:char-caus-struct}. The one dimensional fibers of \smash{$\rho\colon \cE_\beta\ra \cN$} are mapped to the characteristic curves of the causal structure.

Finally, since the vanishing set of the quadratic forms $\Upsilon^{ab}$ defined in \eqref{eq:Segre-bilinear-form-Upsilon} is preserved, the \wsf curvature of the obtained causal structure has to vanish.
\end{proof}

\begin{appendix}
\section{Full structure equations}\label{cha:append-full-struct}

Appendix~\ref{cha:append-full-struct} contains the computations needed to associate an $\{e\}$-structure to a causal geometry. At the end, the symmetries and Bianchi identities for the torsion elements are expressed.

\subsection[Prolongation and an $\{e\}$-structure]{Prolongation and an $\boldsymbol{\{e\}}$-structure}\label{sec:prolongation-an-e}
Following the first order structure equations derived in \eqref{eq:1st-order-str-eqns}, the method of equivalence can be used to associate a preferred choice of coframing to a causal structure. The first step involves prolonging the first order structure equations.

\subsubsection{Prolonged structure group}\label{sec:prol-struct-group}
Using the pseudo-connection in \eqref{eq:1st-order-str-eqns}, one finds that the Cartan characters\footnote{See \cite{Gardner-Book, Olver-Equiv-Book} for an account.} are
\begin{gather*}c_1=2n,\qquad c_2=n-1,\qquad c_i=n-i,\qquad \mathrm{for}\ 3\le i\le n-1.\end{gather*}

With respect to the basis $\{\omega^i,\theta^a,\phi_0,\phi^a{}_{b},\phi_n,\gamma_a,\pi_n,\pi_a\}$, the prolonged structure group, $G^{(1)}_1$, is
\begin{gather} \label{eq:parametric-prolonged-group}
\begin{pmatrix}
\II_{2n} &0\\
\bfR & \II_{p}
\end{pmatrix},\qquad \text{such\ that}\quad
\bfR=\begin{pmatrix}
\bfd & 0 & 0 & 0\\
0 & 0 & 0 & 0\\
 0 & 0 & 0 & 0\\
0 & \bfd & 0 & 0 \\
\bff & 0 & 0 & 0 \\
\bfF & \bff & 0 & \bfd
\end{pmatrix},
\end{gather}
where $p=\frac{1}{2}(n^2+n+4)$, $\bfE\in \RR^{n-1}$, and $\bfd,\bff\in\RR$. The Maurer--Cartan forms of the prolonged group can be written as
\begin{gather} \label{eq:MC-prolonged-group}
 \begin{pmatrix}
0 &0\\
\bfM & 0
\end{pmatrix},\qquad \mathrm{where}\quad
\bfM=\begin{pmatrix}
\pi_0 & 0 & 0 & 0\\
0 & 0 & 0 & 0\\
 0 & 0 & 0 & 0\\
0 & \pi_0 & 0 & 0 \\
\sigma & 0 & 0 & 0 \\
\tau_a & \sigma & 0 & \pi_0
\end{pmatrix},
\end{gather}
Since $\dim G^{(1)}_1=n+1$, the relation \begin{gather*}\dim G^{(1)}_1< \sum^{n-1}_{j=1} jc_j,\end{gather*}
holds for all $n\geq 3$. Thus, equations~\eqref{eq:1st-order-str-eqns} are not involutive, and because an $\{e\}$-structure is not obtained, according to Cartan's method of equivalence, one needs to prolong the structure equations.

\subsubsection{Part of the second order structure equations}\label{sec:second-order-struct}
To avoid taking too much space for computations and for later use, the following 2-forms are introduced
\begin{gather}
\Omega^0=\exd\omega^0+2\phi_0\w\omega^0+\theta_a\w\omega^a,\nonumber\\
\Omega^a=\exd\omega^a+\gamma^a\w\omega^0+\big(\phi^a{}_{b}+(\phi_0+\phi_n)\delta^a{}_{b}\big)\w\omega^b +\theta^a\w\omega^n,\nonumber\\
\Omega^n=\exd\omega^n+\gamma_a\w\omega^a+2\phi_n\w\omega^n,\nonumber\\
\Theta^a= \exd\theta^a+\pi^a\w\omega^0+\pi_n\w\omega^a +\big(\phi^a{}_{b}-(\phi_0-\phi_n)\delta^a{}_{b}\big) \w\theta^b,\nonumber\\
 \Phi_0= \exd\phi_0+\pi_0\w\omega^0+\half\pi_a\w\omega^a-\half\gamma_a\w\theta^a,\nonumber\\
 \Phi^a{}_{b}=\exd\phi^a{}_{b}+\phi^a{}_{c}\w\phi^c_{b}+\theta^a\w\gamma_b-\theta_b\w\gamma^a -\omega^a\w\pi_b+\omega_b\w\pi^a,\nonumber\\
 \Phi_n= \exd\phi_n+\half\pi_a\w\omega^a+\pi_n\w\omega^n+\half\gamma_a\w\theta^a,\nonumber\\
 \Gamma_a= \exd\gamma_a-\omega_a\w\pi_0-\omega^n\w\pi_a -\big(\phi^b{}_{a}+(\phi_0-\phi_n)\delta^b{}_{a}\big)\w\gamma_b,\nonumber\\
 \Pi_n= \exd\pi_n-\theta^a\w\pi_a-2\phi_n\w\pi_n,\nonumber\\
 \Pi_a=\exd\pi_a-\theta_a\w\pi_0-\big(\phi^b{}_{a}+(\phi_0+\phi_n)\delta^b{}_{a}\big)\w\pi_b-\gamma_a\w\pi_n. \label{eq:curvatures1}
 \end{gather}
As a result, structure equations \eqref{eq:1st-order-str-eqns} can be written as
\begin{gather*}
 \Omega^0=0,\qquad \Omega^a=-F^a{}_{bc}\theta^b\w\omega^c-K^a{}_{b}\theta^b\w\omega^0,\qquad
 \Omega^n=-L_a\theta^a\w\omega^0,\\
 \Theta_a=-f_{ ab c}\theta^b\w\omega^c+\half W_{anbc}\omega^b\w\omega^c+ W_{bncn}\omega^c\w\omega^n.
\end{gather*}

From the Lie algebra of the prolonged group in \eqref{eq:MC-prolonged-group}, it follows that
\begin{gather*}\Pi_n=-\sigma\w\omega^0+P_n,\qquad \Pi_a=-\tau_a\w\omega^0-\sigma\w\omega_a+P_a,\end{gather*}
where the 2-forms $P_n$ and $P_a$ are generated by $\{\omega^i,\theta^a,\phi_0,\phi^a{}_b,\phi_n,\gamma_a,\pi_n,\pi_a\}$. Note that the terms involving the 1-form $\pi_0$ in~\eqref{eq:MC-prolonged-group}, are included within the definition of $\Gamma_a$, $\Pi_n$ and $\Pi_a$.

Differentiating \eqref{1st-causal-I}, gives
\begin{gather*}
2\Phi_0\w\omega^0+\Theta_a\w\omega^a=0.
\end{gather*}
As a result of \eqref{1st-causal-IV} and the symmetries in~\eqref{eq:symmetris-of-1st-order-StrEqns}, it follows that $\Theta_a\w\omega^a=0$, and therefore, $\Phi_0= \zeta\w\omega^0$ for some 1-form $\zeta$. By replacing
\begin{gather*}\pi_0\lmt\pi_0-\zeta,\end{gather*}
one arrives at
\begin{gather} \label{eq:causal-Phi_0-1}
 \Phi_0=0.
\end{gather}

Differentiating \eqref{1st-causal-III}, and using the infinitesimal group action \eqref{infin-III}, it is obtained
\begin{subequations}\label{eq:causal-Gamma_a-Phi_n}
\begin{gather}
 \Gamma_a= \half W_{0aij}\omega^i\w\omega^j+ E_{0ab0}\theta^b\w\omega^0+\half C_{abc}\theta^b\w\theta^c\nonumber\\
\hphantom{\Gamma_a=}{} -F_{abc}\gamma^b\w\theta^c-L_a\pi_n\w\omega^0+\tau_{ab}\w\omega^b+\xi_a\w\omega^n,\label{eq:causal-Gamma_a-1} \\
 \Phi_n= W_{nnij}\omega^i\w\omega^j-\half \ell_a\theta^a\w\omega^0+\nu\w\omega^n+\half \xi_a\w\omega^a, \label{eq:causal-Phi_n-1}
\end{gather}
\end{subequations}
for some coefficients $E_{nab0}$, $C_{abc}$, where $\ell_a=L_{a;n}$, and the yet undetermined 1-forms $ \nu$, $\xi_a$, and $\tau_{ab}=\tau_{ba}$ which are congruent to zero modulo~$\Itot$. The reason for using $W_{nnij}$ instead of~$\half W_{nnij}$ in~\eqref{eq:causal-Phi_n-1} becomes clear when the replacements \eqref{eq:pi_0-pi_n-changed-1} are imposed.

By differentiating \eqref{1st-causal-II}, it follows that
\begin{gather*}\xi_a=k_{ab}\,\theta^b,\end{gather*}
where $k_{ab}:=K_{ab;n}$.

Now $\Phi^a{}_{b}$ can be determined as a result of differentiating \eqref{1st-causal-II} and replacing $\Phi_0$ and $\Phi_n$ by the expressions \eqref{eq:causal-Phi_0-1} and \eqref{eq:causal-Phi_n-1}. Using the infinitesimal group action on $K_{ab},F_{abc}$ given by~\eqref{infin-I} and~\eqref{infin-II}, and their symmetries in \eqref{eq:symmetris-of-1st-order-StrEqns}, and \eqref{1st-causal-IV}, the derivative of~\eqref{1st-causal-II} yields
\begin{gather*}
 0\equiv\big(\Phi^a{}_{b}+ \Phi_n\delta^a{}_{b} \big)\w\omega^b
+K_{ab} \theta^b\w\theta_c\w\omega^c -F_{abc} f^b{}_{de} \theta^d\w\omega^e\w\omega^c\\
\hphantom{0\equiv}{} +F_{abc;d}\,\omega^d\w\theta^b\w\omega^c
 +F_{abc;\bd} \theta^d\w\theta^b\w\omega^c+F_{abc} F^c{}_{de} \theta^b\w\theta^d\w\omega^e \\
\hphantom{0\equiv}{} +\half F_{abc} W^b{}_{nde} \omega^d\w\omega^e\w\omega^c \qquad\operatorname{mod} \ I\big\{\omega^0,\omega^n\big\}.
\end{gather*}
Multiplied by the $(n-1)$-form $\eta=\theta^1\w\cdots\w\theta^{n-1}$, the above equation results in
\begin{gather*}
 \big(\Phi^a{}_{b}+ \Phi_n\delta^a{}_{b} \big)\w\omega^b\w\eta\equiv \half -F_{aeb}W^b{}_{ncb} \omega^c\w\omega^d\w\omega^b\w\eta \qquad\operatorname{mod} \ I\big\{\omega^0,\omega^n\big\}.
\end{gather*}
By Cartan's lemma, it follows that
\begin{gather*}
 \big(\Phi^a{}_{b}+ \Phi_n\delta^a{}_{b} \big)\w\eta\equiv -\big(\zeta^a{}_{bc}\w\omega^c+\half F_{abe}W^e{}_{ncd}\omega^c\w\omega^d\big) \w\eta \qquad\operatorname{mod} \ I\big\{\omega^0,\omega^n\big\},
\end{gather*}
for some 1-form $\zeta^a{}_{bc}$, where $\zeta^a{}_{bc}\equiv\zeta^a{}_{cb}$ modulo $\Itot$.

Fixing $c=C$, after taking the wedge product of $\Phi^a{}_{b}+\Phi_n\delta^a{}_{b}$ with all the 1-forms $\omega^i$, such that $ i\neq C$, and using \eqref{eq:causal-Phi_n-1}, the result is
\begin{gather*}\Phi^a{}_{b}\w\omega^0\w\cdots\w\widehat{\omega^C}\w\cdots\w\omega^n\w\eta=\zeta^a{}_{bC} \w\omega^C\w\omega^0\w\cdots\w\widehat{\omega^C}\w\cdots\w\omega^n\w\eta.\end{gather*}
Carrying out the same procedure for $\Phi^b{}_{a}$, using $\Phi_{ab}+\Phi_{ba}=0$, it follows that
\begin{gather*}(\zeta_{abC}+\zeta_{baC})\w\omega^C\w\omega^0\w\cdots\w\widehat{\omega^C}\w\cdots\w\omega^n\w\eta=0,\end{gather*}
and therefore, $\zeta_{abc}+\zeta_{bac}\equiv 0$ modulo $\Itot$. However, since $\zeta^a{}_{bc}\equiv\zeta^a{}_{cb}$, one obtains $\zeta^a{}_{bc}\equiv 0$ modulo $\Itot$.

Consequently, equation
\begin{gather} \label{eq:causal-Phi_ab-1}
 \Phi_{ab}=\half W_{abij} \omega^i\w\omega^j+E_{abci} \theta^c\w\omega^i+\half C_{abcd}\theta^c\w\theta^d+\lambda_{ab}\w\omega^0+\kappa_{ab}\w\omega^n,
\end{gather}
 is obtained where $\lambda_{ab}=-\lambda_{ba}$, $\kappa_{ab}=-\kappa_{ba}$ and $\lambda_{ab},\kappa_{ab}\equiv 0$ modulo~$\Itot$.

Differentiating \eqref{1st-causal-II}, and using \eqref{eq:causal-Gamma_a-Phi_n}, and \eqref{eq:causal-Phi_ab-1}, it is straightforward to find
\begin{gather} \label{eq:causal-d2-omega-a-curvatures-2}
\kappa_{ab}=0,\qquad \nu=0,\qquad \tau_{ab}=E_{0a c b} \theta^c +K_{ab} \pi_n -F_{ab}{}^c \pi_c,\qquad\lambda_{ab}=0,
 \end{gather}
for some $E_{0abc}$.

Equations \eqref{eq:causal-d2-omega-a-curvatures-2}, \eqref{eq:causal-Gamma_a-Phi_n}, \eqref{eq:causal-Phi_ab-1} and \eqref{eq:causal-Phi_0-1} can be combined to give
\begin{gather}
 \Phi_0=0,\nonumber\\
 \Phi_{ab}= \half W_{abij}\omega^i\w\omega^j+E_{abci}\theta^c\w\omega^i+\half C_{abcd}\theta^c\w\theta^d,\nonumber\\
 \Phi_n= W_{nnij}\omega^i\w\omega^j-\half \ell_a\theta^a\w\omega^0+\half k_{ab}\theta^b\w\omega^a,\nonumber\\
 \Gamma_a = \half W_{0aij}\omega^i\w\omega^j+E_{0ab i}\theta^b\w\omega^i +\half C_{abc}\theta^b\w\theta^c\nonumber\\
\hphantom{\Gamma_a =}{} -F_{abc} \gamma^b\w\theta^c -L_{a} \pi_n\w\omega^0 +K_{ab} \pi_n\w\omega^b-F_{ab}{}^{c} \pi_c\w\omega^b. \label{eq:causal-StrEqns-incomplete-1}
\end{gather}
The quantities $W_{ijkl}$, $E_{ijkl}$, $C_{abcd}$ and $C_{abc}$ are not expressed until the replacement~\eqref{eq:pi_0-pi_n-changed-1} is made which makes the computation slightly easier.

\subsubsection{Reduction of the prolonged group}\label{sec:reduct-prol-group}

The infinitesimal action of $G^{(1)}_1$ on $W_{nn0n}$ and $W_{nn0a}$ is given by
\begin{gather*}\exd W_{nn0n}+\half \sigma\equiv 0,\qquad \exd W_{nn0a}+\quar\tau_a\equiv 0\end{gather*}
modulo the ideal generated by the 1-forms $\{\omega^i,\theta_a,\phi_0,\phi_n,\phi^a{}_b,\gamma_a,\pi_n,\pi_a\}_{i=0,a=1}^{n,\ n-1}$.

Consequently, the normalizations $W_{nn0n}=0,W_{nn0a}=0$, reduce $G^{(1)}_1$ to the group $G^{(1)}_2$ containing the one parameter $\bfd$ in \eqref{eq:parametric-prolonged-group}, which corresponds to the 1-form $\pi_0$. As a result, the expression for $\Phi_n$ changes to
\begin{gather} \label{eq:phi-n-after-red-prol-gr}
 \Phi_n= W_{nnab} \omega^a\w\omega^b+2W_{nnan} \omega^a\w\omega^n-\half \ell_a \theta^a\w\omega^0+\half k_{ab} \theta^b\w\omega^a.
\end{gather}
One sees that the action of the parameter $\bfd$ on the torsion coefficients in the 2-forms $\Theta_0,\dots,\Phi_0$, $\dots,\Gamma_{n-1}$ a~priori cannot be used to reduce the prolonged structure group to the trivial group, unless certain non-vanishing conditions are assumed.

The only part left, is to see whether any reduction is possible via the torsion coefficients in $\Pi_n,\dots,\Pi_0$.

\subsubsection[An $\{e\}$-structure]{An $\boldsymbol{\{e\}}$-structure}\label{sec:an-e-structure}

In order to make computing $\Pi_0,\dots,\Pi_n$ easier the replacements
\begin{gather} \label{eq:pi_0-pi_n-changed-1}
 \pi_0\mapsto \pi_0+\half \ell_a \theta^a,\qquad\pi_a\mapsto \pi_a-\half k_{ab} \theta^b+ W_{nnab} \omega^b,\qquad \pi_n\mapsto \pi_n-2W_{nnan} \omega^a.
\end{gather}
 are imposed which result in
\begin{gather*}\Phi_0+\Phi_n=0.\end{gather*}
Similarly to Section~\ref{sec:second-order-struct}, by differentiating the 2-forms $\Gamma_a$, $\Phi^a{}_b$ in~\eqref{eq:causal-StrEqns-incomplete-1}, $\Phi_n$ in \eqref{eq:phi-n-after-red-prol-gr} and $\Phi_0=-\Phi_n$, and carrying out tedious but straightforward calculations, it follows that{\samepage
\begin{gather}
 \Pi_n=\half W_{nij}\omega^i\w\omega^j+E_{nai}\theta^a\w\omega^i+\sigma_{ni}\w\omega^i,\nonumber\\
 \Pi_a=\half W_{aij}\omega^i\w\omega^j+E_{abi}\theta^b\w\omega^i- F_{ab}{}^c\pi_c\w\theta^b +\sigma_{ai}\w\omega^i,\label{eq:causal-dpi_i-1}
\end{gather}
where $\sigma_{ab}=\sigma_{ba}$, $\sigma_{na}=\sigma_{an}$ are congruent to zero modulo $\Itot$.}

The expressions for the torsion coefficients are listed in~\eqref{eq:2nd-order-ess-torsions-01}. It turns out that without assuming certain non-vanishing conditions on the torsion coefficients, the infinitesimal action of the parameter $\mathbf{d}$ on the torsion terms does not reduce the prolonged group.

Since, the prolongation of the reduced group $G^{(1)}_2$ is trivial, it follows that the set of 1-forms \begin{gather*}\big\{\omega^i,\theta^a,\phi_0,\phi^a{}_b,\phi_n,\gamma_a,\pi_i\big\}_{i=0,\,a=1}^{n,\ n-1}\end{gather*} constitutes an $\{e\}$-structure for $\cC$.

\subsubsection{The third order structure equations}\label{sec:third-order-struct}
Similarly to \eqref{eq:curvatures1}, defining the 2-form
\begin{gather*} 
 \Pi_0=\exd\pi_0-2\phi_0\w\pi_0-\gamma^a\w\pi_a,
\end{gather*}
helps making the expressions more compact.

It is a matter of computation, similar to Section~\ref{sec:an-e-structure}, to find
\begin{gather*} 
 \Pi_0= \half W_{0ij}\omega^i\w\omega^j+E_{0ai} \theta^a\w\omega^i- K_{ab} \pi^a\w\theta^b -L_a\pi_n\w\theta^a+\sigma_{0i}\w\omega^i,
\end{gather*}
where 1-forms $\sigma_{ij}$ defined in \eqref{eq:causal-dpi_i-1} satisfy $\sigma_{ij}=\sigma_{ji} $. Note that the symmetry of $\sigma_{ij}$ is a result of imposing $\Phi_0+\Phi_n=0$. Expressions for~$\sigma_{ij}$ are listed in~\eqref{eq:sigma-ij}.

\subsubsection{Full structure equations}\label{sec:full-struct-equat}

The full structure equations after replacement \eqref{eq:pi_0-pi_n-changed-1} using the 2-forms \eqref{eq:curvatures1} are expressed as follows
\begin{gather}
 \Omega^0 =0,\qquad
 \Omega^a =-F^a{}_{bc} \theta^b\w\omega^c-K^a{}_{b} \theta^b\w\omega^0,\qquad
 \Omega^n =-L_a \theta^a\w\omega^0,\nonumber\\
 \Theta_a =-f_{ ab c} \theta^b\w\omega^c-\half k_{ab}\theta^b\w\omega^0-W_{nnab}\omega^0\w\omega^b +\half W_{anbc} \omega^b\w\omega^c+ W_{bncn} \omega^c\w\omega^n,\nonumber\\
 \Phi_0 =-\half W_{nnab}\omega^a\w\omega^b+\half \ell_a\theta^a\w\omega^0-\quar k_{ab}\theta^b\w\omega^a,\nonumber\\
 \Phi_{ab} = \half W_{abij} \omega^i\w\omega^j+E_{abci} \theta^c\w\omega^i+\half C_{abcd} \theta^c\w\theta^d,\nonumber\\
 \Phi_n = \half W_{nnab} \omega^a\w\omega^b-\half \ell_a \theta^a\w\omega^0+\quar k_{ab} \theta^b\w\omega^a,\nonumber\\
 \Gamma_a = \half W_{0aij} \omega^i\w\omega^j+E_{0ab i} \theta^b\w\omega^i +\half C_{abc} \theta^b\w\theta^c\nonumber\\
\hphantom{\Gamma_a =}{} -F_{abc} \gamma^b\w\theta^c -L_{a} \pi_n\w\omega^0 +K_{ab} \pi_n\w\omega^b-F_{ab}{}^{c} \pi_c\w\omega^b,\nonumber\\
 \Pi_n =\half W_{nij} \omega^i\w\omega^j+E_{nai} \theta^a\w\omega^i+\sigma_{ni}\w\omega^i,\nonumber\\
 \Pi_a =\half W_{aij} \omega^i\w\omega^j+E_{abi} \theta^b\w\omega^i- F_{ab}{}^c \pi_c\w\theta^b +\sigma_{ai}\w\omega^i,\nonumber\\
 \Pi_0 = \half W_{0ij} \omega^i\w\omega^j+E_{0ai} \theta^a\w\omega^i- K_{ab} \pi^a\w\theta^b -L_a\pi_n\w\theta^a+\sigma_{0i}\w\omega^i,\label{eq:full-stru-equns-causal}
\end{gather}
The expressions of the quantities on the right hand side are given in~\eqref{eq:2nd-order-ess-torsions-01}, \eqref{eq:sigma-ij} and~\eqref{eq:E_iaj}. When differentiation~\eqref{1st-causal-II}, the vanishing of 3-forms $\theta^b\w\theta^c\w\omega^d$, $\theta^b\w\theta^c\w\omega^0$ combined with the expression of $C_{abcd}$, and $C_{abc}$ in~\eqref{eq:2nd-order-ess-torsions-01}, yield the relations
\begin{subequations} \label{eq:Bianchi-F_abc-K_ab-L_a}
 \begin{gather}
 F_{abc,\und}-F_{abd,\unc}=\half(K_{ac}\ve_{bd}+K_{bc}\ve_{ad}-K_{ad}\ve_{bc}-K_{bd}\ve_{ac}), \label{eq:raw-K_ab-F_abc,d}\\
 K_{ab,\unc}-K_{ac,\unb}=2(\ve_{ab}L_c-\ve_{ac}L_b)+F_{abd}K^d{}_c-F_{acd}K^d{}_b, \label{eq:raw-K_ab,c,L_a}
 \end{gather}
\end{subequations}
Moreover, the vanishing of the 3-forms $\theta_a\w\theta_b\w\omega^0$, in the exterior derivative of \eqref{1st-causal-III} implies $L_{a;\unb}=L_{b;\una}$. This equation together with the contraction of equations \eqref{eq:Bianchi-F_abc-K_ab-L_a} give the following Bianchi identities among the fiber invariants of~$\cC$
\begin{gather} \label{eq:causal-bianchi-K}
 K_{ab}=\textstyle{\frac{2}{n-1}} F^c{}_{ab;\bc}, \qquad 
 L_a=\textstyle{\frac{1}{2(n-2)}}\big(F_{abc}K^{bc}-K_{a}{}^b{}_{;\unb}\big), \qquad 
 L_{a;\unb}=L_{b;\una}. 
 \end{gather}

\subsection{Quantities in (\ref{eq:full-stru-equns-causal})}\label{sec:tors-coeff-2nd-strEq-1}

The expression for the torsion coefficients are as follows
\begin{gather}
 E_{abc0}=\half \big( K_{bc;a}-K_{ac;b} -K_{ad}f^d{}_{bc}+ K_{bd}f^d{}_{ac}\big),\qquad E_{abcn}=0,\nonumber\\
 E_{abcd}= F_{bcd;a}-F_{acd;b} +F_{bde}f^e{}_{ca}-F_{ade}f^e{}_{cb},\nonumber \\
 C_{abcd}=F_{bce}F^e{}_{ad}-F_{ace}F^e{}_{bd}+\half\big( K_{bc}\ve_{ad} +K_{ad}\ve_{bc}- K_{ac}\ve_{bd} -K_{bd}\ve_{ac}\big),\nonumber\\
E_{0ab0}=-L_{b;a}-L_{d}f^d{}_{ab},\qquad E_{0abn}=\half k_{ab},\qquad C_{abc}=L_c\ve_{ab}-L_b\ve_{ac},\nonumber\\
E_{0abc}= \half \big(K_{ab;c} +K_{bc;a} +K_{ad}f^d{}_{bc}+K_{cd}f^d{}_{ab} \big)-F_{abc;0}-\half k_{bd}F^d{}_{ac}. \label{eq:2nd-order-ess-torsions-01}
\end{gather}

The quantities $W_{ijkl}=-W_{ijlk}$ satisfy the following identities
\begin{gather}
 W_{an0b}=-W_{nnab},\qquad W_{an0n}=0,\qquad W_{anbn}=W_{bnan},\qquad W{}^a{}_{nan}=0,\qquad W_{[a|n|bc]}=0,\nonumber\\
 W_{abcn}=W_{cnab}+F_{acd}W^d{}_{nbn}-F_{bcd}W^d{}_{nan},\qquad W_{a[bcd]}+F_{ae[b|}W_{en|cd]}=0,\nonumber\\
 W_{abcd}-W_{cdab}= F_{aec}W^e{}_{nbd} - F_{aed}W^e{}_{bnc} -F_{bec}W^e{}_{nad} +F_{bed}W^e{}_{nac},\nonumber\\
W_{ab0c}=-W_{0cab} +F_{acd}W^d{}_{n0b}-F_{bcd}W^d{}_{n0a}+\half K_{cd}W^d{}_{nba} +\half \big( K_{bd}W^d{}_{nca}-K_{ad}W^d{}_{ncb}\big),\nonumber\\
W_{0abn}=-W_{nnab}-K_{(a|d}W^d{}_{n|b)n},\qquad W_{ab0n}=-2W_{nnab}+K_{[a|d}W^d{}_{n|b]n},\nonumber\\
W_{0[abc]}=0,q\quad W_{0a0n}=L_cW^c{}_{nan},\qquad W_{0[a|0|b]}=\half L_cW^c{}_{nab},\qquad W_{0}{}^a{}_{0a}=0.\label{eq:2nd-order-W-ijkl}
\end{gather}
According to the symmetries above, the independent components of $W_{ijkl}$ are $W_{abcd}$, $ W_{anbc}$, $ W_{0abc}, $ $W_{nabn}, $ $W_{0a0b}$ and $W_{nnab}$. It is shown in Section \ref{sec:bianchi-identities} that $W_{ijkl}$ is generated by $W_{anbn}$.

The 1-forms $\sigma_{ij}=\sigma_{ji}$ are
\begin{gather}
 \sigma_{nn}=0,\qquad
 \sigma_{an}=W_{anbn}\gamma^b,\qquad
 \sigma_{a0}=3W_{nnab}\gamma^b-\ell_a\pi_n-\half k_{ab}\pi^b,\qquad
 \sigma_{0n}=0,\nonumber\\
 \sigma_{ab}=-W_{(a|n|b)c}\gamma^c+\half k_{ab}\pi_n-f_{abc}\pi^c,\label{eq:sigma-ij}\\
 \sigma_{00}=\ts\frac{1}{n-2}\big(2W_{0bba}+K^{cd}W_{dnca}-L^cW_{cnan}\big) \gamma^a -\ts\frac{2}{n-1} L^a{}_{;a}\pi_n-\ts\frac{1}{n-1}\big( K^{ac}{}_{;a}+K_{ab}f^{abc} \big)\pi_c.\nonumber
\end{gather}
The quantities $E_{iaj}=E_{jai}$ are given by
\begin{gather}
 E_{nan}=\ts\frac{2}{n-1}\ve^{bc}W_{0bca},\nonumber\\
 E_{nab}=3W_{nnab}+\half k_{ab;n}+K_{(a|d}W^d{}_{n|b)n},\nonumber\\
 E_{0ab}=-W_{0a0b}+2W_{nnbd}K^d{}_{a}-\ell_{a;b}-\ell_df^d{}_{ab}-\half k_{ab;0}-\quar k_{a}{}^ck_{bc},\nonumber\\
 E_{abc}=W_{(a|b0|a)}-W_{(a|n|c)d}K_{db}-W_{ancn}L_b+\half K_{(a|b;n|c)}-F_{abc;n0}\nonumber\\
 \hphantom{E_{abc}=}{} -\half F_{acd;n}K_{db;n} -\half K_{(a|d;n}F_{db|c);n}+\big(\half L_cW_{cnbn}+\half \ell_{b;n}\big)\ve_{ac},\nonumber\\
 E_{0an}=-\half L_cW_{cnan}-\half \ell_{a;n},\nonumber\\
 E_{0b0}=\ts\frac{1}{n-1}\big( {-}2 F_{abc}W_{0d0c}+W_{0aad}K_{db}-K^{ad}W_{nand}L_b+L_aE_{nba}-L_{a;d}f^{ad}{}_b \nonumber\\
 \hphantom{E_{0b0}=}{} -\big( f^a{}_{bc}L^c-\ve^{ca}L_{b;c}\big)_{;a}L^cf_{acd}f^{ad}{}_b-\half \big( K^a{}_{d;a}+K^{ae}f_{ade}\big) k^d{}_{b}\nonumber\\
\hphantom{E_{0b0}=}{} - \big( K^a{}_{b;a}+K^{ad}f_{adb}\big)_{;0}+L^dW_{nndb}\big)).\label{eq:E_iaj}
\end{gather}
The quantities $W_{ijk}$ satisfy
\begin{gather*}W_{ijk}+W_{kij}+W_{jki}=0,\qquad W_{ijk}=-W_{ikj}\end{gather*} and are obtained from $W_{ijkl}$, $F_{abc}$ as follows
\begin{gather*} 
 W_{nan} =\ts-\frac{1}{n-2}(W_{anba;n}-W_{anbn;a}-F_{adb;n}W_{dnan}),\\
 W_{nab} =\ts 2W_{nnab}+ k_{[a|c}W_{cn|b]n},\\
 W_{n0n} =\ts\frac{1}{(n-2)(n-1)}(W_{cdcd;n}-E_{cdec}W_{dnen}-E_{cded}W_{encn}-W_{adan;d}+W_{addn;a}) \\
\hphantom{W_{n0n} =}{} +\ts\frac{1}{2(n-1)}k_{cd}W_{ancn},\\
 W_{abn} =\ve_{ab}W_{n0n}+W_{anbn;0}-W_{nnab;n}-\half k_{ac}W_{cnbn},\\
 W_{a0n} =\ts \frac{1}{n-2}\big({-}W_{nnba;b}-W_{bnba;0} + \half k_{ed} W_{dnea} +W_{0dda;n} -W_{0ddn;a}+W_{0dan;d} \\
\hphantom{W_{a0n} =}{}-E_{eda}W_{dnen}+E_{ede}W_{dnan} +E_{0edn}W_{dnea} -K_{ad}W_{ndn}+F_{ade}W_{den} \big),\\
 W_{n0a} =\half W_{a0n}-\half \ell_bW_{bnan},\\
 W_{abc} =\ts \ve_{ab}W_{n0c}-\ve_{ac}W_{n0b}-W_{nnab;c}+W_{nnac;b}+W_{anbc;0}\\
\hphantom{W_{abc} =}{} -\half K_{ad}W_{dnbc}+ f_{acd}W_{nndb}-f_{abd}W_{nndc},\\
 W_{0ab} =\ts -2W_{nnab;0}+\ell_cW_{cnab}-k_{[a|d}W_{nnc|b},\\
 W_{00n} =\ts\frac{1}{(n-1)(n-2)}\big( W_{ab0[b;a]}-\half W_{abab;0}+ \half E_{abe0} W_{enba} + \half E_{abe0}W_{nnea} \big)\\
\hphantom{W_{00n} =}{} +\ts\frac{1}{n-1}\big( E_{0adn}W_{nnda}-E_{0ad0}W_{dnad}+L_aW_{nan}\big),\\
 W_{a0b} =\ve_{ab}W_{00n}+W_{0a0b;n}+W_{0abn;0}-E_{0adn}W_{nndb} \\
\hphantom{W_{a0b} =}{} +E_{0ad0}W_{dnbn}-L_aW_{nbn}-K_{ab}W_{n0n} +F_{abc}W_{c0n},\\
 W_{00b} =-W_{0aab;0}+W_{0a0b;a}+E_{0ab0}W_{dba}+E_{0ada}W_{d0b}\\
\hphantom{W_{00b} =}{} -E_{0adb}W_{d0a} -L_a W_{nba} -K_{ab}W_{n0a}+F_{abd}W_{d0a}.
\end{gather*}
In the above expressions the indices are lowered using $\ve^{ab}$ and the summation is taken over repeated indices.

\subsection[Bianchi identities for $W_{ijkl}$]{Bianchi identities for $\boldsymbol{W_{ijkl}}$}\label{sec:bianchi-identities}
The following relations hold among the torsion elements $W_{ijkl}$ in \eqref{eq:full-stru-equns-causal}.
\begin{gather}
 W_{anbc}=\ts{\frac{2}{3}}\big(W_{an[b|n;|\unc]}+F_{a[c}{}^dW_{b]ndn}\big) +\ts\frac{4}{3(n+2)}\big(\ve_{a[c|}W^{d}{}_{n|b]n;\und} +\ve_{a[b}F^{de}{}_{c]}W_{dnen}\big),\nonumber\\
 W_{nnab} =\ts\frac{1}{2(n+3)}\big({-}W^d{}_{nab;\und}+F^{cd}{}_{a}W_{cndb}-F^{cd}{}_{b}W_{cnda}\big),\nonumber\\
 W_{abcd} =W_{[a|ncd;|\unb]}-f^{e}{}_{ac}f_{bde}+f^e{}_{ad}f_{bce}+W_{ane[c}F_{b]d}{}^e -W_{dne[c}F_{b]a}{}^e\nonumber\\
\hphantom{W_{abcd} =}{} -\ve_{a[c|}\big(W_{nn|d]b}-\half k_{|d]b;n}+W_{d]nen}K_{b}{}^e\big)\nonumber\\
\hphantom{W_{abcd} =}{} +\ve_{b[c}\big(W_{nnd]a}-\half k_{|d]a;n}+W_{d]nen}K_{a}{}^e\big),\nonumber\\
 W_{0abc}=2W_{nnbc;\una}-4W_{nnd[c}F_{b]a}{}^d+ k_{a[b;c]}+ f^d{}_{a[c}k_{b]d},\nonumber\\
 W_{0a0b}=\ts-\frac{1}{n-2}W_{0abc;\und}\ve^{cd} -\frac{1}{n-2}W_{0ecb}F^{ec}{}_{a}-\ve^{cd}E_{0cd}\ve_{ab}+E_{0ab} +\ve^{cd}(E_{0acd;b}-E_{0acb;d})\nonumber\\
\hphantom{W_{0a0b}=}{} +E_{0aec}f^{ec}{}_b-F_{abe}E^{ec}{}_{c}+F_{aec}E^{ec}{}_{b}
 +L^cW_{nbca}+K_{ab}E_{nc}{}^c-K_{a}{}^dE_{ndb}.\label{eq:2nd-order-biachies}
\end{gather}
It follows from the relations above that the infinitesimal action of the structure group on $W_{ijkl}$ is obtained from the appropriate chains of coframe derivatives of $ W_{anbn}$.
\end{appendix}

\subsection*{Acknowledgments}
This work was carried out under the guidance of Professor N.~Kamran, whom I would like to thank greatly for his clear explanation of Cartan's method of equivalence and introducing me to the work of Holland and Sparling~\cite{HS2010causal}. His encouragement, help and support during my PhD made this work possible. I am also very much indebted to Dennis The for the many helpful conversations that we had and for patiently explaining some aspects of parabolic geometry. I~thank the anonymous referees for their corrections and suggestions. The funding for this work was provided by McGill University.

\pdfbookmark[1]{References}{ref}
\LastPageEnding


\begin{thebibliography}{99}
\footnotesize\itemsep=0pt

\bibitem{AS-Singularity}
Aazami A.B., Javaloyes M.A., Penrose's singularity theorem in a {F}insler
 spacetime, \href{https://doi.org/10.1088/0264-9381/33/2/025003}{\textit{Classical Quantum Gravity}} \textbf{33} (2016), 025003,
 22~pages, \href{https://arxiv.org/abs/1410.7595}{arXiv:1410.7595}.

\bibitem{AZ-JacobiCurves}
Agrachev A.A., Zelenko I., Geometry of {J}acobi curves.~{I}, \href{https://doi.org/10.1023/A:1013904801414}{\textit{J.~Dynam.
 Control Systems}} \textbf{8} (2002), 93--140.

\bibitem{AG-Projective}
Akivis M.A., Goldberg V.V., Projective differential geometry of submanifolds,
 \textit{North-Holland Mathematical Library}, Vol.~49, North-Holland
 Publishing Co., Amsterdam, 1993.

\bibitem{AG-Conformal}
Akivis M.A., Goldberg V.V., Conformal differential geometry and its
 generalizations, \href{https://doi.org/10.1002/9781118032633}{\textit{Pure and Applied Mathematics (New York)}}, John Wiley \& Sons, Inc., New York, 1996.

\bibitem{Alexandrov-chronogeometry}
Alexandrov A.D., A contribution to chronogeometry, \href{https://doi.org/10.4153/CJM-1967-102-6}{\textit{Canad.~J. Math.}}
 \textbf{19} (1967), 1119--1128.

\bibitem{ANNMonge}
Anderson I., Nie Z., Nurowski P., Non-rigid parabolic geometries of {M}onge
 type, \href{https://doi.org/10.1016/j.aim.2015.01.021}{\textit{Adv. Math.}} \textbf{277} (2015), 24--55, \href{https://arxiv.org/abs/1401.2174}{arXiv:1401.2174}.

\bibitem{BCS-Finsler}
Bao D., Chern S.-S., Shen Z., An introduction to {R}iemann--{F}insler geometry,
 \href{https://doi.org/10.1007/978-1-4612-1268-3}{\textit{Graduate Texts in Mathema\-tics}}, Vol.~200, Springer-Verlag, New York,
 2000.

\bibitem{BRS-Zermelo}
Bao D., Robles C., Shen Z., Zermelo navigation on {R}iemannian manifolds,
 \href{https://doi.org/10.4310/jdg/1098137838}{\textit{J.~Differential Geom.}} \textbf{66} (2004), 377--435,
 \href{https://arxiv.org/abs/math.DG/0311233}{math.DG/0311233}.

\bibitem{BEE-Global}
Beem J.K., Ehrlich P.E., Easley K.L., Global {L}orentzian geometry,
 \textit{Monographs and Textbooks in Pure and Applied Mathematics}, Vol.~202,
 2nd~ed., Marcel Dekker, Inc., New York, 1996.

\bibitem{BelgunNullgeod}
Belgun F.A., Null-geodesics in complex conformal manifolds and the {L}e{B}run
 correspondence, \href{https://doi.org/10.1515/crll.2001.052}{\textit{J.~Reine Angew. Math.}} \textbf{536} (2001), 43--63,
 \href{https://arxiv.org/abs/math.DG/0002225}{math.DG/0002225}.

\bibitem{BryantRemarksFinsler}
Bryant R.L., Some remarks on {F}insler manifolds with constant flag curvature,
 \textit{Houston~J. Math.} \textbf{28} (2002), 221--262,
 \href{https://arxiv.org/abs/math.DG/0107228}{math.DG/0107228}.

\bibitem{CapCorr}
\v{C}ap A., Correspondence spaces and twistor spaces for parabolic geometries,
 \href{https://doi.org/10.1515/crll.2005.2005.582.143}{\textit{J.~Reine Angew. Math.}} \textbf{582} (2005), 143--172,
 \href{https://arxiv.org/abs/math.DG/0102097}{math.DG/0102097}.

\bibitem{CS-Cartan}
\v{C}ap A., Schichl H., Parabolic geometries and canonical {C}artan
 connections, \href{https://doi.org/10.14492/hokmj/1350912986}{\textit{Hokkaido Math.~J.}} \textbf{29} (2000), 453--505.

\bibitem{CS-Parabolic}
\v{C}ap A., Slov\'ak J., Parabolic geometries. {I}.~Background and general
 theory, \href{https://doi.org/10.1090/surv/154}{\textit{Mathematical Surveys and Monographs}}, Vol.~154, Amer. Math.
 Soc., Providence, RI, 2009.

\bibitem{CartanProbleme30}
Cartan \'E., Sur un probl{\`e}me d’{\'e}quivalence et la th{\'e}orie des
 espaces m{\'e}triques g{\'e}n{\'e}ralis{\'e}s, \textit{Mathematica}
 \textbf{4} (1930), 114--136.

\bibitem{CartanFinsler}
Cartan \'E., Les espaces de {F}insler, Hermann, Paris, 1934.

\bibitem{ChernODE}
Chern S.-S., The geometry of the differential equation {$y'''=F(x,y,y',y'')$},
 \textit{Sci. Rep. Nat. Tsing Hua Univ. Ser.~A} \textbf{4} (1940), 97--111.

\bibitem{ChernFinsler}
Chern S.-S., Local equivalence and {E}uclidean connections in {F}insler spaces,
 \textit{Sci. Rep. Nat. Tsing Hua Univ. Ser.~A} \textbf{5} (1948), 95--121.

\bibitem{DS-Segal}
Daigneault A., Sangalli A., Einstein's static universe: an idea whose time has
 come back?, \textit{Notices Amer. Math. Soc.} \textbf{48} (2001), 9--16.

\bibitem{DZquasi}
Doubrov B., Zelenko I., Prolongation of quasi-principal frame bundles and
 geometry of flag structures on manifolds, \href{https://arxiv.org/abs/1210.7334}{arXiv:1210.7334}.

\bibitem{EPS-GeometryOfLight}
Ehlers J., Pirani F.A.E., Schild A., The geometry of free fall and light
 propagation, in General Relativity (papers in honour of {J}.{L}.~{S}ynge),
 Clarendon Press, Oxford, 1972, 63--84.

\bibitem{FoxPath}
Fox D.J.F., Contact path geometries, \href{https://arxiv.org/abs/math.DG/0508343}{math.DG/0508343}.

\bibitem{FoxContactProj}
Fox D.J.F., Contact projective structures, \href{https://doi.org/10.1512/iumj.2005.54.2603}{\textit{Indiana Univ. Math.~J.}}
 \textbf{54} (2005), 1547--1598, \href{https://arxiv.org/abs/math.DG/0402332}{math.DG/0402332}.

\bibitem{GS-causal}
Garc\'{\i}a-Parrado A., Senovilla J.M.M., Causal structures and causal
 boundaries, \href{https://doi.org/10.1088/0264-9381/22/9/R01}{\textit{Classical Quantum Gravity}} \textbf{22} (2005), R1--R84,
 \href{https://arxiv.org/abs/gr-qc/0501069}{gr-qc/0501069}.

\bibitem{Gardner-Book}
Gardner R.B., The method of equivalence and its applications, \href{https://doi.org/10.1137/1.9781611970135}{\textit{CBMS-NSF
 Regional Conference Series in Applied Mathematics}}, Vol.~58, Society for
 Industrial and Applied Mathematics (SIAM), Philadelphia, PA, 1989.

\bibitem{GNThirdODEs}
Godlinski M., Nurowski P., Geometry of third-order {ODE}s, \href{https://arxiv.org/abs/902.4129}{arXiv:0902.4129}.

\bibitem{Griffiths-EDS-Book}
Griffiths P.A., Exterior differential systems and the calculus of variations,
 \href{https://doi.org/10.1007/978-1-4615-8166-6}{\textit{Progress in Mathematics}}, Vol.~25, Birkh\"auser, Boston, Mass., 1983.

\bibitem{GrossmanPath}
Grossman D.A., Torsion-free path geometries and integrable second order {ODE}
 systems, \href{https://doi.org/10.1007/PL00001394}{\textit{Selecta Math. (N.S.)}} \textbf{6} (2000), 399--442.

\bibitem{GuilleminCosmology}
Guillemin V., Cosmology in {$(2 + 1)$}-dimensions, cyclic models, and
 deformations of {$M_{2,1}$}, \href{https://doi.org/10.1515/9781400882410}{\textit{Annals of Mathematics Studies}}, Vol.~121, Princeton University Press, Princeton, NJ, 1989.

\bibitem{Guts-Causal}
Guts A.K., Axiomatic causal theory of space-time, \textit{Gravitation
 Cosmology} \textbf{1} (1195), 211--215.

\bibitem{HE-SpaceTime}
Hawking S.W., Ellis G.F.R., The large scale structure of space-time,
 \textit{Cambridge Monographs on Mathema\-ti\-cal Physics}, Vol.~1, Cambridge
 University Press, London~-- New York, 1973.

\bibitem{HO-causal}
Hilgert J., \'Olafsson G., Causal symmetric spaces: geometry and harmonic
 analysis, \textit{Perspectives in Mathematics}, Vol.~18, Academic Press,
 Inc., San Diego, CA, 1997.

\bibitem{HS2010causal}
Holland J., Sparling G., Causal geometries and third-order ordinary
 differential equations, \href{https://arxiv.org/abs/1001.0202}{arXiv:1001.0202}.

\bibitem{HS2011Causal}
Holland J., Sparling G., Causal geometries, null geodesics, and gravity,
 \href{https://arxiv.org/abs/1106.5254}{arXiv:1106.5254}.

\bibitem{HsuCalculus92}
Hsu L., Calculus of variations via the {G}riffiths formalism,
 \href{https://doi.org/10.4310/jdg/1214453181}{\textit{J.~Differential Geom.}} \textbf{36} (1992), 551--589.

\bibitem{HwangSurveyVMRT}
Hwang J.-M., Geometry of varieties of minimal rational tangents, in Current
 Developments in Algebraic Geometry, \textit{Math. Sci. Res. Inst. Publ.},
 Vol.~59, Cambridge University Press, Cambridge, 2012, 197--226.

\bibitem{HwangVMRTCodim1}
Hwang J.-M., Varieties of minimal rational tangents of codimension~1,
 \textit{Ann. Sci. \'Ec. Norm. Sup\'er.~(4)} \textbf{46} (2013), 629--649.

\bibitem{IL-Book}
Ivey T.A., Landsberg J.M., Cartan for beginners: differential geometry via
 moving frames and exterior differential systems, \textit{Graduate Studies in
 Mathematics}, Vol.~61, Amer. Math. Soc., Providence, RI, 2003.

\bibitem{KP-causal}
Kronheimer E.H., Penrose R., On the structure of causal spaces, \href{https://doi.org/10.1017/S030500410004144X}{\textit{Proc.
 Cambridge Philos. Soc.}} \textbf{63} (1967), 481--501.

\bibitem{KT-gap}
Kruglikov B., The D., The gap phenomenon in parabolic geometries,
 \href{https://doi.org/10.1515/crelle-2014-0072}{\textit{J.~Reine Angew. Math.}} \textbf{723} (2017), 153--215,
 \href{https://arxiv.org/abs/1303.1307}{arXiv:1303.1307}.

\bibitem{LebrunNullGeod}
LeBrun C.R., Spaces of complex null geodesics in complex-{R}iemannian geometry,
 \href{https://doi.org/10.2307/1999312}{\textit{Trans. Amer. Math. Soc.}} \textbf{278} (1983), 209--231.

\bibitem{Lebrun-conformal-gravity}
LeBrun C.R., Twistors, ambitwistors, and conformal gravity, in Twistors in
 Mathematics and Physics, \textit{London Math. Soc. Lecture Note Ser.}, Vol.~156, Cambridge University Press, Cambridge, 1990, 71--86.

\bibitem{Levichev-SegalChronometric}
Levichev A.V., Segal's chronometric theory as a completion of the special
 theory of relativity, \href{https://doi.org/10.1007/BF00562033}{\textit{Russian Phys.~J.}} \textbf{36} (1993), 780--783.

\bibitem{LMSympl87}
Libermann P., Marle C.M., Symplectic geometry and analytical mechanics,
 \href{https://doi.org/10.1007/978-94-009-3807-6}{\textit{Mathematics and its Applications}}, Vol.~35, D.~Reidel Publishing Co.,
 Dordrecht, 1987.

\bibitem{Low-Causal}
Low R.J., The space of null geodesics (and a new causal boundary), in
 Analytical and Numerical Approaches to Mathematical Relativity,
 \href{https://doi.org/10.1007/3-540-33484-X_2}{\textit{Lecture Notes in Phys.}}, Vol.~692, Springer, Berlin, 2006, 35--50.

\bibitem{Omid-Thesis}
Makhmali O., Differential geometric aspects of causal structures, Ph.D.~Thesis,
 McGill University, 2016.

\bibitem{Mettler}
Mettler T., Reduction of {$\beta$}-integrable 2-{S}egre structures,
 \href{https://doi.org/10.4310/CAG.2013.v21.n2.a3}{\textit{Comm. Anal. Geom.}} \textbf{21} (2013), 331--353, \href{https://arxiv.org/abs/1110.3279}{arXiv:1110.3279}.

\bibitem{Minguzzi-Raychaudhuri}
Minguzzi E., Raychaudhuri equation and singularity theorems in {F}insler
 spacetimes, \href{https://doi.org/10.1088/0264-9381/32/18/185008}{\textit{Classical Quantum Gravity}} \textbf{32} (2015), 185008,
 26~pages, \href{https://arxiv.org/abs/1502.02313}{arXiv:1502.02313}.

\bibitem{MiyaokaLieContact}
Miyaoka R., Lie contact structures and normal {C}artan connections,
 \href{https://doi.org/10.2996/kmj/1138039338}{\textit{Kodai Math.~J.}} \textbf{14} (1991), 13--41.

\bibitem{Morimoto-Filtered}
Morimoto T., Geometric structures on filtered manifolds, \href{https://doi.org/10.14492/hokmj/1381413178}{\textit{Hokkaido
 Math.~J.}} \textbf{22} (1993), 263--347.

\bibitem{Olver-Equiv-Book}
Olver P.J., Equivalence, invariants, and symmetry, \href{https://doi.org/10.1017/CBO9780511609565}{Cambridge University Press},
 Cambridge, 1995.

\bibitem{O'neill-Semi-Riemannian}
O'Neill B., Semi-{R}iemannian geometry: with applications to relativity,
 \textit{Pure and Applied Mathematics}, Vol.~103, Academic Press, Inc., New
 York, 1983.

\bibitem{Penrose-Sing}
Penrose R., Gravitational collapse and space-time singularities, \href{https://doi.org/10.1103/PhysRevLett.14.57}{\textit{Phys.
 Rev. Lett.}} \textbf{14} (1965), 57--59.

\bibitem{Penrose-Technique}
Penrose R., Techniques of differential topology in relativity, Society for
 Industrial and Applied Mathematics, Philadelphia, Pa., 1972.

\bibitem{RundNullExt}
Rund H., Congruences of null-extremals in the calculus of variations,
 \textit{J.~Nat. Acad. Math. India} \textbf{1} (1983), 165--178.

\bibitem{RundNullDist}
Rund H., Null-distributions in the calculus of variations of multiple
 integrals, \textit{Math. Colloq. Univ. Cape Town} \textbf{13} (1984), 57--81.

\bibitem{Sachs-Gravitational}
Sachs R., Gravitational waves in general relativity. {VI}.~{T}he outgoing
 radiation condition, \href{https://doi.org/10.1098/rspa.1961.0202}{\textit{Proc. Roy. Soc. Ser.~A}} \textbf{264} (1961), 309--338.

\bibitem{SasakiBook}
Sasaki T., Projective differential geometry and linear homogeneous differential
 equations, Department of Mathematics, Kobe University, 1999.

\bibitem{SYLieContactStr1}
Sato H., Yamaguchi K., Lie contact manifolds, in Geometry of Manifolds
 ({M}atsumoto, 1988), \textit{Perspect. Math.}, Vol.~8, Academic Press,
 Boston, MA, 1989, 191--238.

\bibitem{SY-ODE}
Sato H., Yoshikawa A.Y., Third order ordinary differential equations and
 {L}egendre connections, \href{https://doi.org/10.2969/jmsj/05040993}{\textit{J.~Math. Soc. Japan}} \textbf{50} (1998),
 993--1013.

\bibitem{Schapira-Causal}
Schapira P., Hyperbolic systems and propagation on causal manifolds,
 \href{https://doi.org/10.1007/s11005-013-0637-2}{\textit{Lett. Math. Phys.}} \textbf{103} (2013), 1149--1164,
 \href{https://arxiv.org/abs/1305.3535}{arXiv:1305.3535}.

\bibitem{Segal-Book}
Segal I.E., Mathematical cosmology and extragalactic astronomy, \textit{Pure
 and Applied Mathematics}, Vol.~68, Academic Press, New York~-- London, 1976.

\bibitem{Shen-Book}
Shen Z., Differential geometry of spray and {F}insler spaces, \href{https://doi.org/10.1007/978-94-015-9727-2}{Kluwer Academic
 Publishers}, Dordrecht, 2001.

\bibitem{Struik}
Struik D.J., Lectures on classical differential geometry, 2nd~ed., Dover
 Publications, Inc., New York, 1988.

\bibitem{Tanaka-79}
Tanaka N., On the equivalence problems associated with simple graded {L}ie
 algebras, \href{https://doi.org/10.14492/hokmj/1381758416}{\textit{Hokkaido Math.~J.}} \textbf{8} (1979), 23--84.

\bibitem{Taub-Review}
Taub A.H., Book {R}eview: {M}athematical cosmology and extragalactic astronomy,
 \href{https://doi.org/10.1090/S0002-9904-1977-14356-5}{\textit{Bull. Amer. Math. Soc.}} \textbf{83} (1977), 705--711.

\bibitem{Wormald-Critique}
Wormald L.I., A critique of {S}egal's chronometric theory, \href{https://doi.org/10.1007/BF00762198}{\textit{Gen.
 Relativity Gravitation}} \textbf{16} (1984), 393--401.

\bibitem{Yamaguchi-SimpleDiffSystems}
Yamaguchi K., Differential systems associated with simple graded {L}ie
 algebras, in Progress in Differential Geometry, \textit{Adv. Stud. Pure
 Math.}, Vol.~22, Math. Soc. Japan, Tokyo, 1993, 413--494.

\bibitem{Ye}
Ye Y.G., Extremal rays and null geodesics on a complex conformal manifold,
 \href{https://doi.org/10.1142/S0129167X94000073}{\textit{Internat.~J. Math.}} \textbf{5} (1994), 141--168,
 \href{https://arxiv.org/abs/alg-geom/9206004}{alg-geom/9206004}.

\bibitem{ZadnikLieContact}
Zadnik V., Lie contact structures and chains, \href{https://arxiv.org/abs/0901.4433}{arXiv:0901.4433}.

\bibitem{Zelenko-Tanaka}
Zelenko I., On {T}anaka's prolongation procedure for filtered structures of
 constant type, \href{https://doi.org/10.3842/SIGMA.2009.094}{\textit{SIGMA}} \textbf{5} (2009), 094, 21~pages,
 \href{https://arxiv.org/abs/0906.0560}{arXiv:0906.0560}.

\end{thebibliography}
\end{document}